\documentclass[11pt]{article}
\usepackage{amsthm,amsmath,amssymb,bbm}
\usepackage{color}
\def\bbe{\mathbb E}
\def\bbp{\mathbb P}
\def\bbr{\mathbb R}

\def\bbone{{\mathbbm 1}}

\theoremstyle{plain}
\newtheorem{thm}{Theorem}[section]

\newtheorem{cor}[thm]{Corollary}
\newtheorem{prop}[thm]{Proposition}
\newtheorem{lem}[thm]{Lemma}

\newtheorem{rem}[thm]{Remark}
\numberwithin{equation}{section}

\begin{document}
\title{On Some Operators Associated with Non-Degenerate Symmetric $\alpha$-Stable Probability Measures}
\author{Benjamin Arras\thanks{Universit\'e de Lille, Laboratoire Paul Painlev\'e, CNRS U.M.R. 8524, 
59655 Villeneuve d'Ascq, France; {\tt benjamin.arras@univ-lille.fr}.}\; and Christian Houdr\'e\thanks{Georgia Institute of Technology, School of Mathematics, Atlanta, GA 30332-0160, USA; \newline {\tt houdre@math.gatech.edu}; 
Research supported in part by the grant \# 524678 from the Simons Foundation.
}}

\maketitle

\begin{abstract}
\noindent
Boundedness properties of operators associated with non-degenerate symmetric $\alpha$-stable, 
$\alpha \in (1,2)$, probability measures on $\bbr^d$ are investigated on appropriate, Euclidean or otherwise, 
$L^p$-spaces, 
$p \in (1,+\infty)$.~Our approach is based on first obtaining Bismut-type formulae which lead to 
useful representations for various operators.~In the Euclidean setting, the 
method of transference and one-dimensional multiplier theory combined with fine properties of stable distributions provide dimension-free estimates for the fractional Laplacian. In the non-Euclidean setting, we obtain boundedness results for the non-singular cases as well as dimension-free estimates when the 
reference measure is the rotationally 
invariant $\alpha$-stable probability measure.
\end{abstract}

\section{Introduction}
\noindent
The problem of obtaining dimension-free estimates for the boundedness of Riesz transforms has been 
addressed in many contexts.  Starting with the classical Euclidean case, very sharp results have been obtained, by different methods, ranging from analytic ones to probabilistic ones (\cite{St83, DRF85,B86,BW95,IM96,DV06}). When the reference measure is the Gaussian one, the corresponding problem for the Riesz transform associated with the Ornstein-Uhlenbeck operator was first solved by P.~A.~Meyer \cite{M1} 
using the general Littlewood-Paley-Stein theory of symmetric diffusion semigroup, developed by E.~M.~Stein 
(see \cite{St70}),  in combination with probabilistic tools.~Nowadays, different proofs and refinements of this Gaussian inequality have been provided from  \cite{Gu86,P88,Ar98,G94,LC02} to \cite{DV06,W18}.  
For a more complete set of references on the subject (and related topics in Gaussian Harmonic Analysis),  
the reader is referred to the very recent book \cite{Ur19}.~Moreover, this type of inequalities has further been studied 
and found in different settings such as diffusion theory \cite{B85_1,B85_2}, 
group von Neumann algebras \cite{JMP18}, the Heisenberg group \cite{CMZ96}, infinite dimensional Gaussian analysis  \cite{Shi92,Ar98,CG01} and multidimensional orthogonal expansions on product spaces \cite{W18}. 

The goal of the present manuscript is to analyze the boundedness of various operators naturally associated with non-degenerate and symmetric $\alpha$-stable, $\alpha \in (1,2)$, probability measures on $\bbr^d$ which are built with the help of fractional versions of the gradient, of the Laplacian and of the Ornstein-Uhlenbeck operator.


Let us introduce the strategy of our approach by presenting the very classical case of the fractional Laplacian of order $\alpha \in (1,2)$ in the Euclidean setting.~This operator, with symbol proportional to $-\|\xi\|^{\alpha}$, where $\|\cdot\|$ stands for the Euclidean norm on $\bbr^d$, is the generator of the fractional heat semigroup $(P^\alpha_t)_{t \geq 0}$ associated with the rotationally invariant $\alpha$-stable probability measure on $\bbr^d$ (see Equation \eqref{eq:StheatSM} below with $\mu_\alpha^{\operatorname{rot}}$).~The family of operators of interest in this note is defined, in part, by means of potential operators associated with the fractional Laplacian with symbols given, for all $\xi \in \bbr^d$, all $\lambda_0 >0$ and all $r>0$, by
\begin{align*}
m^{\operatorname{rot}}_{\alpha,\lambda_0,r}(\xi) = \dfrac{1}{\left(\lambda_0 + \frac{\|\xi\|^{\alpha}}{2}\right)^{\frac{r}{2}}}.
\end{align*} 
Moreover, a last building block of our family of operators of interest is a fractional derivative which plays the role of the gradient. Its action on the Schwartz space is given, in the Fourier domain, by multiplication with the vector-valued function $\tau_\alpha^{\operatorname{rot}}$ defined, for all $\xi \in \bbr^d$, by
\begin{align*}
\tau^{\operatorname{rot}}_\alpha(\xi) = \frac{i \alpha}{2} \frac{\xi}{\|\xi\|} \|\xi\|^{\alpha-1}.
\end{align*}
Hence that family of operators of interest is defined (still in the Fourier domain) through multiplication with the symbol given, for all $\xi \in \bbr^d$, all $\lambda_0 >0$ and all $r >0$, by
\begin{align}\label{eq:multiplier_frac}
\tilde{m}^{\operatorname{rot}}_{\alpha,\lambda_0,r}(\xi) = \frac{i \alpha}{2} \frac{\xi}{\|\xi\|}  \dfrac{\|\xi\|^{\alpha-1}}{\left(\lambda_0 + \frac{\|\xi\|^{\alpha}}{2}\right)^{\frac{r}{2}}}. 
\end{align}
This easy-to-establish formula (see the proof of Lemma \ref{lem:Hor_Mihlin}) encodes several interesting facts. First of all, as soon as $r>0$ is large enough, the $L^p(\bbr^d,dx)$-boundedness (with $p\in (1, +\infty)$) of the operator with symbol $\tilde{m}^{\operatorname{rot}}_{\alpha,\lambda_0,r}$ can be established from known results and standard Fourier multiplier theory, since there is still some form of regularity (away from the origin) in this rotationally invariant setting.~Moreover, by choosing a specific value for $r$, a natural definition for the fractional Riesz transform is readily seen.~Indeed, taking $r = r_{\min}(\alpha) =  2(\alpha-1)/\alpha$ and letting $\lambda_0 \rightarrow 0^+$, the multiplier $\tilde{m}^{\operatorname{rot}}_{\alpha,\lambda_0,r}$ boils down, for all $\xi \in \bbr^d$ and all $\alpha \in (1,2)$, to
\begin{align*}
\tilde{m}^{\operatorname{rot}}_{\alpha,0,r_{\min}(\alpha)}(\xi) =  i \frac{\alpha}{2^{\frac{1}{\alpha}}} \frac{\xi}{\|\xi\|},
\end{align*}
which is proportional to the Fourier multiplier of the classical Riesz transform of order $1$.~In the general non-degenerate symmetric $\alpha$-stable situation, it does not seem possible to use standard Fourier multiplier theory to establish the $L^p(\bbr^d,dx)$-boundedness (with $p \in (1,+\infty)$) of the family of operators of interest since regularity is lacking. Indeed,  the multiplier then reads, for all $\xi \in \bbr^d$ with $\xi \ne 0$ and with $e_{\xi} = \frac{\xi}{\|\xi\|}$, all $\lambda_0 >0$ and all $r \geq r_{\min}(\alpha)$, as
\begin{align*}
\tilde{m}_{\alpha,\lambda_0,r}(\xi) & = \left(\int_{\bbr^d} u \left(e^{i \langle u ; \xi \rangle} - 1\right) \nu_\alpha(du) \right) \frac{1}{\left(\lambda_0 + \int_{\mathbb{S}^{d-1}} |\langle \xi ; y \rangle|^\alpha \lambda_1(dy)\right)^{\frac{r}{2}}}, \\
& =  \left(\int_{\bbr^d} u \left(e^{i \langle u ; e_{\xi} \rangle} - 1\right) \nu_\alpha(du) \right) \frac{\|\xi\|^{\alpha-1}}{\left(\lambda_0 + \|\xi\|^{\alpha} \int_{\mathbb{S}^{d-1}} |\langle e_{\xi} ; y \rangle|^\alpha \lambda_1(dy)\right)^{\frac{r}{2}}},
\end{align*}
where $\nu_\alpha$ is a non-degenerate symmetric L\'evy measure on $\bbr^d$ satisfying \eqref{eq:scale} and $\lambda_1$ is the associated spectral measure (see Equation \eqref{eq:rep_spectral_measure} below).~Note that the symbol $-\int_{\mathbb{S}^{d-1}} |\langle \xi ; y \rangle|^\alpha \lambda_1(dy)$ can be very different from the one of the fractional Laplacian since $\lambda_1$ might encode anisotropy by being discrete or singular (with respect to the uniform measure on the sphere).~Also, note that the homogeneous fractional Riesz transform in this general case (non-degenerate and symmetric) is given through the following Fourier multiplier: for all $\xi \in \bbr^d$ with $\xi \ne 0$,
\begin{align}\label{eq:homo_frac_Riesz_transform}
\tilde{m}_{\alpha,0,r_{\min}(\alpha)}(\xi) & = \dfrac{\int_{\bbr^d} u \left(e^{i \langle u ; e_{\xi} \rangle} - 1\right) \nu_\alpha(du)}{\bigg(\int_{\mathbb{S}^{d-1}} |\langle e_{\xi} ; y \rangle|^\alpha \lambda_1(dy)\bigg)^{\frac{\alpha-1}{\alpha}}}.
\end{align}
Thus, one has to bypass this lack of regularity (and thus Fourier methods) in order to establish the $L^p(\bbr^d,dx)$-boundedness (with $p \in (1, +\infty)$) of the corresponding operators and, in particular, of the associated fractional Riesz transforms.~Our approach then relies on stochastic representations, known as Bismut-type formulae (see Proposition \ref{prop:Bismut}), for the action of the fractional gradient on the different stable semigroups.~These allow us to completely circumvent the lack of regularity in the general case and to go beyond the classical theory of Fourier multiplier in the Euclidean setting. Our analysis focuses also on obtaining sharp $L^p(\bbr^d,dx)$-$L^p(\bbr^d,dx)$ estimates which produce dimension-free bounds for the fractional Laplacian.

Let us further describe the content of our notes. The next section introduces notations and definitions, used throughout, 
and proves Bismut-type formulae related to the non-degenerate 
symmetric $\alpha$-stable, $\alpha \in (1,2)$, probability measures in $\bbr^d$ 
(Proposition~\ref{prop:Bismut}).~In turn, these formulae imply  
representations of the operators of interest as averages of (singular) integral operators (Proposition \ref{prop:Riesz_Representation}).~In Section~\ref{sec:CC}, the fractional Euclidean setting is completely 
analyzed starting with Lemma~\ref{lem:Hor_Mihlin} which provides a representation 
for the Riesz transform associated with the fractional Laplacian as the composition 
of the classical homogeneous Riesz transform with a certain $L^p(\bbr^d,dx)$-bounded radial multiplier operator. 
Then, via a transference argument the $L^p(\bbr^d,dx)$-boundedness of the Riesz transforms 
associated with some fractional operators 
(stemming from the non-degenerate symmetric $\alpha$-stable probability measures) is obtained  
in Theorem~\ref{prop:Riesz_sharp} as a particular case; leading in 
Corollary~\ref{cor:dimension_free_euclidean} to a dimension-free bound for the rotationally 
invariant $\alpha$-stable probability measure.~The $L^p(\bbr^d,dx)$-boundedness (with $p \in (1,+\infty)$) of the homogeneous fractional Riesz transform, with multiplier given by \eqref{eq:homo_frac_Riesz_transform}, is analyzed in Corollary~\ref{cor_fractional_homo_Riesz_transform} with the help of Lemma \ref{lem:double_truncation} of the Appendix.~Moreover, an integration by parts and standard 
semigroup arguments, lead to reverse inequalities in Proposition~\ref{prop:lower_bound_Riesz_stable_order_1}.  
Finally, in Section~\ref{sec:SC}, a similar analysis is performed when the reference measure is a 
non-degenerate symmetric $\alpha$-stable probability measure with $\alpha \in (1,2)$.~There, Theorem~\ref{thm:First} provides a general stable analogue of the $L^p$-boundedness of the Gaussian operators and implies a dimension-free estimate when the reference stable measure is rotationally invariant.~The paper concludes with an Appendix of technical results used in the previous sections.  

The main results presented below are Theorem~\ref{prop:Riesz_sharp}, 
Corollary~\ref{cor:dimension_free_euclidean}, Corollary~\ref{cor:dimension_euclidean2} and Corollary~\ref{cor_fractional_homo_Riesz_transform} for the homogeneous case and 
Theorem~\ref{thm:First} and Corollary~\ref{thm:Second} for the stable case. In particular, when the underlying 
measure is a rotationally invariant $\alpha$-stable one, these results provide 
dimension-free estimates 
(see \eqref{ineq:continuity_Riesz_Sharp_2} and \eqref{ineq:dimensionfree_Riesz_Stable}).  
Our approach is  rather close to the one initiated by Pisier in \cite{P88} where the transference method of Coifman and Weiss (\cite{GW77}) was used to prove dimension-free bounds, based on very specific 
representations of the classical and Gaussian Riesz transforms.  
Indeed, this methodology is particularly efficient to produce dimension-free estimates for 
multi-dimensional inequalities and has been applied and extended widely in quite a variety of contexts, e.g., \cite{M2,AC94,CMZ96,LP98,LP041,LP042,LP06,ST14}.

\section{Notations and Preliminaries}
\noindent
Throughout, the Euclidean norm on $\bbr^d$ is denoted by $\|\cdot\|$ and the associated 
inner product by $\langle ;\rangle$, while $\gamma$ is  the standard Gaussian probability measure on $\bbr^d$ defined 
through its characteristic function given, for all $\xi \in \bbr^d$, by 
\begin{align}\label{def:gauss}
\hat{\gamma}(\xi) := \int_{\bbr^d} e^{i \langle y ; \xi\rangle} \gamma(dy) =  \exp\left(- \frac{\|\xi\|^2}{2}\right).
\end{align}
Throughout, a L\'evy measure is a positive Borel measure on $\mathbb{R}^d$ such that $\nu(\{0\})=0$ and 
$\int_{\mathbb{R}^d} (1\wedge \|u\|^2)\nu(du)<+\infty$.  Now, let 
$\alpha \in (0,2)$ and let $\nu_\alpha$ be a (non-degenerate) L\'evy measure on $\bbr^d$ such that, for all $c>0$,
\begin{align}\label{eq:scale}
c^{-\alpha}T_c(\nu_\alpha)(du)=\nu_\alpha(du),
\end{align}
where $T_c(\nu_\alpha)(B):=\nu_{\alpha}(B/c)$, for all $B$ Borel 
subset of $\bbr^d$.  
Then, necessarily,  \cite[Theorem $14.3$, (ii)]{S},
$\nu_\alpha$ is given by 
\begin{align}\label{eq:polar}
\nu_\alpha(du) = \bbone_{(0,+\infty)}(r) \bbone_{\mathbb{S}^{d-1}}(y) \dfrac{dr}{r^{\alpha+1}}\sigma(dy),
\end{align}
where $\sigma$ is a finite positive measure on $\mathbb{S}^{d-1}$, the Euclidean unit sphere of $\bbr^d$.   
{\it In the sequel, the measure $\sigma$ is assumed to be symmetric.}

Let $\mu_\alpha$ be the $\alpha$-stable probability measure on $\bbr^d$ defined through its characteristic function, for all 
$\xi \in \bbr^d$, by
\begin{align}\label{def:stable}
\hat{\mu}_\alpha(\xi) := \int_{\bbr^d} e^{i \langle y ; \xi\rangle} \mu_\alpha(dy) = \left\{
    \begin{array}{ll}
        \exp\left(\int_{\bbr^d} (e^{i \langle u;\xi \rangle}-1-i\langle \xi;u\rangle) \nu_\alpha(du)\right),  & \alpha \in (1,2), \\
        \exp\left(\int_{\bbr^d}\left(e^{i \langle u;\xi \rangle} - 1 -i \langle \xi;u\rangle \bbone_{\|u \|\leq 1}\right) \nu_1(du) \right), & \alpha=1, \\
	\exp\left(\int_{\bbr^d} (e^{i \langle u;\xi \rangle}-1) \nu_\alpha(du)\right), & \alpha \in (0,1).
    \end{array}
\right.
\end{align}
Since $\sigma$ is symmetric, \cite[Theorem 14.13.]{S} provides a useful alternative representation for the characteristic function 
$\hat{\mu}_\alpha$ namely, for all $\xi \in \bbr^d$, 
\begin{align}\label{eq:rep_spectral_measure}
\hat{\mu}_\alpha(\xi) 
= \exp\left(- \int_{\mathbb{S}^{d-1}} |\langle y;\xi \rangle|^\alpha \lambda_1(dy)\right),
\end{align}
where $\lambda_1$, the spectral measure, is a symmetric finite positive measure on 
$\mathbb{S}^{d-1}$,  which is 
proportional to $\sigma$, and non-degenerate in that   
\begin{align}\label{eq:non_degeneracy}
\underset{y \in \mathbb{S}^{d-1}}{\inf} \int_{\mathbb{S}^{d-1}} 
|\langle y;x \rangle|^{\alpha} \lambda_1(dx) \ne 0.  
\end{align}
Let $\lambda$ denote a uniform measure on the Euclidean unit sphere of $\bbr^d$ (in the sense of Sato see \cite[Chapter $3$ page $88$]{S}). For $\alpha \in (1,2)$, 
let $\nu_{\alpha}^{\rm {rot}}$ be the L\'evy measure on $\bbr^d$ with polar decomposition 
\begin{align}\label{eq:Levy_Rot}
\nu_{\alpha}^{\rm {rot}}(du) = c_{\alpha,d} \bbone_{(0,+\infty)}(r) \bbone_{\mathbb{S}^{d-1}}(y) \dfrac{dr}{r^{\alpha+1}}\lambda(dy),
\end{align}
where,
\begin{align}\label{eq:renorm}
c_{\alpha,d} = \dfrac{-\alpha (\alpha-1) \Gamma\left(\frac{\alpha+d}{2}\right)}{4 \cos \left(\frac{\alpha\pi}{2}\right)\Gamma\left(\frac{\alpha+1}{2}\right) \pi^{\frac{d-1}{2}}\Gamma(2-\alpha)}, 
\end{align}
and where $\Gamma$ is the usual Euler Gamma function.~Finally, let $\mu_{\alpha}^{\rm {rot}}$ denote the rotationally 
invariant $\alpha$-stable probability measure on $\bbr^d$ with L\'evy measure given by \eqref{eq:Levy_Rot} and with the following normalization: for all $\xi \in \bbr^d$,
\begin{align}\label{eq:charac_rot}
\hat{\mu}_{\alpha}^{\rm{rot}}\left(\xi\right) = \exp\left(-\frac{\|\xi\|^\alpha}{2}\right).
\end{align}
For $\alpha \in (1,2)$, let $\nu_{\alpha,1}$ be the L\'evy measure on $\bbr$ given by
\begin{align}\label{sym:1}
\nu_{\alpha,1}(du) = \frac{c_{\alpha}du}{|u|^{\alpha+1}},
\end{align}
and with,
\begin{align}
c_{\alpha} = \left( \dfrac{-\alpha (\alpha -1)}{4 \Gamma(2-\alpha)\cos\left(\frac{\alpha \pi}{2}\right)} \right).
\end{align}
Let us denote by $\mu_{\alpha,1}$ the $\alpha$-stable probability measure on $\bbr$ with L\'evy measure $\nu_{\alpha,1}$ and with characteristic function defined, for all $\xi \in \bbr$, by
\begin{align}
\hat{\mu}_{\alpha,1}(\xi) = \exp \left(\int_{\bbr} \left(e^{i \langle u;\xi\rangle}-1-i \langle u;\xi \rangle\right) \nu_{\alpha,1}(du)\right) = \exp\left(-\frac{|\xi|^{\alpha}}{2}\right).
\end{align}
Finally, let us denote by $\mu_{\alpha,d} = \mu_{\alpha,1} \otimes \dots \otimes \mu_{\alpha,1}$ the product probability measure on $\bbr^d$ which characteristic function is given, for all $\xi \in \bbr^d$, by
\begin{align}\label{eq:StableIndAxes}
\hat{\mu}_{\alpha,d}(\xi) = \prod_{k=1}^d \hat{\mu}_{\alpha,1}(\xi_k) = \exp\left(\int_{\bbr^d} \left(e^{i \langle \xi;u \rangle}-1-i\langle \xi;u \rangle\right)\nu_{\alpha,d}(du)\right),
\end{align}
and with,
\begin{align}\label{eq:LevyIndAxes}
\nu_{\alpha,d}(du) = \sum_{k=1}^d \delta_0(du_1) \otimes \dots \otimes \delta_0(du_{k-1}) \otimes \nu_{\alpha,1}(du_k) \otimes \delta_0(du_{k+1}) \otimes \dots \otimes \delta_0(du_d),
\end{align}
where $\delta_0$ is the Dirac measure at $0$.~In the sequel, $\mathcal{S}(\bbr^d)$ is the Schwartz space of infinitely differentiable functions which, with all their derivatives of any order, are rapidly decreasing at infinity, and $\mathcal{F}$ is the Fourier transform operator given, for all $f \in \mathcal{S}(\bbr^d)$ and all $\xi \in \bbr^d$, by
\begin{align*}
\mathcal{F}(f)(\xi) = \int_{\bbr^d} f(x) e^{- i \langle x; \xi \rangle} dx.
\end{align*}
On $\mathcal{S}(\bbr^d)$, the Fourier transform is an isomorphism and the following well-known inversion formula holds
\begin{align*}
f(x) = \frac{1}{(2\pi)^d} \int_{\bbr^d} \mathcal{F}(f)(\xi)e^{i \langle \xi ; x\rangle} d\xi,\quad x\in \bbr^d.
\end{align*}
Next, $\mathcal{C}_c^\infty(\bbr^d)$ denotes the space of compactly supported 
infinitely differentiable functions on $\bbr^d$ and $\|\cdot\|_{\infty,\bbr}$ is 
the supremum norm on $\bbr$ given by $\|f\|_{\infty,\bbr} =\sup_{x\in \bbr} |f(x)|$.~As usual, for $p \in (1,+\infty)$, $L^p(\mu_\alpha)$ is the space of (equivalence classes with respect to $\mu_\alpha$-almost everywhere equality) of functions which are Borel measurable and which are $p$-integrable with respect to the probability measure $\mu_\alpha$. This space is endowed with the norm $\|\cdot\|_{L^p(\mu_\alpha)}$ defined, for all suitable $f$, by
\begin{align*}
\|f\|_{L^p(\mu_\alpha)} := \left(\int_{\bbr^d} |f(x)|^p \mu_\alpha(dx)\right)^{\frac{1}{p}}.
\end{align*}
Similarly, and still as usual, for $p \in (1, +\infty)$, $L^p(\bbr^d,dx)$ denotes the classical Lebesgue space where the reference measure is the Lebesgue measure. It is endowed with the norm $\|\cdot\|_{L^p(\bbr^d,dx)}$ defined, for all suitable $f$, by
\begin{align*}
\|f\|_{L^p(\bbr^d,dx)} := \left(\int_{\bbr^d} |f(x)|^p dx\right)^{\frac{1}{p}}.
\end{align*}
At first, let us introduce four semigroups of operators acting on $\mathcal{S}(\bbr^d)$ naturally associated with $\gamma$ and $\mu_\alpha$. Let $(P^H_t)_{t\geq 0}$, $(P^{\gamma}_t)_{t\geq 0}$, $(P^{\alpha}_t)_{t\geq 0}$ and $(P^{\nu_\alpha}_t)_{t\geq 0}$ be defined, for all $f\in \mathcal{S}(\bbr^d)$, all $x\in \bbr^d$ and all $t\geq 0$, by
\begin{align}\label{eq:heatSM}
P^H_t(f)(x) = \int_{\mathbb{R}^d} f(x+\sqrt{t} y) \gamma(dy),
\end{align}
\begin{align}\label{eq:OUSM}
P^{\gamma}_t(f)(x) = \int_{\mathbb{R}^d} f(xe^{-t}+\sqrt{1-e^{-2t}} y) \gamma(dy),
\end{align}
\begin{align}\label{eq:StheatSM}
P^{\alpha}_t(f)(x) = \int_{\mathbb{R}^d} f(x+ t^{\frac{1}{\alpha}} y) \mu_\alpha(dy),
\end{align}
\begin{align}\label{eq:StOUSM}
P^{\nu_\alpha}_t(f)(x) = \int_{\mathbb{R}^d} f(xe^{-t}+(1-e^{-\alpha t})^{\frac{1}{\alpha}} y)\mu_\alpha(dy).
\end{align}
The semigroups \eqref{eq:heatSM} and \eqref{eq:StheatSM} are special cases of convolution semigroups for which a full theory has been well-developed (see \cite{NJ02_1,NJ02_2,NJ02_3}). The semigroup \eqref{eq:OUSM} is the classical Gaussian Ornstein-Uhlenbeck (OU) semigroup and the semigroup \eqref{eq:StOUSM} is the stable OU semigroup associated with the $\alpha$-stable probability measure $\mu_\alpha$ 
which has recently been put forward, in the context of Stein's method for self-decomposable distributions, in  \cite{AH18_1,AH19_2,AH20_3}. 

Next, let us introduce a fifth semigroup linked to $(P^{\nu_\alpha}_t)_{t\geq 0}$ 
which has already been introduced in \cite[Lemma $5.12$ and Remark $5.13$]{AH20_3} for the rotationally 
invariant $\alpha$-stable probability measure on $\bbr^d$.  
Thanks to the representation \eqref{eq:StOUSM}, the stable OU semigroup, $(P^{\nu_\alpha}_t)_{t\geq 0}$, admits extensions to all $L^p(\mu_\alpha)$, $p\in (1,+\infty)$, which form compatible $C_0$-semigroups of Markovian contractions on $L^p(\mu_\alpha)$. Then, let $((P^{\nu_\alpha}_t)^*)_{t\geq 0}$ be the dual semigroup acting in a compatible way on each $L^p(\mu_\alpha)$, $p \in (1,+\infty)$ (see, e.g., \cite[Chapter $1.10$]{Pa83} for general properties of dual semigroup).  With the help of 
Theorem~\ref{def:squared_Mehler_L2} of the Appendix, let us define the ``carr\'e de Mehler" semigroup 
$(\mathcal{P}_t)_{t\geq0}$ on $L^2(\mu_\alpha)$, for all $f \in L^2(\mu_\alpha)$ and all $t\geq 0$, by
\begin{align*}
\mathcal{P}_t(f) = P_{\frac{t}{\alpha}}^{\nu_\alpha} \circ (P_{\frac{t}{\alpha}}^{\nu_\alpha})^*(f) = (P_{\frac{t}{\alpha}}^{\nu_\alpha})^* \circ P_{\frac{t}{\alpha}}^{\nu_\alpha}(f).
\end{align*}
Finally, from Proposition~\ref{prop:squared_Mehler_Lp} of the Appendix, 
the semigroup $(\mathcal{P}_t)_{t\geq 0}$ extends to a $C_0$-semigroup of Markovian contractions on 
each $L^p(\mu_\alpha)$, $p\in (1,+\infty)$, and admits the following representation, 
for all $p \in (1, +\infty)$, all $f \in L^p(\mu_\alpha)$ and all $t\ge 0$,
\begin{align}\label{eq:squared_Mehler}
\mathcal{P}_t(f) = P_{\frac{t}{\alpha}}^{\nu_\alpha} \circ (P_{\frac{t}{\alpha}}^{\nu_\alpha})^*(f) = (P_{\frac{t}{\alpha}}^{\nu_\alpha})^* \circ P_{\frac{t}{\alpha}}^{\nu_\alpha}(f).
\end{align}
\noindent
In the sequel, let us denote by $\partial_{x_k}$ the partial derivative of order $1$ in the variable $x_k$, by $\nabla$ the gradient operator, by $\Delta$ the Laplacian operator and by $D^{\alpha-1}$, $(D^{\alpha-1})^*$ and $\mathbf{D}^{\alpha-1}$ the fractional operators defined, for all $f \in \mathcal{S}(\bbr^d)$ and all $x\in \bbr^d$, by
\begin{align}\label{eq:FracGrad}
D^{\alpha-1}(f)(x) :=\int_{\bbr^d} (f(x+u)-f(x)) u \nu_\alpha(du),
\end{align}
\begin{align}\label{eq:FracGradDual}
(D^{\alpha-1})^*(f)(x) := \int_{\bbr^d} (f(x-u)-f(x)) u \nu_\alpha(du),
\end{align}
\begin{align}\label{eq:FracGradWB}
\mathbf{D}^{\alpha-1}(f)(x) :=\frac{1}{2}\left(D^{\alpha-1}\left(f\right)(x)-(D^{\alpha-1})^*\left(f\right)(x)\right) .
\end{align}
Note that, since the spherical part of the $\alpha$-stable L\'evy measure is symmetric, for all $f \in \mathcal{S}(\bbr^d)$ and all $x\in \bbr^d$, $\mathbf{D}^{\alpha-1}(f)(x)=D^{\alpha-1}\left(f\right)(x)$.
The next result provides representations for the actions of the operators $\nabla$, $D^{\alpha-1}$ and $\mathbf{D}^{\alpha-1}$ on these semigroups. In the Gaussian~Ornstein-Uhlenbeck case, this representation is the well-known Bismut formula. 

\begin{prop}\label{prop:Bismut}
Let $\alpha \in (1,2)$, let $\nu_\alpha$ be a non-degenerate symmetric L\'evy measure on $\bbr^d$ satisfying \eqref{eq:scale} and let $\mu_\alpha$ be the associated $\alpha$-stable probability measure on $\bbr^d$ given by \eqref{def:stable}. Let $(P^H_t)_{t\geq 0}$, $(P^\gamma_t)_{t\geq 0}$, $(P^\alpha_t)_{t\geq 0}$, $(P^{\nu_\alpha}_t)_{t\geq 0}$ be the semigroups respectively defined by \eqref{eq:heatSM}, \eqref{eq:OUSM}, \eqref{eq:StheatSM} and \eqref{eq:StOUSM}.
Then, for all $f\in \mathcal{S}(\bbr^d)$, all $x\in \bbr^d$ and all $t>0$, 
\begin{align}\label{eq:bismut_heat}
\nabla\left(P^H_t(f)\right)(x) = \frac{1}{\sqrt{t}} \int_{\bbr^d}yf(x+\sqrt{t} y) \gamma(dy),
\end{align}
\begin{align}\label{eq:bismut_OU}
\nabla\left(P^\gamma_t(f)\right)(x) =\frac{e^{-t}}{\sqrt{1-e^{-2t}}} \int_{\bbr^d} y f(xe^{-t}+\sqrt{1-e^{-2t}}y) \gamma(dy),
\end{align}
\begin{align}\label{eq:bismut_stable_heat}
\mathbf{D}^{\alpha-1}(P^\alpha_t(f))(x) = \dfrac{1}{t^{1-\frac{1}{\alpha}}} \int_{\bbr^d} y f(x+t^{\frac{1}{\alpha}}y)\mu_\alpha(dy),
\end{align}
\begin{align}\label{eq:bismut_stable_OU}
\mathbf{D}^{\alpha-1}(P^{\nu_\alpha}_t(f))(x) = \dfrac{e^{-(\alpha-1)t}}{\left(1-e^{-\alpha t}\right)^{1-\frac{1}{\alpha}}} \int_{\bbr^d} y f(xe^{-t}+(1-e^{-\alpha t})^{\frac{1}{\alpha}}y) \mu_\alpha(dy).
\end{align}
\end{prop}

\begin{proof}
The proof is divided into two steps. In the first step, the representation formulae are proved for the stable cases with $\alpha \in (1,2)$. Then, in the second step, a suitable renormalization and a limiting argument (as $\alpha \rightarrow 2$) allow to get the corresponding Gaussian representations.\\
\\ 
\textit{Step 1}: Let $\alpha\in (1,2)$ and let $\nu_\alpha$ be a L\'evy measure on $\bbr^d$ satisfying \eqref{eq:scale} such that its spherical component $\sigma$ is symmetric.~By Fourier inversion, for all $f \in \mathcal{S}(\bbr^d)$, all $x \in \bbr^d$ and all $t\geq 0$,
\begin{align*}
P^\alpha_t(f)(x) &=  \int_{\mathbb{R}^d} f(x+ t^{\frac{1}{\alpha}} y) \mu_\alpha(dy), \\
&= \frac{1}{(2\pi)^d} \int_{\bbr^d} \mathcal{F}(f)(\xi) e^{i \langle x;\xi\rangle} \hat{\mu}_\alpha\left(t^{\frac{1}{\alpha}} \xi\right) d\xi.
\end{align*}
Now, again, using Fourier inversion, for all $f \in \mathcal{S}(\bbr^d)$ and all $x\in \bbr^d$,
\begin{align*}
D^{\alpha-1}(f)(x) &=  \int_{\bbr^d} (f(x+u)-f(x)) u \nu_\alpha(du),\\
& = \frac{1}{(2 \pi)^d} \int_{\bbr^d} \mathcal{F}(f)(\xi) e^{i \langle \xi; x\rangle} \tau_\alpha(\xi) d\xi,
\end{align*}
where, for all $ \xi \in \bbr^d$, 
\begin{align}\label{eq:tau_alpha}
\tau_\alpha(\xi) := \int_{\bbr^d} u \left(e^{i \langle u;\xi \rangle}-1\right) \nu_\alpha(du).
\end{align}
Then, for all $f\in \mathcal{S}(\bbr^d)$, all $x\in \bbr^d$ and all $t\geq 0$,
\begin{align*}
D^{\alpha-1}(P^\alpha_t(f))(x) &= \frac{1}{(2\pi)^d} \int_{\bbr^d} \mathcal{F}(f)(\xi) e^{i \langle x; \xi \rangle} \hat{\mu}_\alpha\left(t^{\frac{1}{\alpha}} \xi\right)\tau_\alpha(\xi) d\xi.
\end{align*}
Next, observe that, for all $j \in \{1, \dots, d\}$ and all $\xi \in \bbr^d$, 
\begin{align*}
\partial_{\xi_j}\left(\hat{\mu}_\alpha(\xi)\right) &= \partial_{\xi_j} \left(\exp\left(\int_{\bbr^d} (e^{i \langle u;\xi \rangle}-1-i\langle \xi;u\rangle) \nu_\alpha(du)\right)\right), \\
&= i \int_{\bbr^d} u_j \left(e^{i \langle u; \xi \rangle}-1\right) \hat{\mu}_\alpha(\xi),
\end{align*}
so that, for all $\xi \in \bbr^d$,
\begin{align*}
\nabla\left(\hat{\mu}_\alpha\right)(\xi) = i \tau_\alpha(\xi) \hat{\mu}_\alpha(\xi).
\end{align*}
Thus, for all $\xi \in \bbr^d$ and all $t\ge 0$,
\begin{align*}
\nabla\left(\hat{\mu}_\alpha (t^{\frac{1}{\alpha}} \xi) \right) = i t^{\frac{1}{\alpha}} \tau_\alpha(t^{\frac{1}{\alpha}} \xi) \hat{\mu}_\alpha(t^{\frac{1}{\alpha}} \xi).
\end{align*}
Moreover, by scale invariance, for all $t>0$,
\begin{align*}
\tau_\alpha\left(t^{\frac{1}{\alpha}}\xi\right)& = \int_{\bbr^d} u \left(e^{i \langle u;t^{\frac{1}{\alpha}} \xi \rangle}-1\right) \nu_\alpha(du) ,\\
& =t^{1-\frac{1}{\alpha}}\tau_\alpha(\xi),
\end{align*}
so that, for all $\xi \in \bbr^d$ and all $t>0$,
\begin{align*}
\nabla\left(\hat{\mu}_\alpha (t^{\frac{1}{\alpha}} \xi) \right)=i t \tau_\alpha(\xi)\hat{\mu}_\alpha(t^{\frac{1}{\alpha}} \xi).
\end{align*}
Thus, for all $f\in \mathcal{S}(\bbr^d)$, all $x\in \bbr^d$ and all $t>0$,
\begin{align*}
D^{\alpha-1}(P^\alpha_t(f))(x) &= \frac{1}{(2\pi)^d} \int_{\bbr^d} \mathcal{F}(f)(\xi) e^{i \langle x; \xi \rangle} \hat{\mu}_\alpha\left(t^{\frac{1}{\alpha}} \xi\right)\tau_\alpha(\xi) d\xi,\\
&= \frac{1}{(2\pi)^dit} \int_{\bbr^d} \mathcal{F}(f)(\xi) e^{i \langle x; \xi \rangle}\nabla\left(\hat{\mu}_\alpha(t^{\frac{1}{\alpha}}\xi)\right) d\xi.
\end{align*}
Now, since $\alpha \in (1,2)$, $\int_{\bbr^d} |x_j| \mu_\alpha(dx)<+\infty$, for $j \in\{1, \dots, d\}$, 
and so, for all $\xi \in \bbr^d$ and all $t>0$,
\begin{align*}
\nabla\left(\hat{\mu}_\alpha(t^{\frac{1}{\alpha}}\xi)\right)& = \nabla \left( \bbe e^{i\langle \xi t^{\frac{1}{\alpha}}; X_\alpha \rangle} \right)= i t^{\frac{1}{\alpha}} \bbe \left(X_\alpha e^{i\langle \xi t^{\frac{1}{\alpha}}; X_\alpha \rangle} \right),
\end{align*}
with $X_\alpha \sim \mu_\alpha$.  
Once more, by Fourier inversion, for all $f\in \mathcal{S}(\bbr^d)$, all $x\in \bbr^d$ and all $t>0$,
\begin{align}\label{eq:repD}
D^{\alpha-1}(P^\alpha_t(f))(x) &= \frac{1}{(2\pi)^d t^{1- \frac{1}{\alpha}}}\int_{\bbr^d} \mathcal{F}(f)(\xi) e^{i \langle x; \xi \rangle}\bbe \left(X_\alpha e^{i\langle \xi t^{\frac{1}{\alpha}}; X_\alpha \rangle} \right) d\xi,\nonumber\\
&= \frac{1}{t^{1-\frac{1}{\alpha}}} \bbe \left( X_\alpha f\left(x+t^{\frac{1}{\alpha}}X_\alpha\right)\right).
\end{align}
Then, observe that the symbol of the operator $(D^{\alpha-1})^*$ is given, for all $\xi \in \bbr^d$, by
\begin{align*}
\tau^*_\alpha(\xi) = \int_{\bbr^d} u \left(e^{-i \langle u;\xi\rangle}-1\right) \nu_\alpha(du) = \overline{\tau_\alpha(\xi)
}.
\end{align*}
Thus, for all $f\in \mathcal{S}(\bbr^d)$, all $x\in \bbr^d$ and all $t>0$,
\begin{align*}
(D^{\alpha-1})^*(P^\alpha_t(f))(x) &= \frac{1}{(2\pi)^d} \int_{\bbr^d} \mathcal{F}(f)(\xi) e^{i \langle x; \xi \rangle} \hat{\mu}_\alpha\left(t^{\frac{1}{\alpha}} \xi\right)\overline{\tau_\alpha(\xi)} d\xi.  
\end{align*}
But, since $\sigma$ is symmetric, the function $\hat{\mu}_\alpha$ is real-valued, and so 
\begin{align*}
\nabla\left(\hat{\mu}_\alpha (t^{\frac{1}{\alpha}} \xi) \right) = \overline{\nabla\left(\hat{\mu}_\alpha (t^{\frac{1}{\alpha}} \xi) \right)}=\overline{i t \tau_\alpha(\xi)\hat{\mu}_\alpha(t^{\frac{1}{\alpha}} \xi)} = -it \overline{\tau_\alpha(\xi)} \hat{\mu}_\alpha(t^{\frac{1}{\alpha}} \xi),
\end{align*}
for all $\xi \in \bbr^d$ and all $t>0$.  
Then, for all $f\in \mathcal{S}(\bbr^d)$, all $x\in \bbr^d$ and all $t>0$,
\begin{align}\label{eq:}
(D^{\alpha-1})^*(P^\alpha_t(f))(x) &= \frac{1}{(2\pi)^d} \int_{\bbr^d} \mathcal{F}(f)(\xi) e^{i \langle x; \xi \rangle} \hat{\mu}_\alpha\left(t^{\frac{1}{\alpha}} \xi\right)\overline{\tau_\alpha(\xi)} d\xi,\nonumber\\
& = \frac{1}{-(2\pi)^d it} \int_{\bbr^d} \mathcal{F}(f)(\xi) e^{i \langle x; \xi \rangle} \nabla\left(\hat{\mu}_\alpha (t^{\frac{1}{\alpha}} \xi) \right)  d\xi,\nonumber\\
&= - \frac{1}{t^{1-\frac{1}{\alpha}}} \bbe \left(X_\alpha f\left(x+t^{\frac{1}{\alpha}}X_\alpha\right)\right).
\end{align}
The representation formula for $\mathbf{D}^{\alpha-1}(P^\alpha_t(f))$ then follows from \eqref{eq:FracGradWB}. 
A completely analogous reasoning allows to get the representation \eqref{eq:bismut_stable_OU}. 
This concludes Step $1$ of the proof.\\
\\
\textit{Step 2}: For $\nu_{\alpha}^{\rm {rot}}$ and $\mu_\alpha^{\rm rot}$ as respectively defined by \eqref{eq:Levy_Rot} and by \eqref{eq:charac_rot} and 
all $\xi \in \bbr^d$ such that $\xi\ne0$,
\begin{align*}
\tau_\alpha^{\rm rot}(\xi) &= \int_{\bbr^d} u \left(e^{i \langle u;\xi\rangle}-1\right) \nu_\alpha^{\rm rot}(du) =c_{\alpha,d} \int_{(0,+\infty)\times\mathbb{S}^{d-1}} ry \left(e^{ir\langle y;\xi \rangle}-1\right) \frac{dr}{r^{1+\alpha}} \lambda(dy),\\
&=c_{\alpha,d} \frac{\Gamma(2-\alpha)}{1-\alpha} \int_{\mathbb{S}^{d-1}} y |\langle\xi;y\rangle|^{\alpha-1}\left(i \operatorname{sign}\left(\langle\xi;y\rangle\right)\cos \left(\frac{\alpha \pi}{2}\right) + \sin \left(\frac{\alpha \pi}{2}\right)\right) \lambda(dy),\\
&= i c_{\alpha,d} \frac{\Gamma(2-\alpha)}{1-\alpha}\cos \left(\frac{\alpha \pi}{2}\right) \int_{\mathbb{S}^{d-1}} y |\langle\xi;y\rangle|^{\alpha-1} \operatorname{sign}\left(\langle\xi;y\rangle\right) \lambda(dy),\\
&= i \frac{\alpha \Gamma\left(\frac{\alpha+d}{2}\right)}{4 \Gamma\left(\frac{\alpha+1}{2}\right) \pi^{\frac{d-1}{2}}}\int_{\mathbb{S}^{d-1}} y |\langle\xi;y\rangle|^{\alpha-1} \operatorname{sign}\left(\langle\xi;y\rangle\right) \lambda(dy).
\end{align*}
Then, for all $\xi \in \bbr^d$ with $\xi \ne 0$,
\begin{align*}
\int_{\mathbb{S}^{d-1}} y |\langle\xi;y\rangle|^{\alpha-1} \operatorname{sign}\left(\langle\xi;y\rangle\right) \lambda(dy) \underset{\alpha \rightarrow 2}{\longrightarrow}  \int_{\mathbb{S}^{d-1}} y \langle\xi;y\rangle \lambda(dy) = \dfrac{ \pi^{\frac{d}{2}}}{\Gamma\left(\frac{d+2}{2}\right)} \xi, 
\end{align*}
and thus,
\begin{align*}
\tau_\alpha^{\rm rot}(\xi) \underset{\alpha \rightarrow 2}{ \longrightarrow }i\xi \frac{\sqrt{\pi}}{2\Gamma\left(\frac{3}{2}\right)} = i\xi.
\end{align*}
Now, letting $\alpha \rightarrow 2$ on both sides of the equality \eqref{eq:repD} with the corresponding 
$\mu_\alpha^{\rm rot}$, and  
for all $f\in \mathcal{S}(\bbr^d)$, all $x\in \bbr^d$ and all $t>0$, 
\begin{align*}
\nabla(P^H_t(f))(x)  =  \frac{1}{\sqrt{t}} \bbe \left( X f\left(x+\sqrt{t} X\right)\right),
\end{align*}
where $X\sim \gamma$, i.e., the law of $X$ is $\gamma$. A similar reasoning with the semigroup $\left(P^{\nu_\alpha^{\rm rot}}_t\right)_{t\geq 0}$ allows to retrieve the classical representation \eqref{eq:bismut_OU}. This concludes the proof of the proposition. 

\end{proof}
\noindent
The generators of the above four semigroups can be obtained through their Fourier representations, and it is straightforward to check that, for $\alpha \in (1,2)$, they are respectively given by
\begin{align}\label{eq:heatgen}
\mathcal{A}_H(f)(x) =\frac{1}{2} \Delta(f)(x),
\end{align}
\begin{align}\label{eq:OUgen}
\mathcal{L}^\gamma(f)(x)=-\langle x;\nabla(f)(x)\rangle + \Delta(f)(x),
\end{align}
\begin{align}\label{eq:Stheatgen}
\mathcal{A}_\alpha(f)(x) = \int_{\bbr^d} \left(f(x+u)-f(x)-\langle u; \nabla(f)(x) \rangle\right) \nu_\alpha(du),
\end{align}
\begin{align}\label{eq:StOUgen}
\mathcal{L}^\alpha(f)(x)= -\langle x;\nabla(f)(x) \rangle + \int_{\bbr^d} \langle \nabla(f)(x+u)-\nabla(f)(x);u \rangle \nu_\alpha(du), 
\end{align}
for all $f\in \mathcal{S}(\bbr^d)$ and all $x\in \bbr^d$.

The generator of the ``carr\'e de Mehler" semigroup is denoted by $\mathcal{L}$ and studied 
in Theorem~\ref{def:squared_Mehler_L2} and Proposition~\ref{prop:squared_Mehler_Lp} 
of the Appendix.

From the representations obtained in Proposition~\ref{prop:Bismut}, it is in turn possible to obtain representations 
for the operators of interest.~To do so, let us first recall the definition of the gamma transform.  
For $(P_t)_{t\geq 0}$ a $C_0$-semigroup of contractions, on a Banach space, having generator $\mathcal{A}$, 
its gamma transform of order $r>0$ is defined, for all suitable $f$, by
\begin{align}\label{eq:Gam_Tr}
\left(E-\mathcal{A} \right)^{-\frac{r}{2}} f = \frac{1}{\Gamma(\frac{r}{2})} \int_0^{+\infty} \dfrac{e^{-t}}{t^{1-\frac{r}{2}}} P_t(f) dt,
\end{align}
where $E$ is the identity operator and where the above integral is understood in the Bochner sense. 

With the above, we now have:

\begin{prop}\label{prop:Riesz_Representation}
Let $\alpha \in (1,2)$, let $r>2(\alpha-1)/\alpha$ (or let $r>1$ in the Gaussian case),~let $\nu_\alpha$ be a non-degenerate symmetric L\'evy measure on $\bbr^d$ satisfying \eqref{eq:scale} and let $\mu_\alpha$ be the associated $\alpha$-stable probability measure on $\bbr^d$ given by \eqref{def:stable}. Then, for all $f\in \mathcal{S}(\bbr^d)$ and all $x\in \bbr^d$,
\begin{align}\label{eq:Riesz_Heat}
\nabla \circ \left(E-\mathcal{A}_H\right)^{-\frac{r}{2}}(f)(x) & =\frac{1}{2}\int_{\bbr^d} y \left( \int_{0}^{+\infty}\dfrac{e^{-t}}{\Gamma(\frac{r}{2})t^{\frac{3}{2}-\frac{r}{2}}} \left(f(x+\sqrt{t} y)- f(x-\sqrt{t} y)\right)  dt \right) \gamma(dy),
\end{align}
\begin{align}\label{eq:Riesz_OU}
\nabla \circ \left(E-\mathcal{L}^\gamma\right)^{-\frac{r}{2}} (f) (x)& = \frac{1}{2} \int_{\bbr^d} y \bigg(\int_0^{+\infty}\dfrac{e^{-2t}}{\Gamma(\frac{r}{2})t^{1-\frac{r}{2}}\sqrt{1-e^{-2t}}}\bigg( f(xe^{-t}+\sqrt{1-e^{-2t}}y)\nonumber\\
&\qquad\qquad\qquad-f(xe^{-t}-\sqrt{1-e^{-2t}}y)\bigg) dt \bigg) \gamma(dy),
\end{align}
\begin{align}\label{eq:Riesz_StableH}
\mathbf{D}^{\alpha-1} \circ \left(E-\mathcal{A}_\alpha\right)^{-\frac{r}{2}}(f)(x) & = \frac{1}{2} \int_{\bbr^d} y \left( \int_{0}^{+\infty} \dfrac{e^{-t}}{\Gamma(\frac{r}{2}) t^{2-\frac{1}{\alpha}-\frac{r}{2}}} \bigg(f(x+t^{\frac{1}{\alpha}} y)- f(x-t^{\frac{1}{\alpha}} y)\bigg) dt\right) \mu_\alpha(dy),
\end{align}
\begin{align}\label{eq:Riesz_OU_Stable}
\mathbf {D}^{\alpha-1} \circ \left(E-\mathcal{L}^\alpha\right)^{-\frac{r}{2}}(f)(x) & =\frac{1}{2} \int_{\bbr^d} y \bigg(\int_{0}^{+\infty} \dfrac{e^{-\alpha t}}{\Gamma(\frac{r}{2})t^{1-\frac{r}{2}}\left(1-e^{-\alpha t}\right)^{1-\frac{1}{\alpha}}}\bigg(f\left(x e^{-t}+(1-e^{-\alpha t})^{\frac{1}{\alpha}} y\right)\nonumber\\
&\qquad\qquad\qquad- f\left(x e^{-t}-(1-e^{-\alpha t})^{\frac{1}{\alpha}} y\right)\bigg) dt\bigg) \mu_\alpha(dy).
\end{align}
\end{prop}

\begin{proof}
Let $\alpha \in (1,2)$, let $r>2(\alpha-1)/\alpha$ and let $\nu_\alpha$ be as in the hypotheses.   
From \eqref{eq:repD} and \eqref{eq:Gam_Tr} it follows that 
for all $f \in \mathcal{S}(\bbr^d)$ and all $x \in \bbr^d$,
\begin{align*}
D^{\alpha-1} \circ \left(E-\mathcal{A}_\alpha\right)^{-\frac{r}{2}}(f)(x)&= \frac{1}{\Gamma(\frac{r}{2})} \int_{0}^{+\infty} \frac{e^{-t}}{t^{1-\frac{r}{2}}} D^{\alpha-1} \left(P^\alpha_t(f)\right)(x) dt,\\
&= \frac{1}{\Gamma(\frac{r}{2})} \int_{0}^{+\infty} \frac{e^{-t}}{t^{2-\frac{1}{\alpha}-\frac{r}{2}}} \bbe \left( X_\alpha f\left(x+t^{\frac{1}{\alpha}}X_\alpha\right)\right)dt.
\end{align*}
Since $\sigma$ is symmetric, $X_\alpha =_{\mathcal{L}} -X_\alpha$, where $=_\mathcal{L}$ indicates equality in law, (see, e.g., \cite[Theorem $14.13$]{S}). 
Thus, for all $t\ge 0$,
\begin{align*}
\bbe \left( X_\alpha f\left(x+t^{\frac{1}{\alpha}}X_\alpha\right)\right) = \bbe \left( -X_\alpha f\left(x-t^{\frac{1}{\alpha}}X_\alpha\right)\right),
\end{align*}
and, for all $f \in \mathcal{S}(\bbr^d)$ and all $x\in \bbr^d$,
\begin{align*}
D^{\alpha-1} \circ \left(E-\mathcal{A}_\alpha\right)^{-\frac{r}{2}}(f)(x)&= \frac{1}{2\Gamma(\frac{r}{2})} \int_{0}^{+\infty} \frac{e^{-t}}{t^{2-\frac{1}{\alpha}-\frac{r}{2}}} \left(\bbe \bigg( X_\alpha f\left(x+t^{\frac{1}{\alpha}}X_\alpha\right)\right) \\
& \quad\quad- \bbe \left(X_\alpha f\left(x-t^{\frac{1}{\alpha}}X_\alpha\right)\right)\bigg) dt.
\end{align*}
Next, a Fubini's argument ensures that, for all $f \in \mathcal{S}(\bbr^d)$ and all $x\in \bbr^d$,
\begin{align*}
D^{\alpha-1} \circ \left(E-\mathcal{A}_\alpha\right)^{-\frac{r}{2}}(f)(x)&= \frac{1}{2\Gamma(\frac{r}{2})} \int_{\bbr^d} y \left(\int_0^{+\infty} \frac{e^{-t}}{t^{2-\frac{r}{2}-\frac{1}{\alpha}}} \left(f\left(x+t^{\frac{1}{\alpha}}y\right)- f\left(x-t^{\frac{1}{\alpha}}y\right)\right) dt\right) \mu_\alpha(dy).
\end{align*}
A similar reasoning provides an analogous representation for the action of 
$(D^{\alpha-1})^* \circ \left(E-\mathcal{A}_\alpha\right)^{-\frac{r}{2}}$ on $\mathcal{S}(\bbr^d)$ leading to 
\eqref{eq:Riesz_StableH}.  
The identities \eqref{eq:Riesz_Heat}, \eqref{eq:Riesz_OU} and \eqref{eq:Riesz_OU_Stable} 
can be obtained in a completely similar manner. 
\end{proof}
\noindent
In the next two sections, the $L^p$-continuity properties of operators associated with the four semigroups  respectively defined in \eqref{eq:heatSM}--\eqref{eq:StOUSM} are investigated.  While the results on 
$(P^H_t)_{t\geq 0}$ and $(P_t^\gamma)_{t\geq 0}$ are very classical, the ones on 
$(P^\alpha_t)_{t\geq 0}$ and $(P_t^{\nu_\alpha})_{t\geq 0}$ seem to be new.~(For related results 
regarding the fractional Laplacians on $\bbr^d$ see \cite{JMP18}, for results regarding characterizations of the Bessel fractional spaces see \cite{Strichartz67}, and for Fourier multiplier theorems in connection with L\'evy processes see \cite{BB07,BBB11,BanOse12,BBL16,BK19}.)~Note, nevertheless, that the method of proofs which is based on 
Bismut-type formulae is reminiscent of the one developed by Pisier in \cite{P88}.

\section{The Classical Cases}\label{sec:CC}
\noindent
Starting with the rotationally invariant case, this section studies continuity properties of the operators 
$\mathbf{D}^{\alpha-1,\operatorname{rot}} \circ \left(\lambda_0 E-\mathcal{A}^{\operatorname{rot}}_\alpha\right)^{-\frac{r}{2}}$, for $\alpha \in (1,2)$, $r>0$ and $\lambda_0 > 0$. In particular, an appropriate choice of $r>0$ provides an operator which is similar to a Riesz transform in this fractional setting. Indeed, thanks to Lemma \ref{lem:Hor_Mihlin}, $r = \frac{2}{\alpha}(\alpha-1)$ ensures that the associated operator has good homogeneity properties (when $\lambda_0$ tends to $0$) which is the natural requirement for a definition of a Riesz transform.

The initial result uses the 
$L^p(\bbr^d,dx)$-boundedness of the classical homogeneous Riesz transform 
(see, e.g., \cite{D01, G08}) to obtain dimension-dependent estimates.   
For this purpose, recall the spherical representation of this classical transform, e.g.,  see \cite[Proof of Theorem $1$]{DRF85}:  For all $f\in \mathcal{S}(\bbr^d)$, all $x\in \bbr^d$ and all $j\in \{1, \dots, d\}$,
\begin{align}\label{eq:sphe_int_rep}
\mathcal{R}^{2,h}_j(f)(x) = \frac{\Gamma\left(\frac{d+1}{2}\right)}{2\pi^{\frac{d-1}{2}}} \int_{\mathbb{S}^{d-1}} y_j H_{y}(f)(x) \sigma_L(dy),
\end{align}
where $H_y$ is a directional Hilbert transform characterized in the Fourier domain by the directional bounded multiplier $m_y$ defined, for all $\xi \in \bbr^d$ and all $y\in \mathbb{S}^{d-1}$, by
\begin{align}\label{eq:sphe_mult}
m_y(\xi) = -i  \rm{sign}\left(\langle y ; \xi \rangle\right),
\end{align}
and where $\sigma_L$ is the spherical part of the $d$-dimensional Lebesgue measure.~In turn, the classical homogeneous Riesz transform admits the following Fourier representation (see, e.g., \cite[Proposition 4.1.14.]{G08}); valid for all $f\in \mathcal{S}(\bbr^d)$, all $\xi \in \bbr^d$ with $\xi \ne 0$, and 
all $j \in \{1, \dots, d\}$, 
\begin{align}\label{eq:Fourier_Riesz}
\mathcal{F}\left(\mathcal{R}^{2,h}_j(f)\right)(\xi) = -\frac{i\xi_j}{\|\xi\|} \mathcal{F}(f)(\xi).
\end{align}

\begin{lem}\label{lem:Hor_Mihlin}
Let $\alpha \in (1,2)$, let $r>0$ and let $\nu_{\alpha}^{\rm rot}$ be the L\'evy measure given by 
\eqref{eq:Levy_Rot}. Then, for all $f\in \mathcal{S}(\bbr^d)$, all $x\in \bbr^d$, and all $j \in \{1, \dots, d\}$, 
\begin{align}\label{eq:rep_sphe}
\mathbf{D}^{\alpha-1,\operatorname{rot}}_j \circ \left(E-\mathcal{A}_\alpha^{\operatorname{rot}}\right)^{-\frac{r}{2}}(f)(x) = \mathcal{R}^{2,h}_j\left(S_{\alpha,r}(f)\right)(x) = S_{\alpha,r}(\mathcal{R}^{2,h}_j(f))(x), 
\end{align}
where $S_{\alpha,r}$ is a linear operator defined (in the Fourier domain) by  
\begin{align}\label{eq:def_S}
\mathcal{F}\left(S_{\alpha,r}(f)\right)(\xi) = \frac{-\alpha 2^{\frac{r}{2}}\|\xi\|^{\alpha-1}}{2\left(2+ \|\xi\|^\alpha\right)^{\frac{r}{2}}} \mathcal{F}(f)(\xi).   
\end{align}
Moreover, for all $r \in [r_{\min}(\alpha),+\infty)$, all $p\in (1,+\infty)$, all $f \in L^p(\bbr^d,dx)$ and all $j \in \{1, \dots, d\}$, 
\begin{align}\label{eq:Lp_cont}
\|\mathbf{D}^{\alpha-1,\operatorname{rot}}_j \circ \left(E-\mathcal{A}_\alpha^{\operatorname{rot}}\right)^{-\frac{r}{2}}(f)\|_{L^p(\bbr^d, dx)} \leq C_{\alpha,d,p,r} \|f\|_{L^p(\bbr^d, dx)},
\end{align}
for some $C_{\alpha,d,p,r}>0$ depending on $\alpha$, on $d$, on $p$ and on $r$ and with $r_{\min}(\alpha)$ defined by 
\begin{align}\label{eq:r_min}
r_{\min}(\alpha):= \frac{2}{\alpha}(\alpha-1). 
\end{align}
\end{lem}

\begin{proof}
\textit{Step 1}: Let us start with the proof of \eqref{eq:rep_sphe}. Thanks to the normalization chosen for $\mu_\alpha^{\operatorname{rot}}$, it follows that the operator $ \left(E-\mathcal{A}_\alpha^{\operatorname{rot}}\right)^{-\frac{r}{2}}$ is defined on the Fourier side, for all $f \in \mathcal{S}(\bbr^d)$ and all $\xi \in \bbr^d$, by
\begin{align*}
\mathcal{F}\left(\left(E-\mathcal{A}_\alpha^{\operatorname{rot}}\right)^{-\frac{r}{2}}(f)\right)(\xi) = \dfrac{2^{\frac{r}{2}}\mathcal{F}(f)(\xi)}{\left(2+\|\xi\|^{\alpha}\right)^{\frac{r}{2}}}. 
\end{align*}
Moreover, the symbol of the operator $\mathbf{D}^{\alpha-1,\operatorname{rot}}$ denoted by $\tau_\alpha^{\operatorname{rot}}$ is given, for all $\xi \in \bbr^d$, $\xi \ne 0$, by 
\begin{align*}
\tau_\alpha^{\operatorname{rot}}(\xi) = \int_{\bbr^d} u \left(e^{i \langle u ; \xi \rangle}-1\right) \nu_{\alpha}^{\operatorname{rot}}(du) = \frac{i \alpha}{2} \frac{\xi}{\|\xi\|} \|\xi\|^{\alpha-1}. 
\end{align*}
Then, by Fourier inversion formula, it follows easily that, for all $f \in \mathcal{S}(\bbr^d)$ and all $x\in \bbr^d$, 
\begin{align*}
\mathbf{D}^{\alpha-1,\operatorname{rot}}_j \circ \left(E-\mathcal{A}_\alpha^{\operatorname{rot}}\right)^{-\frac{r}{2}}(f)(x) &= \frac{1}{(2\pi)^d} \int_{\bbr^d} \mathcal{F}(f)(\xi) e^{i \langle x; \xi \rangle}  \frac{i \alpha}{2} \frac{\xi}{\|\xi\|}   \dfrac{2^{\frac{r}{2}}\|\xi\|^{\alpha-1}}{\left(2+\|\xi\|^{\alpha}\right)^{\frac{r}{2}}} d\xi. 
\end{align*}
Next, let $S_{\alpha,r}$ be the linear operator defined, for all $f\in \mathcal{S}(\bbr^d)$ and all $\xi \in \bbr^d$, by
\begin{align*}
\mathcal{F}(S_{\alpha,r}(f))(\xi) = \frac{-\alpha 2^{\frac{r}{2}}\|\xi\|^{\alpha-1}}{2\left(2+ \|\xi\|^\alpha\right)^{\frac{r}{2}}} \mathcal{F}(f)(\xi).
\end{align*}
Clearly, from \eqref{eq:Fourier_Riesz}, for all $f \in \mathcal{S}(\bbr^d)$ and all $x \in \bbr^d$,
\begin{align*}
\mathbf{D}^{\alpha-1,\operatorname{rot}}_j \circ \left(E-\mathcal{A}^{\operatorname{rot}}_\alpha\right)^{-\frac{r}{2}}(f)(x) = \mathcal{R}^{2,h}_j(S_{\alpha,r}(f))(x). 
\end{align*}
\textit{Step 2}: To conclude let us prove that, for all $p \in (1,+\infty)$, 
$S_{\alpha,r}$ is a bounded linear operator from $L^p(\bbr^d,dx)$ to itself as soon as $r \in [r_{\min}(\alpha),+\infty)$ with $r_{\min}(\alpha) = 2(\alpha-1)/\alpha$. Let $m_{\alpha,r}$ be defined, for all $\xi \in \bbr^d$, by
\begin{align*}
m_{\alpha,r}(\xi) := \frac{-\alpha 2^{\frac{r}{2}}\|\xi\|^{\alpha-1}}{2\left(2+ \|\xi\|^\alpha\right)^{\frac{r}{2}}}.
\end{align*}
Note that $m_{\alpha,r}$ is bounded on $\bbr^d$ as soon as $r \geq 2(\alpha-1)/\alpha$. Now, let $\psi_{\alpha,r}$ be the function defined, for all $s\in (0,+\infty)$, by
\begin{align*}
\psi_{\alpha,r}(s):=  \frac{-\alpha 2^{\frac{r}{2}}s^{\frac{\alpha-1}{2}}}{2\left(2+ s^{\frac{\alpha}{2}}\right)^{\frac{r}{2}}}, 
\end{align*}
so that $\psi_{\alpha,r}\left(\|\xi\|^2\right) = m_{\alpha,r}(\xi)$, for $\xi \in \bbr^d$. Let $\beta$ be a multi-index such that $|\beta| \leq \lfloor \frac{d}{2} \rfloor+1$, where $|\beta| = \sum_{k=1}^d \beta_k$ and where $\lfloor x \rfloor$ denotes the integer part of the real number $x$.

First, observe that, for all $s>0$ and all $k \in \{1,\dots, \lfloor \frac{d}{2} \rfloor+1\}$,
\begin{align}\label{eq:derivatives}
\psi^{(k)}_{\alpha,r}(s) = \dfrac{s^{\frac{\alpha-1}{2}-k}P_{\alpha,k,r}(s^{\frac{\alpha}{2}})}{\left(2+s^\frac{\alpha}{2}\right)^{\frac{r}{2}+k}},
\end{align}
where $P_{\alpha,k,r}$ is a polynomial of degree $k$ with coefficients depending on $\alpha$, on $k$ and on $r$ (see Lemma \ref{lem:simple} of the Appendix). 
Let us then start the bounding procedures with the extreme cases.   To do so, 
let $j \in \{1, \dots, d\}$ and assume that $\beta=(0, \dots, 0, |\beta|,0,\dots,0)$ where the non-zero coordinate 
is at the $j$-th position.  
Note that $\partial_{\xi_j}\left(\|\xi\|^2\right) = 2 \xi_j$, $\partial^2_{\xi_j}\left(\|\xi\|^2\right) = 2$ and $\partial^k_{\xi_j}\left(\|\xi\|^2\right)=0$, for all $k \geq 3$ (and also $\partial^2_{\xi_j,\xi_m}\left(\|\xi\|^2\right) =0$, for all $j,m \in \{1, \dots, d\}$ with $j \ne m$). Then, by a recursive argument (see Lemma \ref{lem:simple2} of the Appendix), for all $\xi \in \bbr^d$, $\xi \ne 0$,
\begin{align*}
D^{\beta}_{\xi}\left(\psi_{\alpha,r}\left(\|\xi\|^2\right)\right) =  \partial_{\xi_j}^{|\beta|}\left(\psi_{\alpha,r}\left(\|\xi\|^2\right)\right) = \sum_{p=0}^{\lfloor\frac{|\beta|}{2}\rfloor} C_{p}\left(|\beta|\right) (\xi_j)^{|\beta|-2p} \psi_{\alpha,r}^{(|\beta|-p)}\left(\|\xi\|^2\right),
\end{align*}
for some constants $C_{p}\left(|\beta|\right)>0$ only depending on $p$ and on $|\beta|$. Thus, for all $\xi \in \bbr^d$, 
\begin{align}\label{ineq:bound_1_Hor_Mihlin}
\left| D^{\beta}_{\xi}\left(\psi_{\alpha,r}\left(\|\xi\|^2\right)\right) \right| &\leq \sum_{p=0}^{\lfloor\frac{|\beta|}{2}\rfloor} C_p\left(|\beta|\right) \|\xi\|^{|\beta|-2p} \left| \psi_{\alpha,r}^{(|\beta|-p)}\left(\|\xi\|^2\right)\right|.
\end{align}
Using \eqref{eq:derivatives}, for all $\xi \in \bbr^d$ with $\xi \ne 0$, and all $p \in \left\{0, \dots, \lfloor\frac{|\beta|}{2}\rfloor\right\}$, 
\begin{align*}
\|\xi\|^{|\beta|-2p} \left| \psi_{\alpha,r}^{(|\beta|-p)}\left(\|\xi\|^2\right)\right| & \leq \|\xi\|^{|\beta|-2p} \dfrac{\|\xi\|^{\alpha-1-2|\beta|+2p}|P_{\alpha,|\beta|-p,r}(\|\xi\|^{\alpha})|}{\left(2+\|\xi\|^\alpha\right)^{\frac{r}{2}+|\beta|-p}},\\
& \leq C_{\alpha,\beta,p,r} \left(1+\dots + \|\xi\|^{\alpha(|\beta|-p)}\right) \dfrac{\|\xi\|^{\alpha-1-|\beta|}}{\left(2+\|\xi\|^{\alpha}\right)^{\frac{r}{2}+|\beta|-p}} ,\\
&\leq C_{\alpha,\beta,p,r} \left(1+\dots + \|\xi\|^{\alpha(|\beta|-p)}\right) \dfrac{\|\xi\|^{\alpha-1}}{\left(2+\|\xi\|^{\alpha}\right)^{\frac{r}{2}+|\beta|-p}} \|\xi\|^{-|\beta|}.
\end{align*}
Now, notice that the non-negative 
function $\Psi_{\alpha,|\beta|,r}$ defined, for all $s \geq 0$, by
\begin{align*}
\Psi_{\alpha,|\beta|,r}(s) = \left(1+\dots + s^{\alpha(|\beta|-p)}\right) \dfrac{s^{\alpha-1}}{\left(2+s^{\alpha}\right)^{\frac{r}{2}+|\beta|-p}},
\end{align*}
is bounded on $[0,+\infty)$ as soon as $r \geq 2(\alpha-1)/\alpha$. Thus, for all $\xi \in \bbr^d$ with $\xi \ne 0$ and all $p \in \left\{0, \dots, \lfloor\frac{|\beta|}{2}\rfloor\right\}$,
\begin{align}\label{ineq:bound_2_Hor_Mihlin}
\|\xi\|^{|\beta|-2p} \left| \psi_{\alpha,r}^{(|\beta|-p)}\left(\|\xi\|^2\right)\right| & \leq \tilde{C}_{\alpha,\beta,p,r} \|\xi\|^{-|\beta|},
\end{align}
for some $\tilde{C}_{\alpha,\beta,p}>0$ only depending on $\alpha$, $\beta$ and $p$. Combining \eqref{ineq:bound_1_Hor_Mihlin} and \eqref{ineq:bound_2_Hor_Mihlin} leads to, 
\begin{align}\label{ineq:bound_3_Hor_Mihlin}
\left| D^{\beta}_{\xi}\left(\psi_{\alpha,r}\left(\|\xi\|^2\right)\right) \right| &\leq \tilde{C}_{\alpha,\beta,r} \|\xi\|^{-|\beta|},
\end{align}
for all $\xi \in \bbr^d$, $\xi \ne 0$.  
Next, let $j_1, \dots, j_{|\beta|}\in \{1, \dots, d\}$ be such that $j_k\ne j_{\ell}$ for all $k \ne \ell$ 
in $\{1, \dots, |\beta|\}$ (note that $|\beta| \leq d$). Then, for all $\xi \in \bbr^d$, $\xi \ne 0$,
\begin{align*}
D^{\beta}_{\xi}\left(\psi_{\alpha,r}\left(\|\xi\|^2\right)\right) = \partial^{|\beta|}_{\xi_{j_1} \dots \xi_{j_{|\beta|}}}\left(\psi_{\alpha,r}\left(\|\xi\|^2\right)\right) = 2^{|\beta|} \xi_{j_1} \dots \xi_{j_{|\beta|}} \psi_{\alpha,r}^{(|\beta|)}\left(\|\xi\|^2\right) .
\end{align*}
Thus, 
\begin{align*}
|D^{\beta}_{\xi}\left(\psi_{\alpha,r}\left(\|\xi\|^2\right)\right)|&\leq 2^{|\beta|} \|\xi\|^{|\beta|} \left|\psi_{\alpha,r}^{(|\beta|)}\left(\|\xi\|^2\right)\right|,\\
&\leq C_{\alpha,|\beta|,r} \|\xi\|^{|\beta|} \dfrac{\|\xi\|^{\alpha-1-2|\beta|}\left(1+\dots + \|\xi\|^{\alpha |\beta|}\right)}{(2+ \|\xi\|^{\alpha})^{\frac{r}{2}+|\beta|}}\\
&\leq C_{\alpha,|\beta|,r} \dfrac{\|\xi\|^{\alpha-1}\left(1+\dots + \|\xi\|^{\alpha |\beta|}\right)}{(2+ \|\xi\|^{\alpha})^{\frac{r}{2}+|\beta|}} \|\xi\|^{-|\beta|},\\
&\leq \tilde{C}_{\alpha,|\beta|,r} \|\xi\|^{-|\beta|}.
\end{align*}
To conclude, let us deal with the general case. Let $\gamma_0\in \{1,\dots,|\beta| \}$, let $j_1,\dots, j_{\gamma_0}\in \{1, \dots, d\}$ be such that $j_k\ne j_{\ell}$ for all $k \ne \ell$ in $\{1, \dots, \gamma_0 \}$ and let $m_1, \dots, m_{\gamma_0}$ be integers greater than $1$ such that $\sum_{\ell=1}^{\gamma_0} m_{\ell} = |\beta|$.  
Then, for all $\xi \in \bbr^d$, $\xi \ne 0$,
\begin{align*}
D^{\beta}_{\xi} \left(\psi_{\alpha,r}(\|\xi\|^2)\right) = \partial^{m_1}_{\xi_{j_1}} \dots  \partial^{m_{\gamma_0}}_{\xi_{j_{\gamma_0}}} \left(\psi_{\alpha,r}(\|\xi\|^2)\right), 
\end{align*}
and using Lemma~\ref{lem:simple2} iteratively, 
\begin{align}\label{eq:bound_4_Hor_Mihlin}
D^{\beta}_{\xi} \left(\psi_\alpha(\|\xi\|^2)\right) &= \partial^{m_2}_{\xi_{j_2}} \dots  \partial^{m_{\gamma_0}}_{\xi_{j_{\gamma_0}}} \left(\sum_{p=0}^{\lfloor \frac{m_1}{2} \rfloor} C_p(m_1) \left(\xi_{j_1}\right)^{m_1-2p} \psi_{\alpha,r}^{(m_1-p)}\left(\|\xi\|^2\right)\right) ,\nonumber \\
&= \sum_{p_1=0}^{\lfloor \frac{m_1}{2} \rfloor} \dots  \sum_{p_{\gamma_0}=0}^{\lfloor \frac{m_{\gamma_0}}{2} \rfloor} C_{p_1}(m_1) \dots C_{p_{\gamma_0}}(m_{\gamma_0}) \left(\xi_{j_1}\right)^{m_1-2p_1} \dots \left(\xi_{j_{\gamma_0}}\right)^{m_{\gamma_0}-2p_{\gamma_0}} \psi_{\alpha,r}^{(|\beta|-(p_1+\dots+p_{\gamma_0}))}\left(\|\xi\|^2\right).
\end{align}
Now, for all $\xi \in \bbr^d$ such that $\xi \ne 0$, and all $p_1 ,\dots, p_{\gamma_0}$,
\begin{align}\label{ineq:bound_5_Hor_Mihlin}
\left| \left(\xi_{j_1}\right)^{m_1-2p_1} \dots \left(\xi_{j_{\gamma_0}}\right)^{m_{\gamma_0}-2p_{\gamma_0}} \psi_{\alpha,r}^{(|\beta|-(p_1+\dots+p_{\gamma_0}))}\left(\|\xi\|^2\right) \right| &\leq C_{\alpha,|\beta|,p_1,\dots,p_{\gamma_0},r} \|\xi\|^{|\beta|-2(p_1+\dots+p_{\gamma_0})} \nonumber\\
&\quad\quad\times \dfrac{\|\xi\|^{\alpha-1-2|\beta|+2(p_1+\dots+p_{\gamma_0})}}{\left(2+\|\xi\|^{\alpha}\right)^{\frac{r}{2}+|\beta|-(p_1+\dots+p_{\gamma_0})}}\nonumber \\
&\quad\quad\times \left(1+\dots +\|\xi\|^{\alpha \left(|\beta|-(p_1+\dots+p_{\gamma_0})\right)}\right),\nonumber\\
&\leq C_{\alpha,|\beta|,p_1,\dots,p_{\gamma_0},r} \dfrac{\|\xi\|^{\alpha-1}}{\left(2+\|\xi\|^{\alpha}\right)^{\frac{r}{2}+|\beta|-(p_1+\dots+p_{\gamma_0})}} \nonumber\\
&\quad\quad\times \|\xi\|^{-|\beta|}\left(1+\dots +\|\xi\|^{\alpha \left(|\beta|-(p_1+\dots+p_{\gamma_0})\right)}\right),\nonumber\\
&\leq \tilde{C}_{\alpha,|\beta|,p_1,\dots,p_{\gamma_0},r} \|\xi\|^{-|\beta|}.
\end{align}
Combining \eqref{eq:bound_4_Hor_Mihlin} and \eqref{ineq:bound_5_Hor_Mihlin}, it follows that 
for all $\xi \in \bbr^d$, 
$\xi \ne 0$,
\begin{align*}
\left|D^{\beta}_{\xi} \left(\psi_{\alpha,r}(\|\xi\|^2)\right)\right| \leq \tilde{C}_{\alpha,|\beta|,\gamma_0,r} \|\xi\|^{-|\beta|},
\end{align*}
for some $\tilde{C}_{\alpha,|\beta|,\gamma_0,r}>0$ only depending on $\alpha$, $|\beta|$, $\gamma$ and $r$. Then, \cite[Theorem $5.2.7.$]{G08} implies that, for all $p\in (1,+\infty)$, $S_{\alpha,r}$ is a bounded linear operator from $L^p(\bbr^d, dx)$ to $L^p(\bbr^d, dx)$ as soon as $r \geq 2(\alpha-1)/\alpha$. The end of the proof follows, 
in a straightforward manner, using the representation \eqref{eq:rep_sphe}.
\end{proof}

Before moving to generalizations and sharper versions of the above 
result, a few words are in order.~First, note that the bounds obtained on the quantities $\big|D^{\beta}_{\xi} \left(m_{\alpha,r}(\xi)\right)\big|$, 
for $\beta \in \mathbb{N}_0^d$ (the set of vectors of $\bbr^d$ which coordinates are non-negative integers) with $|\beta| \leq \lfloor \frac{d}{2} \rfloor+1$, depend 
on the dimension. Thus, the corresponding bound on the operator norm of $S_{\alpha,r}$ does not seem to be sharp. 
On the other hand, in \cite{AC94}, for a certain class of radial multipliers, dimension-free estimates are obtained for the $L^p(\bbr^d,dx)$-boundedness of the corresponding convolution operators (see \cite[Theorem 1]{AC94}). 
The proof there rests upon transference methods (see \cite{GW77} for a standard reference).

In order to get a more precise (dimension-free) estimate on the constant $C_{\alpha, d, p, r}$ 
of Lemma~\ref{lem:Hor_Mihlin} (and to generalize this result to other symmetric $\alpha$-stable L\'evy measures), 
let us next combine Proposition~\ref{prop:Riesz_Representation} together with 
transference methods and one-dimensional multiplier theory to get:   

\begin{thm}\label{prop:Riesz_sharp}
Let $\alpha \in (1,2)$, let $r \in [r_{\min}(\alpha),+\infty)$, let $\nu_\alpha$ be a non-degenerate symmetric L\'evy measure on $\bbr^d$ satisfying \eqref{eq:scale} , let $\mu_\alpha$ be the associated $\alpha$-stable probability measure on $\bbr^d$ given by \eqref{def:stable}, let ${\mathbf D}^{\alpha-1}$ be defined by \eqref{eq:FracGradWB} and let $\mathcal{A}_\alpha$ be defined by \eqref{eq:Stheatgen}. Then, for all $p\in (1,+\infty)$ and all $f\in L^p(\bbr^d,dx)$,
\begin{align}\label{ineq:continuity_Riesz}
\| \|{\mathbf D}^{\alpha-1} \circ \left(E-\mathcal{A}_\alpha\right)^{-\frac{r}{2}}(f)\| \|_{L^p(\bbr^d,dx)} \leq C_{\alpha,p,r} \left(\int_{\bbr^d} \|y\| \mu_\alpha(dy)\right) \|f\|_{L^p(\bbr^d, dx)},
\end{align}
where $C_{\alpha,p,r}$ is given by
\begin{align}\label{ineq:const_Riesz}
C_{\alpha,p,r} = \frac{C}{2\Gamma(\frac{r}{2})} \max \left(p, (p-1)^{-1}\right) \max \left( (1-\alpha(1-\frac{r}{2}))\Gamma\left(\frac{1}{\alpha}-(1-\frac{r}{2})\right) +\alpha \Gamma\left(\frac{r}{2}+\frac{1}{\alpha}\right), \|\rho_{\alpha,r}\|_{\infty,\bbr}\right),
\end{align}
with $\rho_{\alpha,r}$ defined, for all $\xi \in \bbr$, by
\begin{align*}
\rho_{\alpha,r}(\xi) = 2i \alpha  \int_{0}^{+\infty}  \dfrac{e^{-t^\alpha}}{t^{\alpha(1-\frac{r}{2})}} \sin\left(\xi t\right)dt,
\end{align*}
and with $C>0$ a numerical constant independent of $\alpha$, $p$, $d$ and $r$. 
\end{thm}

\begin{proof}
First, for $y \in \bbr^d$ fixed,  introduce the operator $T^y_{\alpha,r}$ defined, 
for all $f \in \mathcal{S}(\bbr^d)$ and all $x\in \bbr^d$, 
by 
\begin{align*}
T^y_{\alpha,r}(f)(x):= \int_0^{+\infty} \dfrac{e^{-t}}{t^{2-\frac{r}{2}-\frac{1}{\alpha}}} \left(f(x+t^{\frac{1}{\alpha}}y)-f(x-t^{\frac{1}{\alpha}}y)\right) dt.  
\end{align*}
Then, from Proposition~\ref{prop:Riesz_Representation}, for all $r>2(\alpha-1)/\alpha$, all $f \in \mathcal{S}(\bbr^d)$ and all $x\in \bbr^d$,
\begin{align}\label{eq:Riesz_Int_T_1}
{\mathbf D^{\alpha-1}} \circ \left(E-\mathcal{A}_\alpha\right)^{-\frac{r}{2}}(f)(x) =\frac{1}{2\Gamma(\frac{r}{2})}\int_{\bbr^d} y T^y_{\alpha,r}(f)(x) \mu_\alpha(dy).
\end{align}
Moreover,  in the singular case $r=r_{\min}(\alpha)$,
\begin{align}\label{eq:Riesz_Int_T_2}
{\mathbf D^{\alpha-1}} \circ \left(E-\mathcal{A}_\alpha\right)^{-\frac{r_{\min}(\alpha)}{2}}(f) = \frac{1}{2\Gamma(\frac{r_{\min}(\alpha)}{2})}\int_{\bbr^d} y T^y_{\alpha,r_{\min}(\alpha)}(f) \mu_\alpha(dy),
\end{align}
(which can easily be seen by a truncation argument).~Now, let $y$ be fixed in 
 $\bbr^d$, $y \ne 0$, and let $(\psi_u^y)_{u \in \bbr}$ be the continuous family of maps from $ \bbr^d$ to $\bbr^d$ defined, for all $x \in \bbr^d$ and all $u \in \bbr$, by
\begin{align*}
\psi_u^y(x) = x+u y.
\end{align*}
Note that, for $u,v \in \bbr$, $\psi^y_{u+v}(x,y) = \psi^y_u \circ \psi^y_v(x)$, that $\psi_0^y(x)=x$ and that $\psi^{y}_u$ is measure preserving in that, for all $p \in (1,+\infty)$, all $f \in L^p(\bbr^d, dx)$ and all $u \in \bbr$,
\begin{align*}
\int_{\bbr^d} |f\circ \psi^y_{u}(x)|^pdx = \int_{\bbr^d} |f(x+u y)|^pdx = \int_{\bbr^d} |f(x)|^pdx,
\end{align*}
from the translation invariance of the Lebesgue measure. 
Associated with this continuous family of maps, introduce the continuous family of linear operators, $(R_u^y)_{u\in \bbr}$, defined, for all $f \in L^p(\bbr^d,dx)$ and all $u\in \bbr$, by
\begin{align*}
R_u^y (f) =f \circ \psi_u^y.
\end{align*}
Note that, for all $f,g \in \mathcal{S}(\bbr^d)$ and all $u \in \bbr$,
\begin{align*}
\int_{\bbr^d} R_u^y (f)(x) g(x) dx = \int_{\bbr^d} f(x) R_{-u}^y(g)(x) dx,
\end{align*}
so that the operator $R_{-u}^y$ is the dual of the operator $R_u^y$ in $L^2(\bbr^d,dx)$. Now, by the change of variables $u = t^{{1/\alpha}}$, for all $f \in \mathcal{S}(\bbr^d)$ and all $x\in \bbr^d$, 
\begin{align}\label{eq:Hilbert_TR_Rep}
T^y_{\alpha,r}(f)(x) = \alpha \int_0^{+\infty} \dfrac{e^{-u^\alpha}}{u^{\alpha (1 - \frac{r}{2})}} \left(R^y_u(f)(x)-R_{-u}^y(f)(x)\right)du.
\end{align}
Next, let us further introduce the one-dimensional integral operator $\mathcal{H}_{\alpha,r}$ defined, for all $f \in \mathcal{S}(\bbr)$ and all $v \in \bbr$, by
\begin{align}\label{eq:def_Halpha}
\mathcal{H}_{\alpha,r}(f)(v) = \alpha \int_{0}^{+\infty}  \dfrac{e^{-u^\alpha}}{u^{\alpha(1- \frac{r}{2})}} \left(f(v+u)-f(v-u)\right) du.
\end{align}
Note that when $r = r_{\min}(\alpha)= 2(\alpha-1)/\alpha$, the operator $\mathcal{H}_{\alpha,r_{\min}(\alpha)}$ is just a tempered version of the classical Hilbert transform on $\bbr$. From \eqref{eq:Hilbert_TR_Rep} and \eqref{eq:def_Halpha}, it follows that 
for all $f \in \mathcal{S}(\bbr^d)$ and all $x\in \bbr^d$,
\begin{align}\label{eq:Transf1}
T^y_{\alpha,r}(f)(x)& = \mathcal{H}_{\alpha,r}(R^y(f)(x))(0).
\end{align}
Moreover, since $y$ is a non-null vector of $\bbr^d$, every $x \in \bbr^d$ can be written as $x = v_x e_y + w_x$ where $e_y={y/\|y\|}$, where $v_x$ is a scalar in $\bbr$ and where 
$w_x$ belongs to the orthogonal complement of the linear span of $e_y$.  
Thus, thanks to \eqref{eq:Transf1}, for all $f \in \mathcal{S}(\bbr^d)$ and all $x \in \bbr^d$, 
\begin{align}\label{eq:Transf2}
T^y_{\alpha,r}(f)(x)& = \mathcal{H}_{\alpha,r}(R^y(f)(x))(0),\nonumber \\
& = \mathcal{H}_{\alpha,r}(R^y(f)(v_x e_y + w_x))(0),\nonumber\\
& = \mathcal{H}_{\alpha,r}(R^y(f)(w_x))\left(\frac{v_x}{\|y\|}\right).
\end{align}
Assume now that the following bound (to be shown) holds true: for all $p \in (1,+\infty)$ and all $f \in L^p(\bbr,dx)$,
\begin{align}\label{ineq:Lp_Halpha}
\left\|\mathcal{H}_{\alpha,r}(f)\right\|_{L^p(\bbr,dx)} \leq \tilde{C}_{\alpha,p,r} \|f\|_{L^p(\bbr,dx)},
\end{align}
for some positive constant $\tilde{C}_{\alpha,p,r}$ depending on $\alpha$, on $p$ and on $r$. Then, by transference, for all $p \in (1,+\infty)$ and all $f \in \mathcal{S}(\bbr^d)$,
\begin{align*}
\int_{\bbr^d} |T^y_{\alpha,r}(f)(x)|^p dx &= \int_{\bbr^{d-1}}  \left(\int_{\bbr} |\mathcal{H}_{\alpha,r}(R^y(f)(w))\left(\frac{v}{\|y\|}\right)|^p dv\right) dw,\\
&=\|y\| \int_{\bbr^{d-1}}  \left(\int_{\bbr} |\mathcal{H}_{\alpha,r}(R^y(f)(w))\left(v\right)|^p dv\right) dw,\\
& = \|y\| \left\| \|\mathcal{H}_{\alpha,r}(R^y(f)(w))\|_{L^p(\bbr, dv)}\right\|^p_{L^p(\bbr^{d-1}, dw)},\\
& \leq \|y\| \tilde{C}_{\alpha,p,r}^p \| \|R^y(f)(w)\|_{L^p(\bbr, dv)}\|^p_{L^p(\bbr^{d-1}, dw)} ,\\
& \leq \|y\| \tilde{C}_{\alpha,p,r}^p \int_{\bbr^{d-1}} \left( \int_{\bbr} |R_v^y(f)(w)|^p dv\right) dw , \\
& \leq \|y\| \tilde{C}_{\alpha,p,r}^p \int_{\bbr^{d-1}} \left( \int_{\bbr} |f(vy+w)|^p dv\right) dw ,\\
&\leq \tilde{C}_{\alpha,p,r}^p \int_{\bbr^{d-1}} \left( \int_{\bbr} |f(ve_y+w)|^p dv\right) dw.
\end{align*}
Thus, for all $p \in (1,+\infty)$ and all $f \in \mathcal{S}(\bbr^d)$,
\begin{align}\label{ineq:Lp_T}
\|T^y_{\alpha,r}(f)\|_{L^p(\bbr^d,dx)} \leq \tilde{C}_{\alpha,p,r} \|f\|_{L^p(\bbr^d,dx)}.
\end{align}
Combining \eqref{eq:Riesz_Int_T_1} (or \eqref{eq:Riesz_Int_T_2}) and \eqref{ineq:Lp_T} together with Minkowski's integral inequality leads to:
\begin{align*}
\| \|{\mathbf D^{\alpha-1}} \circ \left(E-\mathcal{A}_\alpha\right)^{-\frac{r}{2}}(f)\| \|_{L^p(\bbr^d,dx)} &= \frac{1}{2\Gamma(\frac{r}{2})} \left\| \left\|\int_{\bbr^d} y T^y_{\alpha,r}(f) \mu_\alpha(dy) \right\| \right\|_{L^p(\bbr^d,dx)}, \\
&\leq \frac{\tilde{C}_{\alpha,p,r}}{2\Gamma(\frac{r}{2})} \left( \int_{\bbr^d} \|y\| \mu_\alpha(dy) \right) \|f\|_{L^p(\bbr^d, dx)}.
\end{align*}
To conclude the proof of the theorem, let us prove the inequality \eqref{ineq:Lp_Halpha}. First, by Fourier inversion, for all $f \in \mathcal{S}(\bbr)$ and all $x\in \bbr$,
\begin{align*}
\mathcal{H}_{\alpha,r}(f)(x) &=  \alpha \int_{0}^{+\infty}  \dfrac{e^{-u^\alpha}}{u^{\alpha (1-\frac{r}{2})}} \left(f(x+u)-f(x-u)\right) du,\\
&= \int_{0}^{+\infty}  \dfrac{e^{-t}}{t^{2-\frac{r}{2}-\frac{1}{\alpha}}} \left(f(x+t^{\frac{1}{\alpha}})-f(x-t^{\frac{1}{\alpha}})\right) dt,\\
&= \frac{1}{(2\pi)} \int_{\bbr} \mathcal{F}(f)(\xi) e^{i x \xi} \left(\int_{0}^{+\infty}  \dfrac{e^{-t}}{t^{2-\frac{r}{2}-\frac{1}{\alpha}}} \left(e^{i \xi t^{\frac{1}{\alpha}}}-e^{-i \xi t^{\frac{1}{\alpha}}}\right)dt \right) d\xi,\\
&= \frac{1}{(2\pi)} \int_{\bbr} \mathcal{F}(f)(\xi) e^{i x \xi} \rho_{\alpha,r}(\xi) d\xi,
\end{align*}
where, for all $\xi \in \bbr$,
\begin{align*}
\rho_{\alpha,r}(\xi) := 2i \int_{0}^{+\infty}  \dfrac{e^{-t}}{t^{2-\frac{r}{2}-\frac{1}{\alpha}}} \sin\left(\xi t^{\frac{1}{\alpha}}\right)dt = 2i \alpha  \int_{0}^{+\infty}  \dfrac{e^{-t^\alpha}}{t^{\alpha(1-\frac{r}{2})}} \sin\left(\xi t\right)dt.
\end{align*}
Let us precisely bound the symbol $\rho_{\alpha,r}$ (and its first derivative) by taking into account 
the regularization effect induced by the sine function.  
First, observe that $\rho_{\alpha,r}(0)=0$ and that 
$\lim_{|\xi| \rightarrow +\infty} \rho_{\alpha,r}(\xi) = 0$ (by the Riemann-Lebesgue lemma), for all $r \in (r_{\min}(\alpha),+\infty)$, and that, the function $\rho_{\alpha,r}$ is continuous on $\bbr$, so that $\|\rho_{\alpha,r}\|_{\infty,\bbr} <+\infty$, for all $r\in (r_{\min}(\alpha),+\infty)$. Moreover, by a Fubini argument, for all $\xi >0$ and all $r \geq r_{\min}(\alpha)$,
\begin{align*}
\rho_{\alpha,r}(\xi) &= 2 \alpha i \int_{0}^{+\infty}  \dfrac{e^{-t^\alpha}}{t^{\alpha(1- \frac{r}{2})}} \sin\left(\xi t\right)dt, \\
& = 2i \alpha \int_0^\xi \left(\int_0^{+\infty} e^{-t^\alpha} t^{1- \alpha(1-\frac{r}{2})} \cos(tu) dt\right) du, \\
& = 2 \Gamma\left(\frac{2}{\alpha}-1+\frac{r}{2}\right) i \int_0^\xi \left(\int_0^{+\infty} f_{\alpha,r}(t) \cos(tu) dt\right) du,
\end{align*}
where $f_{\alpha,r}$ is the probability density function defined, for all $t>0$, by
\begin{align*}
f_{\alpha,r}(t) = \frac{ \alpha t^{1- \alpha(1-\frac{r}{2})}}{\Gamma\left(\frac{2}{\alpha}+\frac{r}{2}-1\right)} e^{-t^\alpha}.
\end{align*}
The function $f_{\alpha,r}$ is differentiable on $(0,+\infty)$ and its first derivative, $f_{\alpha,r}^\prime$, 
is integrable on $(0,+\infty)$ so that $\varphi_{\alpha,r}$ defined, for all $u \in (0,+\infty)$, by
\begin{align*}
\varphi_{\alpha,r}(u) = \int_0^{+\infty} f_{\alpha,r}(t) \cos(tu) dt,
\end{align*}
decays at least as $1/u$ as $u$ tends to $+\infty$. In fact, two integrations by parts and since $\alpha \in (1,2)$ prove that $\|u^2 \varphi_{\alpha,r_{\min}(\alpha)} \|_{\infty, \bbr}<+\infty$. Then, $\|\rho_{\alpha,r}\|_{\infty, \bbr}<+\infty$, for all $r \geq r_{\min}(\alpha)$.

Moreover, by straightforward computations,
\begin{align*}
\int_{0}^{+\infty} |f^{\prime}_{\alpha,r}(t)| dt \leq \dfrac{(1-\alpha(1-\frac{r}{2}))\Gamma\left(\frac{1}{\alpha}-(1-\frac{r}{2})\right)}{\Gamma\left(\frac{2}{\alpha}+\frac{r}{2}-1\right)} + \alpha \frac{\Gamma(\frac{1}{\alpha}+\frac{r}{2})}{\Gamma(\frac{2}{\alpha}+\frac{r}{2}-1)}.
\end{align*}
Then, $\rho_{\alpha,r}$ is differentiable on $\bbr$ and, for all $\xi>0$ (and similarly for all $\xi<0$),
\begin{align*}
\rho^{\prime}_{\alpha,r}(\xi) = 2 \Gamma\left(\frac{2}{\alpha}+\frac{r}{2}-1\right) i \varphi_{\alpha,r}(\xi) = -2 \Gamma\left(\frac{2}{\alpha}+\frac{r}{2}-1\right) i \frac{1}{\xi} \int_{0}^{+\infty} f^{\prime}_{\alpha,r}(t) \sin\left(t \xi\right) dt,
\end{align*}
where an integration by parts has been used in the last equality. Thus, for all $\xi \in \bbr$ with $\xi \ne 0$,
\begin{align}\label{ineq:decay_deriv}
\left|\rho^{\prime}_{\alpha,r}(\xi)\right| \leq 2\left( (1-\alpha(1-\frac{r}{2}))\Gamma\left(\frac{1}{\alpha}-(1-\frac{r}{2})\right)+\alpha \Gamma\left(\frac{r}{2}+\frac{1}{\alpha}\right)\right) \frac{1}{|\xi|}.
\end{align} 
Finally, \cite[Theorem $5.2.7$]{G08} combined with \eqref{ineq:decay_deriv} concludes the proof of the 
theorem. 
\end{proof}

Before moving to a reverse inequality, let us comment a bit on the bound just obtained.  
Taking $\nu_{\alpha}^{\rm rot}$ (given by \eqref{eq:Levy_Rot}) and $\alpha=2$, 
Theorem~\ref{prop:Riesz_sharp} suggests that, for all $r \geq 1$, all $p \in (1,+\infty)$ and all $f\in L^p(\bbr^d,dx)$, 
\begin{align}\label{ineq:continuity_Riesz_Gaussian}
\| \| \nabla \circ \left(E-\mathcal{A}_H\right)^{-\frac{r}{2}}(f)\| \|_{L^p(\bbr^d,dx)} \leq C_{2,p,r} \left(\int_{\bbr^d} \|y\| \gamma(dy)\right) \|f\|_{L^p(\bbr^d, dx)},
\end{align}
with $C_{2,p,r}$ given by
\begin{align}\label{eq:constant_Gaussian}
C_{2,p,r} = \frac{C}{2\Gamma(\frac{r}{2})} \max \left(p, (p-1)^{-1}\right) \max \left( (r-1)\Gamma\left(\frac{r-1}{2}\right) + 2 \Gamma\left(\frac{r+1}{2}\right) , \|\rho_{2,r}\|_{\infty} \right).
\end{align}
for some $C>0$ independent of $p$, $d$ and $r$.  Above, the pre-factor 
$\int_{\bbr^d} \|y\| \gamma(dy)$ depends on the dimension $d$ of the ambient space $\bbr^d$ and 
so \eqref{ineq:continuity_Riesz_Gaussian} is not optimal.  
Actually, in this case and as explained 
next, it is possible to refine the argument of the proof of Theorem~\ref{prop:Riesz_sharp}, by using the independence of the coordinates of a random vector with law $\gamma$, in order to reach dimension-free estimate.  
Indeed, by duality, for all $f\in \mathcal{S}(\bbr^d)$ and all $x\in \bbr^d$,
\begin{align*}
\| \nabla \circ \left(E-\mathcal{A}_H\right)^{-\frac{r}{2}}(f)(x)\| & = \frac{1}{2 \Gamma(\frac{r}{2})} \underset{z \in \bbr^d, \|z\| = 1}{\sup} \left|\langle z; \int_{\bbr^d} y T^y_{2,r}(f)(x) \gamma(dy)  \rangle \right| , \\
& = \frac{1}{2 \Gamma(\frac{r}{2})} \underset{z \in \bbr^d, \|z\| = 1}{\sup} \left| \int_{\bbr^d} \langle z; y \rangle T^y_{2,r}(f)(x) \gamma(dy)\right|.  
\end{align*}
Now, by H\"older's inequality and since, under $\gamma$, $\langle z ; y\rangle $ is a centered Gaussian random variable with variance $\sum_{k=1}^d z_k^2=1$, 
\begin{align*}
\| \nabla \circ \left(E-\mathcal{A}_H\right)^{-\frac{r}{2}}(f)(x)\| & \leq \frac{\gamma_2(q)}{2 \Gamma(\frac{r}{2})} \left(\int_{\bbr^d} |T^y_{2,r}(f)(x)|^p \gamma(dy)\right)^{\frac{1}{p}},
\end{align*}
with $q=p/(p-1)$ and 
\begin{align}\label{eq:gamma_2_gauss}
\gamma_2(q) := \left(\int_{\bbr} |t|^q e^{- \frac{|t|^2}{2}} \frac{dt}{(2\pi)^{\frac{1}{2}}}\right)^{\frac{1}{q}}.
\end{align}
Then, \eqref{ineq:continuity_Riesz_Gaussian} with $r = 1$ improves to  
\begin{align}\label{ineq:continuity_Riesz_Gaussian_dimension_free}
\| \| \nabla \circ \left(E-\mathcal{A}_H\right)^{-\frac{1}{2}}(f)\| \|_{L^p(\bbr^d,dx)} \leq \tilde{C} \max \left(p, (p-1)^{-1}\right)  \gamma_2(q) \|f\|_{L^p(\bbr^d, dx)}.
\end{align}
Recalling that $\gamma_2(q) \sim \sqrt{\frac{q}{e}}$, as $q\to +\infty$, one sees that, above, 
the dependency on $p$ (as $p\to 1$) is comparable to the one obtained in \cite{P88}.  
An analogous reasoning can be applied to the general stable framework of Theorem~\ref{prop:Riesz_sharp}.  
Indeed, once again, by duality and H\"older's inequality, for all $r \in [r_{\min}(\alpha); +\infty)$, all $p \in (\alpha/(\alpha-1),+\infty)$ and all $f \in \mathcal{S}(\bbr^d)$, 
\begin{align}\label{ineq:continuity_Riesz_Sharp}
\| \|{\mathbf D^{\alpha-1}} \circ \left(E-\mathcal{A}_\alpha\right)^{-\frac{r}{2}}(f)\| \|_{L^p(\bbr^d,dx)} \leq C_{\alpha,p,r} \underset{z \in \bbr^d, \|z\| = 1}{\sup}\left(\int_{\bbr^d} |\langle y ; z \rangle|^q \mu_\alpha(dy)\right)^{\frac{1}{q}} \|f\|_{L^p(\bbr^d, dx)},
\end{align}
$q=p/(p-1)$.  Let us discuss briefly how to estimate the quantity 
$\gamma_\alpha(q)$ defined, for all $q \in (1, \alpha)$, by
\begin{align}\label{eq:def_gamma_alpha}
\gamma_\alpha(q) := \underset{z \in \bbr^d, \|z\| = 1}{\sup}\left(\int_{\bbr^d} |\langle y ; z \rangle|^q \mu_\alpha(dy)\right)^{{1/q}},
\end{align}
based on assumptions on the corresponding L\'evy measure $\nu_\alpha$. For all $z \in \mathbb{S}^{d-1}$, under $\mu_\alpha$, the random variable $\langle z;y\rangle$ has an $\alpha$-stable distribution with characteristic function given (from \eqref{eq:rep_spectral_measure}), for all $\xi \in \bbr$, by
\begin{align*}
\int_{\bbr^d} e^{i \xi \langle z;y\rangle} \mu_\alpha(dy) = \exp \left(-|\xi|^\alpha \int_{\mathbb{S}^{d-1}} |\langle z; x \rangle|^\alpha \lambda_1(dx)\right).  
\end{align*}
In other words, under $\mu_\alpha$, $\langle z;y\rangle$ is an $\alpha$-stable symmetric random variable on $\bbr$ with scale parameter depending on $z$. Thus, denoting by $Y_\alpha$ a random variable with characteristic function given, for all $\xi \in \bbr$, by $\bbe \exp (i \xi Y_\alpha) = \exp \left( - |\xi|^\alpha\right)$, one has,
\begin{align*}
\int_{\bbr^d} e^{i \xi \langle z;y\rangle} \mu_\alpha(dy) = \bbe e^{i \xi \sigma_\alpha(z) Y_\alpha},
\end{align*}
with $\sigma_\alpha(z)$ defined, for all $z\in \mathbb{S}^{d-1}$, by $ \sigma_\alpha(z) = \left( \int_{\mathbb{S}^{d-1}} |\langle z; x \rangle|^\alpha \lambda_1(dx) \right)^{{1/\alpha}}$.  Then, for all $q \in (1, \alpha)$, 
\begin{align}\label{eq:rep_gamma_alpha}
\gamma_\alpha(q) := \left(\bbe |Y_\alpha|^q\right)^{{1/q}} \underset{z \in \bbr^d, \|z\| = 1}{\sup}\sigma_\alpha(z). 
\end{align}

{\it In particular, for the rotationally invariant $\alpha$-stable probability measure on $\bbr^d$,  \eqref{ineq:continuity_Riesz_Sharp} improves into the following dimension-free estimate:} 

\begin{cor}\label{cor:dimension_free_euclidean}
Let $\alpha \in (1,2)$, let $r \in [r_{\min}(\alpha),+\infty)$, let $\nu_{\alpha}^{\rm rot}$ be the L\'evy measure on $\bbr^d$ defined by \eqref{eq:Levy_Rot}, let $\mu_{\alpha}^{\rm rot}$ be the associated rotationally invariant $\alpha$-stable probability measure on $\bbr^d$ given by \eqref{eq:charac_rot}, let $\mathbf{D}^{\alpha-1,\operatorname{rot}}$ be defined by \eqref{eq:FracGradWB} with $\nu_{\alpha}^{\rm rot}$ and let $\mathcal{A}_\alpha^{\operatorname{rot}}$ be defined by \eqref{eq:Stheatgen} with $\nu_{\alpha}^{\rm rot}$. Then, for all $p\in (\alpha/(\alpha-1),+\infty)$ and all $f\in L^p(\bbr^d,dx)$,
\begin{align}\label{ineq:continuity_Riesz_Sharp_2}
\| \| \mathbf{D}^{\alpha-1,\operatorname{rot}} \circ \left(E-\mathcal{A}_\alpha^{\operatorname{rot}}\right)^{-\frac{r}{2}}(f)\| \|_{L^p(\bbr^d,dx)} \leq \frac{C_{\alpha,p,r}}{2^{{1/\alpha}}} \left( \bbe |Y_\alpha|^{q}\right)^{{1/q}} \|f\|_{L^p(\bbr^d, dx)},
\end{align}
with $q = p/(p-1)$, $C_{\alpha,p,r}$ given by \eqref{ineq:const_Riesz} and $Y_\alpha$ is a real-valued random variable with characteristic function given, for all $\xi \in \bbr$, by $\bbe \exp(i Y_\alpha \xi) = \exp \left( - |\xi|^\alpha\right)$.
\end{cor}
\noindent
For the probability measure $\mu_{\alpha,d}$ defined by \eqref{eq:StableIndAxes}, the previous analysis provides the following result.
\begin{cor}\label{cor:dimension_euclidean2}
Let $\alpha \in (1,2)$, let $r \in [r_{\min}(\alpha),+\infty)$, let $\nu_{\alpha,d}$ be the L\'evy measure on $\bbr^d$ defined by \eqref{eq:LevyIndAxes}, let $\mu_{\alpha,d}$ be the associated probability measure on $\bbr^d$ given by \eqref{eq:StableIndAxes}, let $\mathbf{D}^{\alpha-1,d}$ be defined by \eqref{eq:FracGradWB} with $\nu_{\alpha, d}$ and let $\mathcal{A}_{\alpha,d}$ be defined by \eqref{eq:Stheatgen} with $\nu_{\alpha, d}$. Then, for all $p\in (\alpha/(\alpha-1),+\infty)$ and all $f\in L^p(\bbr^d,dx)$,
\begin{align}\label{ineq:continuity_Riesz_Sharp_3}
\| \|\mathbf{D}^{\alpha-1,d} \circ \left(E-\mathcal{A}_{\alpha,d}\right)^{-\frac{r}{2}}(f)\| \|_{L^p(\bbr^d,dx)} \leq \frac{C_{\alpha,p,r}}{2^{\frac{1}{\alpha}}} \left( \bbe |Y_\alpha|^{q}\right)^{\frac{1}{q}} d^{\frac{1}{\alpha}-\frac{1}{2}} \|f\|_{L^p(\bbr^d, dx)},
\end{align}
with $q=p/(p-1)$, $C_{\alpha,p,r}$ given by \eqref{ineq:const_Riesz} and $Y_\alpha$ is a real-valued random variable with characteristic function given, for all $\xi \in \bbr$, by $\bbe \exp(i Y_\alpha \xi) = \exp \left( - |\xi|^\alpha\right)$.  
\end{cor}

\begin{proof}
First, let us recall the inequality \eqref{ineq:continuity_Riesz_Sharp}, for all $p \in (\alpha/(\alpha-1),+\infty)$ and all $f \in \mathcal{S}(\bbr^d)$,
\begin{align*}
\| \|\mathbf{D}^{\alpha-1,d} \circ \left(E-\mathcal{A}_{\alpha,d}\right)^{-\frac{r}{2}}(f)\| \|_{L^p(\bbr^d,dx)} \leq C_{\alpha,p,r} \underset{z \in \bbr^d, \|z\| = 1}{\sup}\left(\int_{\bbr^d} |\langle y ; z \rangle|^q \mu_{\alpha,d}(dy)\right)^{\frac{1}{q}} \|f\|_{L^p(\bbr^d, dx)}.
\end{align*}
Now, for all $\xi \in \bbr$,
\begin{align*}
\int_{\bbr^d} e^{i \xi \langle z;y\rangle} \mu_{\alpha,d}(dy) & = \prod_{k=1}^d \int_{\bbr} e^{i \xi z_k y_k}\mu_{\alpha,1}(dy_k) = \prod_{k=1}^d \exp \left( \int_{\bbr} \left(e^{i \xi z_k u}-1-i\xi z_k u\right) \nu_{\alpha,1}(du) \right) ,\\
&= \prod_{k=1}^d \exp \left( -\frac{|z_k \xi|^{\alpha}}{2}\right).
\end{align*}
Denoting by $Y_\alpha$ a random variable with characteristic function given, for all $\xi \in \bbr$, by $\bbe \exp (i \xi Y_\alpha) = \exp \left( - |\xi|^\alpha\right)$,
\begin{align*}
\int_{\bbr^d} e^{i \xi \langle z;y\rangle} \mu_{\alpha,d}(dy) & = \bbe \exp\left(i \xi \frac{\|z\|_\alpha}{2^{\frac{1}{\alpha}}} Y_\alpha\right),
\end{align*}
where $\|z\|^\alpha_\alpha = \sum_{k=1}^d |z_k|^\alpha$. Applying H\"older's inequality (since $\alpha \in (1,2)$),
\begin{align*}
\| \|\mathbf{D}^{\alpha-1,d} \circ \left(E-\mathcal{A}_{\alpha,d}\right)^{-\frac{r}{2}}(f)\| \|_{L^p(\bbr^d,dx)} \leq \frac{C_{\alpha,p,r}}{2^{\frac{1}{\alpha}}} d^{\frac{1}{\alpha}-\frac{1}{2}} \left(\bbe |Y_{\alpha}|^q\right)^{\frac{1}{q}} \|f\|_{L^p(\bbr^d, dx)},
\end{align*}
concluding the proof of the corollary.
\end{proof}
\noindent
As a straightforward by-product of our methodology, it is possible to obtain a sharp estimate for the $L^p(\bbr^d,dx)$-$L^p(\bbr^d,dx)$ boundedness of the linear operator with symbol $\tilde{m}_{\alpha,0,r_{\min}(\alpha)}$ (see Equation \eqref{eq:homo_frac_Riesz_transform}).~The corresponding one-dimensional operator appearing in the transference argument (see Lemma \ref{lem:double_truncation} of the Appendix) is (then) given, for all $f \in \mathcal{S}(\bbr)$ and all $x \in \bbr$, by
\begin{align*}
\mathcal{H}_\alpha(f)(x) & = \int_{0}^{+\infty} \left(f(x+t^{\frac{1}{\alpha}}) - f(x - t^{\frac{1}{\alpha}})\right) \frac{dt}{t} = \alpha \int_0^{+\infty} \left(f(x+t) - f(x - t)\right) \frac{dt}{t} , \\
& = \frac{\alpha}{2\pi} \int_{\bbr} \mathcal{F}(f)(\xi) e^{i x\xi} \left(i \pi \operatorname{sign}(\xi) \right) d\xi,
\end{align*}
which is (up to some constant) the classical Hilbert transform on the real line. From \cite[Theorem $4.1$ page $177$]{Pichorides_72}, for all $p \in (1, +\infty)$ and all $f \in L^p(\bbr,dx)$,
\begin{align*}
\| \mathcal{H}_\alpha(f) \|_{L^p(\bbr,dx)} \leq \alpha \pi C_p \|f\|_{L^p(\bbr,dx)},
\end{align*}
with $C_p$ given by
\begin{align}\label{eq:constant_optimal_hilbert}
C_p = \left\{
    \begin{array}{ll}
        \tan\left(\frac{\pi}{2p}\right) & \mbox{if}\, p \in (1,2], \\
        \cot\left(\frac{\pi}{2p}\right) & \mbox{if}\, p \in [2,+\infty),
    \end{array}
\right.
\end{align}
and this is the best possible constant. Thus, one has.
\begin{cor}\label{cor_fractional_homo_Riesz_transform} 
Let $\alpha \in (1,2)$, let $\nu_\alpha$ be a non-degenerate symmetric L\'evy measure on $\bbr^d$ satisfying \eqref{eq:scale} , let $\mu_\alpha$ be the associated $\alpha$-stable probability measure on $\bbr^d$ characterized by \eqref{def:stable}, let ${\mathbf D}^{\alpha-1}$ be defined by \eqref{eq:FracGradWB} and let $\mathcal{A}_\alpha$ be defined by \eqref{eq:Stheatgen}. Then, for all $p\in (\alpha/(\alpha-1),+\infty)$ and all $f\in L^p(\bbr^d,dx)$,
\begin{align}\label{ineq:fractional_homo_Riesz_transform}
\| \| \mathbf{D}^{\alpha-1} \circ \left(-\mathcal{A}_\alpha\right)^{-\frac{r_{\min}(\alpha)}{2}}(f)\| \|_{L^p(\bbr^d,dx)} \leq \frac{\alpha \pi \gamma_\alpha(q)}{2 \Gamma\left(1-\frac{1}{\alpha}\right)} C_p \|f\|_{L^p(\bbr^d,dx)},
\end{align}
where $C_{p}$ is given by \eqref{eq:constant_optimal_hilbert}, $\gamma_\alpha(q)$ by \eqref{eq:def_gamma_alpha} and $q=p/(p-1)$.
\end{cor}

\begin{rem}\label{rem:discussion_constant}
Let us analyze the constant obtained in \eqref{ineq:fractional_homo_Riesz_transform} in the rotationally invariant situation with $\alpha \rightarrow 2$.  Then, in this case, the constant boils down, for all $p \in [2, +\infty)$, to
\begin{align*}
\tilde{C}_p = \sqrt{\pi} \gamma_2(q) \cot\left(\frac{\pi}{2p}\right) = \sqrt{\pi} \left(\bbe |X_1|^q \right)^{\frac{1}{q}} \cot\left(\frac{\pi}{2p}\right),
\end{align*}
where $X_1$ is a standard Gaussian random variable. Moreover, recall that
\begin{align}\label{eq:moment_gauss} 
\bbe |X_1|^q = \dfrac{2^{\frac{q+1}{2}} \Gamma\left(\frac{q+1}{2}\right)}{\sqrt{2\pi}}, \quad q \geq 1.
\end{align}
According to our normalization, the operator on the left-hand side of inequality \eqref{ineq:fractional_homo_Riesz_transform} is given, (``symbolically") for all $f \in \mathcal{S}(\bbr^d)$ and all $x \in \bbr^d$, by
\begin{align*}
\mathbf{D}^{1,\operatorname{rot}} \circ \left(-\mathcal{A}^{\operatorname{rot}}_2\right)^{-\frac{1}{2}}(f)(x) = \sqrt{2} \nabla \circ \left(- \Delta\right)^{-\frac{1}{2}}(f)(x).
\end{align*}
According to \cite[pages $168$-$169$]{DV06} and to \cite[page 745]{W18}, to date the best constants for the vector classical Riesz transform (symbolically defined by $\nabla \circ \left(- \Delta\right)^{-\frac{1}{2}}$) are given, for all $p\in[2,+\infty)$, by 
\begin{align*}
\tilde{C}_{1,p} := \sqrt{2}H_p(1), \quad \tilde{C}_{2,p} :=2(p-1),
\end{align*} 
where $H_p(1) = \|R_1+iR_2\|_{p}$ and $R_1$, $R_2$ are planar Riesz transforms. The quantity $H_p(1)$ does not seem to have been computed in the literature but one can estimate it using the exact scalar result of \cite[Theorem $1.1$]{IM96} to obtain
\begin{align*}
\tilde{C}_{1,p} \leq 2 \sqrt{2} \cot\left(\frac{\pi}{2p}\right). 
\end{align*}
Note that, for values of $p$ close to $2$, $\tilde{C}_{2,p}$ is sharper than $2\sqrt{2}\cot(\pi/2p)$ and, for large values of $p$, $2\sqrt{2}\cot(\pi/2p)$ is smaller than $\tilde{C}_{2,p}$.
Our approach provides the following constant for the classical Riesz transform $\nabla \circ (-\Delta)^{-\frac{1}{2}}$: for all $p \in [2,+\infty)$, 
\begin{align*}
\tilde{C}_{3,p} : = \sqrt{\frac{\pi}{2}}\left(\bbe |X_1|^q \right)^{\frac{1}{q}} \cot\left(\frac{\pi}{2p}\right).
\end{align*}
Moreover,  one has the following asymptotic for $\tilde{C}_{3,p}$:
\begin{align*}
\underset{p\rightarrow 2}\lim \tilde{C}_{3,p} = \sqrt{\frac{\pi}{2}}, \quad \tilde{C}_{3,p} \underset{p \rightarrow +\infty}{\sim} \cot\left(\frac{\pi}{2p}\right).
\end{align*}
Both asymptotics are better than the ones induced by $2(p-1)$ and $2\sqrt{2}\cot(\pi/2p)$. Finally, by H\"older's inequality, for all $p\in [2,+\infty)$, 
\begin{align*}
\sqrt{\frac{2}{\pi}} \leq  \left( \bbe |X_1|^q \right)^{\frac{1}{q}} \leq 1,
\end{align*}
so that $\tilde{C}_{3,p} \leq \sqrt{\pi/2} \cot\left(\pi /2p\right)$.~The right-hand side of this last inequality is better than $2\sqrt{2} \cot(\pi/2p)$ and than $2(p-1)$, for all $p \in [2,+\infty)$. 
\end{rem}
\noindent
Moreover, by taking $f$ in the $L^p(\bbr^d,dx)$-domain of the unbounded operator $\left(E-\mathcal{A}_\alpha\right)^{{r/2}}$,  Theorem~\ref{prop:Riesz_sharp} provides the first-half of the problem of the equivalence of norms (relevant for $r = r_{\min}(\alpha)$) associated with the operators $(-\mathcal{A}_\alpha)^{{r/2}}$ and ${\mathbf D}^{\alpha-1}$.  
Indeed, for all $r \in [r_{\min}(\alpha);+\infty)$, all $p \in (1,+\infty)$ and all $f$ in the $L^p(\bbr^d,dx)$-domain of $\left(E-\mathcal{A}_\alpha\right)^{{r/2}}$,
\begin{align}\label{ineq:continuity_Riesz_2}
\| \|{\mathbf D}^{\alpha-1}(f)\| \|_{L^p(\bbr^d,dx)} \leq C_{\alpha,p,r} \left(\int_{\bbr^d} \|y\| \mu_\alpha(dy)\right) \|\left(E-\mathcal{A}_\alpha\right)^{{r/2}}(f)\|_{L^p(\bbr^d, dx)}.
\end{align}
\noindent
In order to study the reverse inequality in this fractional setting, let us investigate the continuity properties of a second fractional Riesz transform of order $1$ naturally linked to the fractional operator $\mathcal{A}_\alpha$.The main intuition behind these definitions of first order Riesz transforms comes from the following decomposition of the operator $\mathcal{A}_\alpha$: by an integration by parts in the radial coordinate, for all $f \in \mathcal{C}_c^\infty(\bbr^d)$ and all $x \in \bbr^d$, 
\begin{align*}
\mathcal{A}_\alpha(f)(x) & = \frac{1}{\alpha} \int_{\bbr^d} \langle \nabla(f)(x+u)-\nabla(f)(x) ; u \rangle \nu_\alpha(du) , \\
& = \frac{1}{\alpha} \sum_{k=1}^d D_k^{\alpha-1} \left(\partial_k (f)\right)(x) =  \frac{1}{\alpha} \sum_{k=1}^d \partial_k\left(D^{\alpha-1}_k(f)\right)(x). 
\end{align*}
Moreover, from the symmetry of the spherical component of $\nu_\alpha$, for all $f \in \mathcal{C}^\infty_c(\bbr^d)$ 
and all $x \in \bbr^d$,
\begin{align*}
\mathcal{A}_\alpha(f)(x) & = -  \frac{1}{\alpha} \sum_{k=1}^d (D_k^{\alpha-1})^* \left(\partial_k (f)\right)(x) = - \frac{1}{\alpha} \sum_{k=1}^d \partial_k\left((D^{\alpha-1}_k)^*(f)\right)(x).
\end{align*}
Then, it is natural to consider the family of operators defined, in the Fourier domain, for all $f \in \mathcal{S}(\bbr^d)$, all $\xi \in \bbr^d$, all $\lambda_0>0$ and all $r \geq 1/\alpha$, by
\begin{align}\label{def:multiplier_operator_classical_gradient}
\mathcal{F}\bigg(\nabla \circ \left(\lambda_0 E- \mathcal{A}_{\alpha}\right)^{-r}(f)\bigg)(\xi) = \dfrac{i \xi \mathcal{F}(f)(\xi)}{\left(\lambda_0+\int_{\mathbb{S}^{d-1}}|\langle y ; \xi\rangle|^\alpha\lambda_1(dy)\right)^{r}}.
\end{align} 
In particular,  when $\nu_\alpha = \nu_\alpha^{\operatorname{rot}}$, the previous multiplier operator boils down, for all $f \in \mathcal{S}(\bbr^d)$, all $\xi \in \bbr^d$, all $\lambda_0>0$ and all $r \geq 1/\alpha$, to
\begin{align*}
\mathcal{F}\bigg(\nabla \circ \left(\lambda_0 E- \mathcal{A}^{\operatorname{rot}}_{\alpha}\right)^{-r}(f)\bigg)(\xi) = \dfrac{i \xi \mathcal{F}(f)(\xi)}{\left(\lambda_0+\frac{\|\xi\|^\alpha}{2}\right)^{r}},
\end{align*}
with the chosen normalization. Taking $r=1/\alpha$ and letting $\lambda_0\rightarrow 0^+$, 
\begin{align*}
\mathcal{F}\bigg(\nabla \circ \left(- \mathcal{A}^{\operatorname{rot}}_{\alpha}\right)^{-\frac{1}{\alpha}}(f)\bigg)(\xi) = 2^{\frac{1}{\alpha}} \dfrac{i \xi}{\|\xi\|}\mathcal{F}(f)(\xi),
\end{align*}
which is proportional to the classical Riesz transform of order $1$. Then, one might expect dimension-free estimate for the operator defined in \eqref{def:multiplier_operator_classical_gradient} with $\nu_\alpha = \nu_\alpha^{\operatorname{rot}}$. 

In order to get the $L^p(\bbr^d,dx)$-boundedness (for $p \in (1,+\infty)$) of the operators $\nabla \circ \left(\lambda_0 E- \mathcal{A}_{\alpha}\right)^{-r}$, we adopt an approach similar to the one developed for the fractional operators $D^{\alpha-1} \circ \left(\lambda_0 E - \mathcal{A}_\alpha\right)^{-r/2}$. First, let us prove a Bismut-type formula which provides a stochastic representation for the action of the classical gradient operator on the fractional semigroup $\left(P_t^\alpha\right)_{t \geq 0}$.  Recall that $p_\alpha$, the Lebesgue density of a non-degenerate $\alpha$-stable probability measure with $\alpha \in (1,2)$, is positive on $\bbr^d$ (see, e.g., \cite[Lemma $2.1.$]{W07}).

\begin{lem}\label{lem:bismut_stable_2}
Let $\alpha \in (1,2)$, let $\nu_\alpha$ be a non-degenerate symmetric L\'evy measure on $\bbr^d$ satisfying \eqref{eq:scale} and let $\mu_\alpha$ be the corresponding centered $\alpha$-stable probability measure on $\bbr^d$ characterized by \eqref{def:stable}. Then, for all $f \in \mathcal{S}(\bbr^d)$, all $t >0$, all $x \in \bbr^d$ and all $j \in \{1, \dots, d\}$ 
\begin{align*}
\partial_j P^\alpha_t(f)(x) = \frac{1}{t^{\frac{1}{\alpha}}} \int_{\bbr^d} f(x + y t^{\frac{1}{\alpha}}) \dfrac{-\partial_j (p_\alpha)(y)}{p_\alpha(y)} \mu_\alpha(dy),
\end{align*}
where $p_\alpha$ is the Lebesgue density of the probability measure $\mu_\alpha$. 
\end{lem}

\begin{proof}
This follows by a simple integration by parts. 
\end{proof} 
\noindent
Note that the symmetry of the probability measure $\mu_\alpha$ ensures, for all $y\in \bbr^d$ and all $j \in \{1,  \dots,  d\}$, that
\begin{align*}
p_\alpha(y) = p_\alpha(-y), \quad \partial_j(p_\alpha)(y) = - \partial_j(p_\alpha)(-y). 
\end{align*}
As previously, one can infer the following representation for $\nabla \circ (\lambda_0 E - \mathcal{A}_\alpha)^{-r}$ based on the Bismut-type formula of Lemma \ref{lem:bismut_stable_2}: for all $f \in \mathcal{S}(\bbr^d)$, all $x \in \bbr^d$, all $\lambda_0>0$ and all $r \geq 1/\alpha$, 
\begin{align*}
\nabla \circ \left(\lambda_0 E - \mathcal{A}_\alpha\right)^{-r}(f)(x) = \frac{1}{2 \Gamma\left(r\right)} \int_{\bbr^d} F_{\alpha}(y) S_{\alpha,r}^{\lambda_0,y}(f)(x) \mu_\alpha(dy),
\end{align*}
with, 
\begin{align*}
F_{\alpha}(y) = \dfrac{- \nabla(p_\alpha)(y)}{p_\alpha(y)}, \quad S_{\alpha,r}^{\lambda_0,y}(f)(x) = \int_{0}^{+\infty} \frac{e^{-\lambda_0 t}}{t^{1-r+\frac{1}{\alpha}}} \left(f(x+t^{\frac{1}{\alpha}}y) - f(x-t^{\frac{1}{\alpha}}y) \right) dt.
\end{align*}
A transference argument similar to the one developed in the proof of Theorem \ref{prop:Riesz_sharp} allows to extend the operator $S_{\alpha,r}^{\lambda_0, y}$ from $\mathcal{S}(\bbr^d)$ to $L^p(\bbr^d,dx)$ (for $p \in (1,+\infty)$) with, for all $f \in L^p(\bbr^d,dx)$ and all $y \in \bbr^d$, 
\begin{align*}
\| S_{\alpha,r}^{\lambda_0, y}(f) \|_{L^p(\bbr^d,dx)} \leq C_{\alpha,p,r,\lambda_0} \| f \|_{L^p(\bbr^d,dx)},
\end{align*}
for some $C_{\alpha,p,r,\lambda_0}>0$ independent of the dimension $d$. Moreover, since the L\'evy measure verifies the non-degeneracy condition \eqref{eq:non_degeneracy}, it is well-known (see, e.g., \cite[Corollary page $340$]{BOG_86}, \cite[Theorem $1$]{BJP_00} or \cite[Theorem $2.1$]{KS13}) that 
\begin{align*}
\int_{\bbr^d} \| F_\alpha(y) \| \mu_\alpha(dy) = \int_{\bbr^d} \| \nabla(p_\alpha)(y) \| dy <+\infty.
\end{align*}
Thus, a straightforward application of Minkowski's integral inequality and transference argument ensure, for all $r \geq \frac{1}{\alpha}$, all $p\in (1,+\infty)$ and all $\lambda_0>0$, that
\begin{align}\label{ineq:Riesz_Transform_2_all_p}
\| \nabla \circ \left(\lambda_0 E - \mathcal{A}_\alpha\right)^{-r}(f) \|_{L^p(\bbr^d,dx)} \leq \frac{C_{\alpha,p,r,\lambda_0}}{2\Gamma(r)} \left(\int_{\bbr^d} \|F_\alpha(y)\| \mu_\alpha(dy) \right) \|f\|_{L^p(\bbr^d,dx)}.
\end{align}
Based on the previous estimate, we are now in position to prove the lower bound for the fractional Riesz transform $D^{\alpha-1} \circ ( E - \mathcal{A}_\alpha)^{-\frac{\alpha-1}{\alpha}}$ under non-degeneracy and symmetry of the  L\'evy measure $\nu_\alpha$. For the sake of completeness,  let us recall the very classical case of the heat semigroup. 

\begin{lem}\label{lem:duality}
Let $\gamma$ be the standard Gaussian measure on $\bbr^d$ given by \eqref{def:gauss} and let $\mathcal{A}_H$ be the operator defined by \eqref{eq:heatgen}. Then, for all $p \in (1,+\infty)$ and all $f \in \mathcal{C}_c^{\infty}(\bbr^d)$,
\begin{align*}
\|\left(E-\mathcal{A}_H\right)^{\frac{1}{2}}(f)\|_{L^p(\bbr^d,dx)} \leq \|f\|_{L^p(\bbr^d,dx)} + \tilde{C}\max\left(p/(p-1), p-1\right) \gamma_2(p) \| \| \nabla(f) \| \|_{L^p(\bbr^d,dx)},
\end{align*}
where $\tilde{C} > 0$ is a numerical constant independent of 
both $p$ and $d$, and where $\gamma_2(q)$ is given by \eqref{eq:gamma_2_gauss}. 
\end{lem}

\begin{proof}
Let $p \in (1,+\infty)$ and let $f \in \mathcal{C}_c^{\infty}(\bbr^d)$. By a Fourier argument and an integration by parts, for all $g \in \mathcal{S}(\bbr^d)$, 
\begin{align*}
\int_{\bbr^d} \left(E-\mathcal{A}_H\right)^{\frac{1}{2}}(f)(x)\left(E-\mathcal{A}_H\right)^{\frac{1}{2}}(g)(x) dx & = \langle \left(E-\mathcal{A}_H\right)^{\frac{1}{2}}(f); \left(E-\mathcal{A}_H \right)^{\frac{1}{2}}(g) \rangle_{L^2(\bbr^d,dx)} , \\ 
& = \langle f ; \left(E-\mathcal{A}_H\right)(g) \rangle_{L^2(\bbr^d,dx)} , \\
& = \langle f; g \rangle_{L^2(\bbr^d,dx)} - \langle f; \mathcal{A}_H(g) \rangle_{L^2(\bbr^d,dx)} , \\
& = \langle f; g \rangle_{L^2(\bbr^d,dx)} + \frac{1}{2} \int_{\bbr^d}\langle \nabla(f)(x) ;\nabla(g)(x) \rangle dx.
\end{align*}
Now, let $q=p/(p-1)$ and let $\mathcal{E}_{H}$ be a subset of $L^q(\bbr^d,dx)$, defined by
\begin{align*}
\mathcal{E}_H := \{\left(E-\mathcal{A}_H\right)^{\frac{1}{2}}(g),\quad g\in \mathcal{D}(\left(E-\mathcal{A}_H\right)^{\frac{1}{2}})\},
\end{align*}
where $\mathcal{D}(\left(E-\mathcal{A}_H\right)^{{1/2}})$ is the $L^q(\bbr^d,dx)$-domain of the unbounded operator $\left(E-\mathcal{A}_H\right)^{{1/2}}$. But, \cite[Lemma 1]{Russ} ensures that the set $\mathcal{E}_H$ is dense in $L^q(\bbr^d,dx)$. Then, by duality, 
\begin{align*}
\| \left(E-\mathcal{A}_H\right)^{\frac{1}{2}}(f) \|_{L^p(\bbr^d,dx)} =\underset{g^* \in \mathcal{E}_H,\, \|g^*\|_{L^q(\bbr^d,dx)}\leq 1}{\sup} \left|\langle \left(E-\mathcal{A}_H\right)^{\frac{1}{2}}(f);g^*\rangle \right|.
\end{align*}
Next, let $g^* \in \mathcal{E}_H$ with $\|g^*\|_{L^q(\bbr^d,dx)}\leq 1$ and let $g \in \mathcal{D}(\left(E-\mathcal{A}_H\right)^{{1/2}})$ be such that $\left(E-\mathcal{A}_H\right)^{{1/2}}(g)=g^*$. Hence, by H\"older's inequality, 
\begin{align*}
\left|\langle \left(E-\mathcal{A}_H\right)^{\frac{1}{2}}(f);g^*\rangle \right| &= \left| \langle f; g \rangle + \frac{1}{2} \int_{\bbr^d}\langle \nabla(f)(x) ;\nabla(g)(x) \rangle dx \right| ,\\
& \leq \|f\|_{L^p(\bbr^d,dx)}\|g\|_{L^q(\bbr^d,dx)} + \frac{1}{2} \| \| \nabla(f)\| \|_{L^p(\bbr^d,dx)} \| \| \nabla(g) \| \|_{L^q(\bbr^d,dx)}.
\end{align*}
Note that the gamma transform $(E- \mathcal{A}_H)^{-\frac{1}{2}}$ is a contraction on $L^q(\bbr^d,dx)$. Then, $\|g\|_{L^q(\bbr^d,dx)} \leq 1$. Finally, applying \eqref{ineq:continuity_Riesz_Gaussian_dimension_free},
\begin{align*}
\| \left(E-\mathcal{A}_H\right)^{\frac{1}{2}}(f) \|_{L^p(\bbr^d,dx)}  \leq \|f\|_{L^p(\bbr^d,dx)} + \tilde{C}\max\left(q, (q-1)^{-1}\right) \gamma_2(p)\| \| \nabla(f) \| \|_{L^p(\bbr^d,dx)},
\end{align*}
which concludes the proof of the lemma. 
\end{proof}
\noindent
The next result is the fractional analog of the previous lemma. 

\begin{prop}\label{prop:lower_bound_Riesz_stable_order_1}
Let $\alpha \in (1,2)$, let $\nu_\alpha$ be a non-degenerate symmetric L\'evy measure on $\bbr^d$ satisfying \eqref{eq:scale}, let $\mu_\alpha$ be the associated $\alpha$-stable 
probability measure on $\bbr^d$ characterized by \eqref{def:stable}, let $D^{\alpha-1}$ be defined by \eqref{eq:FracGrad} and let $\mathcal{A}_\alpha$ be defined by \eqref{eq:Stheatgen}.~Then, for all $p \in (1,+\infty)$ and all $f \in \mathcal{C}^{\infty}_c(\bbr^d)$, 
\begin{align*}
\|\left(E-\mathcal{A}_\alpha\right)^{\frac{\alpha-1}{\alpha}}(f)\|_{L^p(\bbr^d,dx)} \leq \|f\|_{L^p(\bbr^d,dx)} +\frac{C_{\alpha,q,1/\alpha,1}}{2\Gamma(1/\alpha)} \left(\int_{\bbr^d}\|F_\alpha(y)\| \mu_\alpha(dy)\right) \| \| D^{\alpha-1}(f) \| \|_{L^p(\bbr^d,dx)},
\end{align*}
where $q = p/(p-1)$ and where the constant $C_{\alpha,q,1/\alpha,1}$ is given by \eqref{ineq:Riesz_Transform_2_all_p}.
\end{prop}

\begin{proof}
First, observe that since $\sigma$ is symmetric, the semigroup $(P^\alpha_t)_{t \geq 0}$ is symmetric (on $L^2(\bbr^d,dx)$) and so, by the spectral theorem or by Fourier arguments, for all $f,g \in \mathcal{C}^{\infty}_c(\bbr^d)$,
\begin{align*}
\int_{\bbr^d} \left(E-\mathcal{A}_\alpha\right)^{\frac{\alpha-1}{\alpha}}(f)(x) \left(E-\mathcal{A}_\alpha\right)^{\frac{1}{\alpha}}(g)(x) dx  = \langle f;g\rangle_{L^2(\bbr^d,dx)} - \langle f ; \mathcal{A}_\alpha(g)\rangle_{L^2(\bbr^d,dx)}.
\end{align*}
Next, observe that, for all $f,g \in \mathcal{C}_c^{\infty}(\bbr^d)$,
\begin{align*}
- \langle f ; \mathcal{A}_\alpha(g)\rangle_{L^2(\bbr^d,dx)} = \frac{1}{\alpha} \int_{\bbr^d} \langle \nabla(f)(x) ; D^{\alpha-1}(g)(x) \rangle dx = \frac{1}{\alpha} \int_{\bbr^d} \langle D^{\alpha-1}(f)(x) ; \nabla(g)(x) \rangle dx,
\end{align*}
so that, for all $f,g \in \mathcal{C}^{\infty}_c(\bbr^d)$,
\begin{align*}
\int_{\bbr^d} \left(E-\mathcal{A}_\alpha\right)^{\frac{\alpha-1}{\alpha}}(f)(x) \left(E-\mathcal{A}_\alpha\right)^{\frac{1}{\alpha}}(g)(x) dx  = \langle f;g\rangle_{L^2(\bbr^d,dx)} + \frac{1}{\alpha} \int_{\bbr^d} \langle D^{\alpha-1}(f)(x) ; \nabla(g)(x) \rangle dx.
\end{align*}
Then, by H\"older's inequality (with $q=p/(p-1)$), 
\begin{align*}
\left|\int_{\bbr^d} \left(E-\mathcal{A}_\alpha\right)^{\frac{\alpha-1}{\alpha}}(f)(x) \left(E-\mathcal{A}_\alpha\right)^{\frac{1}{\alpha}}(g)(x) dx\right| &\leq \|f\|_{L^p(\bbr^d,dx)}\|g\|_{L^q(\bbr^d,dx)}\\
&\quad\quad +\| \| D^{\alpha-1}(f) \| \|_{L^p(\bbr^d,dx)} \| \|\nabla(g) \| \|_{L^q(\bbr^d,dx)}.
\end{align*}
Thus, from \eqref{ineq:Riesz_Transform_2_all_p} with $r = 1/\alpha$ and $\lambda_0 = 1$, 
\begin{align}\label{ineq:dual_stable}
\left|\int_{\bbr^d} \left(E-\mathcal{A}_\alpha\right)^{\frac{\alpha-1}{\alpha}}(f)(x) \left(E-\mathcal{A}_\alpha\right)^{\frac{1}{\alpha}}(g)(x) dx\right| &\leq \|f\|_{L^p(\bbr^d,dx)} \|g\|_{L^q(\bbr^d,dx)}\nonumber \\
&\quad + \frac{C_{\alpha,q,1/\alpha,1}}{2\Gamma(1/\alpha)} \| \| D^{\alpha-1}(f) \| \|_{L^p(\bbr^d,dx)} \nonumber \\
&\quad \times \left(\int_{\bbr^d} \|F_\alpha(y)\| \mu_\alpha(dy)\right) \|\left(E-\mathcal{A}_\alpha\right)^{\frac{1}{\alpha}}(g)\|_{L^q(\bbr^d, dx)}.
\end{align}
Next, let 
\begin{align*}
\mathcal{E}_\alpha := \{\left(E-\mathcal{A}_\alpha\right)^{\frac{1}{\alpha}}(g),\quad g\in \mathcal{D}(\left(E-\mathcal{A}_\alpha\right)^{\frac{1}{\alpha}})\},
\end{align*}
where $\mathcal{D}(\left(E-\mathcal{A}_\alpha\right)^{{1/\alpha}})$ is the $L^q(\bbr^d,dx)$-domain of the unbounded operator $\left(E-\mathcal{A}_\alpha\right)^{{1/\alpha}}$.  

Recall that, by the theory of Bernstein functions (see, e.g., \cite[Chapter $3.9$]{NJ02_1}) and since $\alpha\in (1,2)$, the unbounded operator $-\left(E-\mathcal{A}_\alpha\right)^{{1/\alpha}}$ is the generator of a subordinated semigroup $(Q^\alpha_t)_{t \geq 0}$ with respect to $(e^{-t}P^\alpha_t)_{t\geq 0}$. In particular, the subordinated semigroup inherits the properties of the semigroup $(e^{-t}P_t^\alpha)_{t\geq 0}$: $(Q^\alpha_t)_{t\geq 0}$ is a semigroup of contracting symmetric operators on $L^2(\bbr^d,dx)$ which are contracting on $L^p(\bbr^d,dx)$ for $p \in [1, +\infty]$. Invoking \cite[Lemma 1]{Russ}, the set $\mathcal{E}_\alpha$ is dense in $L^q(\bbr^d,dx)$ for $q\in (1,+\infty)$. Then, by duality, 
\begin{align*}
\left\| \left(E-\mathcal{A}_\alpha\right)^{\frac{\alpha-1}{\alpha}}(f)\right\|_{L^p(\bbr^d,dx)} = \underset{g^* \in \mathcal{E}_\alpha,\ \|g^*\|_{L^q(\bbr^d,dx)}\leq 1}{\sup} \left|\langle \left(E-\mathcal{A}_\alpha\right)^{\frac{\alpha-1}{\alpha}}(f) ; g^* \rangle\right|.
\end{align*}
Now, since $g^* \in \mathcal{E}_\alpha$, let $g \in \mathcal{D}(\left(E-\mathcal{A}_\alpha\right)^{{1/\alpha}})$ be such that $\left(E-\mathcal{A}_\alpha\right)^{{1/\alpha}}(g)=g^*$. By \eqref{ineq:dual_stable},
\begin{align*}
\left|\langle \left(E-\mathcal{A}_\alpha\right)^{\frac{\alpha-1}{\alpha}}(f) ; g^* \rangle\right| &\leq \|f\|_{L^p(\bbr^d,dx)} \|g\|_{L^q(\bbr^d,dx)} \\
&\quad + \frac{C_{\alpha,q,1/\alpha,1}}{2\Gamma(1/\alpha)} \| \| D^{\alpha-1}(f) \| \|_{L^p(\bbr^d,dx)} \left(\int_{\bbr^d} \| F_\alpha(y) \| \mu_\alpha(dy)\right)\\
&\quad\times \|\left(E-\mathcal{A}_\alpha\right)^{\frac{1}{\alpha}}(g)\|_{L^q(\bbr^d, dx)}.
\end{align*}
The end of the proof follows.
\end{proof}

\begin{rem}\label{rem:second_order_Riesz_transform_Calderon_Zygmund_estimate}
As a byproduct of the results of this section, it is clear that one obtains the $L^p(\bbr^d,dx)$-$L^p(\bbr^d,dx)$ boundedness (with $p \in (1,+\infty)$) of the fractional second order Riesz transform defined, in the Fourier domain, for all $f \in \mathcal{S}(\bbr^d)$, all $\lambda_0>0$, all $\xi \in \bbr^d$ and all $j,k \in \{1, \dots, d\}$, by
\begin{align}\label{eq:Fourier_Rep_Second_Order_Riesz_Transform}
\mathcal{F}\left(D^{\alpha-1}_j \partial_k \circ \left(\lambda_0 E - \mathcal{A}_\alpha\right)^{-1}(f) \right)(\xi) = \dfrac{\tau_{\alpha,j}(\xi) (i \xi_k)}{\lambda_0 + \int_{\mathbb{S}^{d-1}}|\langle \xi ; y \rangle|^\alpha \lambda_1(dy)},
\end{align}
with,
\begin{align*}
D^{\alpha-1}_j(f)(x)  = \int_{\bbr^d} u_j (f(x+u)-f(x))\nu_\alpha(du), \quad
\tau_{\alpha,j}(\xi)  = \int_{\bbr^d} u_j \left(e^{i \langle u ;\xi \rangle}-1\right)\nu_\alpha(du). 
\end{align*}
Indeed, thanks to \eqref{eq:Fourier_Rep_Second_Order_Riesz_Transform},  
\begin{align*}
D^{\alpha-1}_j \partial_k \circ \left(\lambda_0 E - \mathcal{A}_\alpha\right)^{-1} =  \left(D^{\alpha-1}_j \circ  \left(\lambda_0 E - \mathcal{A}_\alpha\right)^{-\frac{\alpha-1}{\alpha}} \right) \circ \left(\partial_k \circ \left(\lambda_0 E - \mathcal{A}_\alpha\right)^{-\frac{1}{\alpha}}\right),
\end{align*}
(the above can also be seen as a straightforward consequence of the commutation relations $\partial_k \circ P^\alpha_t = P^\alpha_t \circ \partial_k$ and $D_j^{\alpha-1} \circ P_t^\alpha = P_t^\alpha \circ D_j^{\alpha-1}$). Then, it is natural to consider fractional Hessian matrices associated with the non-degenerate and symmetric L\'evy measures $\nu_\alpha$: for all $\alpha \in (1,2)$, all $f \in \mathcal{S}(\bbr^d)$, all $j,k \in \{1, \dots, d\}$ and all $x \in \bbr^d$,
\begin{align*}
\left(\operatorname{Hess}_\alpha(f)(x)\right)_{j,k} = \left(D^{\alpha-1}_j \partial_k\right)(f)(x).
\end{align*}
Now, a reformulation of the $L^p(\bbr^d,dx)$-$L^p(\bbr^d,dx)$ boundedness of the fractional second order Riesz transform ensures the following Calder\' on-Zygmund type inequality on $\mathcal{S}(\bbr^d)$: for all $p \in (1,+\infty)$, all $f \in \mathcal{S}(\bbr^d)$ and all $j,k \in \{1, \dots, d\}$,
\begin{align*}
\| \left(\operatorname{Hess}_\alpha(f)\right)_{j,k} \|_{L^p(\bbr^d,dx)} \leq C_{\alpha,p,d,j,k} \left( \| f \|_{L^p(\bbr^d,dx)} + \| \mathcal{A}_{\alpha}(f) \|_{L^p(\bbr^d,dx)} \right), \quad\quad \operatorname{CZI}(\alpha,p)
\end{align*}
for some $C_{\alpha,p,d,j,k} >0$ depending on $\alpha$, $p$, $d$, $j$ and $k$.  Note that the previous inequality is linked to the fact that $\mathcal{S}(\bbr^d)$ is a subspace of the $L^p(\bbr^d,dx)$-domain of the generator $\mathcal{A}_\alpha$.~The $L^p(\bbr^d,dx)$-domains of the generators of such convolution semigroups have already been studied in depth in \cite[Section $3.3$]{NJ02_2}.  Indeed,  \cite[Theorem $3.3.11$]{NJ02_2} ensures that, for all $p \in (1,+\infty)$, the $L^p(\bbr^d,dx)$-domain of the generator $\mathcal{A}_\alpha$ is equal to the Bessel potential space $H^{\psi_\alpha,2}_p(\bbr^d,dx)$ (see \cite[Definition $3.3.3$]{NJ02_2}) associated with the continuous negative definite function $\psi_\alpha$ defined,  for all $\xi \in \bbr^d$, by
\begin{align}\label{eq:symbol_psi_alpha}
\psi_\alpha(\xi) = \int_{\mathbb{S}^{d-1}} |\langle \xi ; y \rangle|^\alpha \lambda_1(dy).
\end{align}
Moreover, according to \cite[Theorem $3.3.8$]{NJ02_2}, the space $\mathcal{S}(\bbr^d)$ is dense in $H^{\psi_\alpha,2}_p(\bbr^d,dx)$.  Thus, $CZI(\alpha,p)$ extends to the whole space $H^{\psi_\alpha,2}_p(\bbr^d,dx)$.
\end{rem}

\begin{rem}\label{rem:Comparaison_Meyer}
In \cite[Chapter $IV$]{M0}, P.A. Meyer studies the Littlewood-Paley-Stein inequalities associated with symmetric convolution semigroups on $\bbr^d$ of which the semigroup $(P^\alpha_t)_{t \geq 0}$, with $\alpha \in (1,2)$,  is a particular instance. Based on the ``carr\'e du champs" operator and on the associated Littlewood-Paley-Stein inequalities, he states the following solution to the problem of equivalence of norms for the symmetric convolution semigroup $(P^\alpha_t)_{t \geq 0}$ (see \cite[Theorem $3$ page $180$]{M0}): for all $p \in (1, + \infty)$ and all $f$ good enough, 
\begin{align}\label{ineq:norm_equivalence}
c_p \| \sqrt{\Gamma_\alpha(f,f)} \|_{L^p(\bbr^d,dx)} \leq \|B_\alpha f\|_{L^p(\bbr^d,dx)} \leq C_p \| \sqrt{\Gamma_\alpha(f,f)} \|_{L^p(\bbr^d,dx)},
\end{align}
for some $C_p, c_p>0$ independent of the dimension, where $B_\alpha$ is the generator of the Cauchy semigroup subordinated to $(P^\alpha_t)_{t \geq 0}$ and where $\Gamma_\alpha$ is the square field operator associated with $\mathcal{A}_\alpha$ and given, for all $f,g \in \mathcal{S}(\bbr^d)$ and all $x \in \bbr^d$, by
\begin{align*}
\Gamma_\alpha(f,g)(x) = \frac{1}{2} \left(\mathcal{A}_\alpha(f\overline{g})(x) - f(x) \mathcal{A}_\alpha(\overline{g})(x) - \overline{g}(x) \mathcal{A}_\alpha(f)(x)\right),
\end{align*}
where $\overline{g}$ is the complex conjugate of $g$.  By Fourier inversion, for all $f,g \in \mathcal{S}(\bbr^d)$ and all $x \in \bbr^d$, 
\begin{align*}
&\Gamma_\alpha(f,g)(x) = \frac{1}{2} \int_{\bbr^{2d}} \mathcal{F}(f)(\xi_1) \overline{\mathcal{F}(g)(\xi_2)} e^{i \langle x ; \xi_1 - \xi_2 \rangle} \left(\psi_\alpha(\xi_1)+ \psi_\alpha(\xi_2) - \psi_\alpha(\xi_1-\xi_2) \right) \frac{d\xi_1 d\xi_2}{(2\pi)^{2d}}, \\
&B_\alpha(f)(x) = \frac{1}{(2\pi)^d} \int_{\bbr^d} \mathcal{F}(f)(\xi) e^{i \langle x ; \xi\rangle}\left(- \sqrt{\psi_\alpha(\xi)}\right) d\xi.
\end{align*}
Finally, \cite{M0} reformulates an inequality necessary to obtain one side of \eqref{ineq:norm_equivalence} as a multiplier theorem on $L^p(\bbr^d,dx)$, with $p \in (1,+\infty)$, reminiscent of the one obtained in Corollary \ref{cor_fractional_homo_Riesz_transform}.  No regularity assumption is imposed there on the Fourier transform of the admissible distribution $\Lambda$ (see \cite[Inequality $(9)$ page $178$]{M0}) and the operator norms are dimension-free (see \cite[Theorem $6$ page $182$]{M0}). The method is completely different from the one followed in this manuscript but predicts results at a higher degree of generality. On the other hand, the combination of Theorem \ref{prop:Riesz_sharp} and of Proposition \ref{prop:lower_bound_Riesz_stable_order_1} provides a new formulation of the equivalence of norms associated with the ``fractional derivatives" $D^{\alpha-1}$ and $\left(E-\mathcal{A}_\alpha\right)^{\frac{\alpha-1}{\alpha}}$.
\end{rem}

\begin{rem}\label{rem:Junge_JEMS}
\cite{JMP18} obtains solutions to Meyer's problem (MP) of equivalence of norms in a very general setting which encompasses the case of symmetric convolution semigroups on $\bbr^d$. As noticed by the authors,  for the Cauchy semigroup,  the probabilistic formulation of (MP) fails to be true if $p\in (1, 2d/(d+1)]$ and $d \geq 2$. For the sake of completeness, let us recall and modify slightly the argument contained in \cite[Appendix $D$]{JMP18}.  Let $f$ be a non-zero infinitely differentiable function on $\bbr^d$ with compact support contained in the Euclidean unit ball. Recall that the square field operator associated with a well-chosen rotationally invariant Cauchy probability measure on $\bbr^d$ admits the following representation: for all $x \in \bbr^d$,
\begin{align*}
\Gamma_1(f,f)(x) =  \frac{1}{2} \int_{\bbr^d} |f(x+u) - f(x)|^2 \frac{du}{\|u\|^{d+1}}.
\end{align*}
Take $x \in \bbr^d$ with $\|x\| \geq 3$. Then, 
\begin{align*}
\Gamma_1(f,f)(x) & =  \frac{1}{2} \int_{\bbr^d} |f(x+u) |^2 \frac{du}{\|u\|^{d+1}} = \frac{1}{2} \int_{\mathcal{B}(0,1)} |f(v)|^2 \frac{dv}{\|v-x\|^{d+1}}  \geq \frac{c_{f,d}}{\|x\|^{d+1}},
\end{align*}
for some $c_{f,d}>0$ depending on $f$ and on $d$.  Then, if $p \in (1, 2d/(d+1)]$, it is clear that $\sqrt{\Gamma_1(f,f)} \notin L^p(\bbr^d,dx)$ whereas $(-\mathcal{A}_1)^{\frac{1}{2}}(f) \in L^p(\bbr^d,dx)$.  By a similar argument, the probabilistic formulation of (MP) fails to be true in the rotationally invariant situation for $\alpha \in (1,2)$, $d \geq 2$ and $p \in (1,2d/(\alpha+d)]$.  Finally, let us discuss the results obtained in \cite{JMP18} regarding the boundedness of Riesz transforms associated with the standard fractional Laplacian operators.This setting corresponds to the invariance by rotation of the spherical part of the L\'evy measure.  According to \cite[Section $1.4$, Example $A$]{JMP18},  \cite[Theorem $A.1$]{JMP18} for locally compact unimodular groups provides boundedness results for the Riesz transforms defined
in the Fourier domain, for all $f \in \mathcal{S}(\bbr^d)$, all $\xi \in \bbr^d$ and all $j \geq 1$, by
\begin{align*}
\mathcal{F} \left(\mathcal{R}_{\beta,j}(f)\right)(\xi) = \dfrac{\langle b(\xi) ; e_j \rangle_{\mathcal{H}_\beta}}{\|\xi\|^\beta} \mathcal{F}(f)(\xi),
\end{align*}
where $\beta \in (0,1)$, $b(\xi)(.) = \exp \left(i \langle . ; \xi \rangle\right) - 1$,  
\begin{align*}
\mathcal{H}_{\beta} = L^2\left(\bbr^d, \dfrac{dx}{k_d(\beta)\|x\|^{2\beta+d}}\right),
\end{align*}
for some normalizing constant $k_d(\beta)>0$, and where $\{e_j ,\, j \geq 1\}$ is some orthonormal basis of $\mathcal{H}_{\beta}$.
\end{rem}

\begin{rem}\label{rem:kim}
Beyond the non-degenerate symmetric $\alpha$-stable case, \cite{BB07,BBB11,BanOse12,BBL16,BK19} obtain Fourier multiplier theorems for symbols which are linked to L\'evy processes.~For example, \cite[Theorem $4.1$]{BK19} ensures the $L^p(\bbr^d,dx)$-boundedness, $p \in (1,+\infty)$, of convolution operators with symbols defined, for all $\xi \in \bbr^d$, by
\begin{align*}
m_{\phi,\nu}(\xi) := \int_{\bbr^d} \int_0^{+\infty} \left|e^{i \langle u ; \xi \rangle}-1\right|^2 \phi(t,u) e^{- 2t \operatorname{Re}(\psi(\xi))}\nu(du)dt,
\end{align*}
where $\nu$ is a L\'evy measure on $\bbr^d$, $\psi$ is the corresponding L\'evy exponent satisfying the Hartman–Wintner condition and $\phi$ is a bounded and measurable real-valued function on $(0,+\infty) \times \bbr^d$. In particular, when $\psi$ is the L\'evy exponent of a non-degenerate symmetric $\alpha$-stable probability measure on $\bbr^d$ and when $\phi$ is time independent, then $m_{\phi,\nu_\alpha}$ boils down to
\begin{align*}
m_{\phi,\nu_\alpha}(\xi) := \dfrac{\int_{\bbr^d} \left|e^{i \langle u ; \xi \rangle}-1\right|^2 \phi(u) \nu_\alpha(du)}{2 \psi_\alpha(\xi)}, \quad \xi \in \bbr^d.
\end{align*}
\end{rem}

\section{The Stable Cases}\label{sec:SC}
\noindent
We start by introducing some notations. Let $K_{\alpha,r}$ be the non-negative, integrable, function defined, 
for all $r \in (r_{\min}(\alpha);+\infty)$ and all $t>0$, by
\begin{align}\label{def:kernel_alpha}
K_{\alpha,r}(t) = \frac{1}{2 \Gamma\left(\frac{r}{2}\right)} \dfrac{e^{-\alpha t}}{t^{1-\frac{r}{2}} \left(1-e^{-\alpha t}\right)^{1-\frac{1}{\alpha}}}, 
\end{align}
and let $\mathcal{I}_{\alpha,r}$ be the linear operator, defined, for all $f \in \mathcal{S}(\bbr^d)$ and all $x,y \in \bbr^d$, by
\begin{align}\label{def:singular_operator_alpha}
\mathcal{I}_{\alpha,r}(f)(x,y) = \int_{0}^{+\infty} K_{\alpha,r}(t) \left(f\left(xe^{-t}+y(1-e^{-\alpha t})^{\frac{1}{\alpha}}\right)-f\left(xe^{-t}-y(1-e^{-\alpha t})^{\frac{1}{\alpha}}\right)\right) dt.
\end{align}
Then, the operator ${\mathbf D}^{\alpha-1} \circ \left(E-\mathcal{L}^\alpha\right)^{-\frac{r}{2}}$ admits the following representation: for all $f \in \mathcal{S}(\bbr^d)$ and all $x \in \bbr^d$,
\begin{align}\label{compact:rep}
{\mathbf D}^{\alpha-1} \circ \left(E-\mathcal{L}^\alpha\right)^{-\frac{r}{2}}(f)(x) & 
= \int_{\bbr^d} y \mathcal{I}_{\alpha,r}(f)(x,y) \mu_\alpha(dy). 
\end{align}
This representation is analoguous to the one obtained in the proof of \cite[Corollary $3$]{F90} for the Gaussian case where the continuity properties of an ergodic singular operator associated with a dilation of a re-parametrization of the Gaussian OU semigroup is first analyzed. The next theorem studies the $L^p(\mu_\alpha)$-continuity of this operator for the non-singular situation $r > r_{\min}(\alpha)$.

\begin{thm}\label{thm:First}
Let $\alpha \in (1,2)$, let $r \in (r_{\min}(\alpha);+\infty)$, let $\nu_\alpha$ be a non-degenerate symmetric L\'evy measure on $\bbr^d$ satisfying \eqref{eq:scale}, let $\mu_\alpha$ be the associated $\alpha$-stable probability 
measure on $\bbr^d$ given by \eqref{eq:rep_spectral_measure}, let $\mathbf{D}^{\alpha-1}$ be defined by \eqref{eq:FracGradWB} and let $\mathcal{L}^\alpha$ be defined by \eqref{eq:StOUgen}.  
Then, for all $p\in (\alpha/(\alpha-1),+\infty)$ and all $f\in L^p(\mu_\alpha)$,
\begin{align}\label{ineq:continuity_Riesz_stable1}
\| \|{\mathbf D}^{\alpha-1} \circ \left(E-\mathcal{L}^\alpha\right)^{-\frac{r}{2}}(f)\| \|_{L^p(\mu_\alpha)} \leq  2\gamma_\alpha(q) \|K_{\alpha,r}\|_{L^1((0,+\infty),dt)} \|f\|_{L^p(\mu_\alpha)},
\end{align}
with $q=p/(p-1)$ and with
\begin{align*}
\gamma_\alpha(q) =  \left(\bbe |Y_\alpha|^q\right)^{1/q} \underset{z \in \bbr^d, \|z\| = 1}{\sup}\left( \int_{\mathbb{S}^{d-1}} |\langle z; x \rangle|^\alpha \lambda_1(dx) \right)^{{1/\alpha}},
\end{align*}
where $\lambda_1$ is the spectral measure of $\mu_\alpha$ and $Y_\alpha$ is a symmetric $\alpha$-stable random variable with characteristic function given by $\exp \left( - |\xi|^\alpha\right)$, for $\xi \in \bbr$.
\end{thm}

\begin{proof}
First, by duality and the representation \eqref{compact:rep}, for all $f \in \mathcal{S}(\bbr^d)$ and 
all $x \in \bbr^d$, 
\begin{align*}
\|{\mathbf D}^{\alpha-1} \circ \left(E-\mathcal{L}^\alpha\right)^{-\frac{r}{2}}(f)(x) \| 
= \underset{z\in \bbr^d,\, \|z\|=1}{\sup} \left|\langle z ; \int_{\bbr^d} y \mathcal{I}_{\alpha,r}(f)(x,y) \mu_\alpha(dy) \rangle\right|.
\end{align*}
Let $p \in (\alpha/(\alpha-1),+\infty)$ and let $q$ be its conjugate exponent, then by H\"older's inequality,  
\begin{align*}
\|{\mathbf D}^{\alpha-1} \circ \left(E-\mathcal{L}^\alpha\right)^{-\frac{r}{2}}(f)(x) \| \leq  \underset{z\in \bbr^d,\, \|z\|=1}{\sup} \left(\int_{\bbr^d} \left|\langle y;z \rangle\right|^q \mu_\alpha(dy)\right)^{\frac{1}{q}} \left(\int_{\bbr^d} \left|\mathcal{I}_{\alpha,r}(f)(x,y)\right|^p \mu_\alpha(dy)\right)^{\frac{1}{p}}. 
\end{align*}
Now, performing an analysis similar to the one performed 
between the inequalities \eqref{eq:def_gamma_alpha} and \eqref{eq:rep_gamma_alpha}, it follows that, 
\begin{align*}
\|{\mathbf D}^{\alpha-1} \circ \left(E-\mathcal{L}^\alpha\right)^{-\frac{r}{2}}(f)(x) \| \leq  \gamma_\alpha(q)\left(\int_{\bbr^d} \left|\mathcal{I}_{\alpha,r}(f)(x,y)\right|^p \mu_\alpha(dy)\right)^{\frac{1}{p}}, 
\end{align*}
for all $f \in \mathcal{S}(\bbr^d)$ and all $x \in \bbr^d$.  
Thus, for all $f \in \mathcal{S}(\bbr^d)$, 
\begin{align*}
\| \|{\mathbf D}^{\alpha-1} \circ \left(E-\mathcal{L}^\alpha\right)^{-\frac{r}{2}}(f)\| \|_{L^p(\mu_\alpha)} \leq  \gamma_\alpha(q) \|\mathcal{I}_{\alpha,r}(f)\|_{L^p(\mu_\alpha \otimes \mu_\alpha)}.
\end{align*}
Minkowski's integral inequality together with the fact that under $\mu_\alpha \otimes \mu_\alpha$, $xe^{-t}\pm(1-e^{-\alpha t})^{\frac{1}{\alpha}}y$ has law $\mu_\alpha$ finishes the proof of the theorem. 
\end{proof}

Note that for the rotationally invariant case, \eqref{ineq:continuity_Riesz_stable1} boils down 
to the \textit{dimension-free estimate},
\begin{align}\label{ineq:dimensionfree_Riesz_Stable}
\| \|{\mathbf D}^{\alpha-1, \operatorname{rot}} \circ \left(E-\mathcal{L}^{\alpha,\operatorname{rot}}\right)^{-\frac{r}{2}}(f)\| \|_{L^p(\mu_{\alpha}^{\rm rot})} \leq   2^{(\alpha-1)/\alpha} \left(\bbe |Y_\alpha|^q\right)^{{1/q}}\|K_{\alpha,r}\|_{L^1((0,+\infty),dt)} \|f\|_{L^p(\mu_{\alpha}^{\rm rot})}, 
\end{align}
which is valid for all $p\in (\alpha/(\alpha-1),+\infty)$ and all 
$f\in L^p(\mu_{\alpha}^{\rm rot})$.  

A drawback of the previous inequality is that it does not pass to the limit 
as $\alpha$ tends to $2$ with $r = 1$ or as $r$ tends to $r_{\min}(\alpha)$ with $\alpha$ fixed in $(1,2)$. Indeed, the $L^1$-norm of the non-negative function $K_{\alpha,1}$ diverges as $\alpha$ tends to $2$. Thus, a more refined approach seems necessary in order to retrieve the celebrated Gaussian case 
(see, e.g., \cite[Chapter $9$]{Ur19} for relevant pointers to this very rich literature). One approach to the Gaussian case which provides dimension-free estimate is based, as in the homogeneous case, on an elegant transference argument (see \cite{P88}) in order to obtain a sharp $L^p(\gamma\otimes\gamma)$-continuity result for the operator $\mathcal{I}_{2,1}$ (see also \cite{F90}). This transference argument takes into account the regularization effect due to the symmetry of the Gaussian measure; however it is not clear if such a transference argument passes to the general symmetric $\alpha$-stable situation.~To end this section, let us prove a similar result for the ``carr\'e de Mehler".  


\begin{cor}\label{thm:Second}
Let $\alpha \in (1,2)$, let $r \in (r_{\min}(\alpha),+\infty)$, let $\nu_\alpha$ be a non-degenerate symmetric L\'evy measure on $\bbr^d$ satisfying \eqref{eq:scale}, let $\mu_\alpha$ be the associated $\alpha$-stable probability 
measure on $\bbr^d$ given by \eqref{eq:rep_spectral_measure}, let $\mathbf{D}^{\alpha-1}$ be defined by \eqref{eq:FracGradWB} and let $\mathcal{L}$ be the generator of the ``carr\'e de Mehler".  Then, for all $p\in (\alpha/(\alpha-1),+\infty)$ and all $f\in L^p(\mu_\alpha)$,
\begin{align}\label{ineq:continuity_Riesz_stable3}
\| \|{\mathbf D}^{\alpha-1} \circ \left(E-\mathcal{L}\right)^{-\frac{r}{2}}(f)\| \|_{L^p(\mu_\alpha)} 
\leq  \alpha^{(\alpha - 1)/\alpha}C_\alpha \gamma_\alpha(q) \dfrac{\Gamma\left(\frac{r}{2}+\frac{1}{\alpha}-1\right)}{\Gamma(\frac{r}{2})} \|f\|_{L^p(\mu_\alpha)},
\end{align}
with $q=p/(p-1)$ and with
\begin{align*}
C_\alpha = \underset{t>0}{\max} \dfrac{t^{1-\frac{1}{\alpha}}e^{-(\alpha-1)t}}{\left(1-e^{-\alpha t}\right)^{1-\frac{1}{\alpha}}}, \quad
\gamma_\alpha(q) =  \left(\bbe |Y_\alpha|^q\right)^{{1/q}} \underset{z \in \bbr^d, \|z\| = 1}{\sup}\left( \int_{\mathbb{S}^{d-1}} |\langle z; x \rangle|^\alpha \lambda_1(dx) \right)^{{1/\alpha}},
\end{align*}
where $\lambda_1$ is the spectral measure of $\mu_\alpha$ and $Y_\alpha$ is a symmetric real-valued $\alpha$-stable random variable with characteristic function given by $\exp \left( - |\xi|^\alpha\right)$, for $\xi \in \bbr$.
\end{cor}

\begin{proof}
From \eqref{eq:bismut_stable_OU}, for all $f \in \mathcal{S}(\bbr^d)$, all $x \in \bbr^d$ and all $t>0$,
\begin{align}\label{eq:bismut_stable_OU2}
\mathbf{D}^{\alpha-1}(P^{\nu_\alpha}_t(f))(x) = \dfrac{e^{-(\alpha-1)t}}{\left(1-e^{-\alpha t}\right)^{1-\frac{1}{\alpha}}} \int_{\bbr^d} y f(xe^{-t}+(1-e^{-\alpha t})^{\frac{1}{\alpha}}y) \mu_\alpha(dy).
\end{align}
Now, by duality and H\"older's inequality ($q\in (1, \alpha)$ and $q=p/(p-1)$), for all $f \in \mathcal{S}(\bbr^d)$, all $x \in \bbr^d$ and all $t>0$, 
\begin{align*}
\|\mathbf{D}^{\alpha-1}(P^{\nu_\alpha}_t(f))(x)\| & =\dfrac{e^{-(\alpha-1)t}}
{\left(1-e^{-\alpha t}\right)^{1-\frac{1}{\alpha}}} \underset{z\in \bbr^d,\, \|z\|=1}{\sup} \left|\langle z; \int_{\bbr^d} y f(xe^{-t}+(1-e^{-\alpha t})^{\frac{1}{\alpha}}y) \mu_\alpha(dy) \rangle\right| \\
& = \dfrac{e^{-(\alpha-1)t}}{\left(1-e^{-\alpha t}\right)^{1-\frac{1}{\alpha}}} \underset{z\in \bbr^d,\, \|z\|=1}{\sup} \left| \int_{\bbr^d} \langle z;y \rangle f(xe^{-t}+(1-e^{-\alpha t})^{\frac{1}{\alpha}}y) \mu_\alpha(dy)\right| \\
& \leq \dfrac{e^{-(\alpha-1)t}}{\left(1-e^{-\alpha t}\right)^{1-\frac{1}{\alpha}}} \gamma_\alpha(q) \left(\int_{\bbr^d} \left| f(xe^{-t}+(1-e^{-\alpha t})^{\frac{1}{\alpha}}y) \right|^p \mu_\alpha(dy)\right)^{{1/p}} \\
&\leq \frac{C_\alpha}{t^{1-\frac{1}{\alpha}}}  \gamma_\alpha(q) \left(\int_{\bbr^d} \left| f(xe^{-t}+(1-e^{-\alpha t})^{\frac{1}{\alpha}}y) \right|^p \mu_\alpha(dy)\right)^{{1/p}},
\end{align*}
with $C_\alpha = \underset{t>0}{\max} \dfrac{t^{1-\frac{1}{\alpha}}e^{-(\alpha-1)t}}{\left(1-e^{-\alpha t}\right)^{1-\frac{1}{\alpha}}}$ and $\gamma_\alpha(q) =  \left(\bbe |Y_\alpha|^q\right)^{{1/q}} \underset{z \in \bbr^d, \|z\| = 1}{\sup}\left( \int_{\mathbb{S}^{d-1}} |\langle z; x \rangle|^\alpha \lambda_1(dx) \right)^{{1/\alpha}}$.
Integrating with respect to $\mu_\alpha$ gives
\begin{align*}
\left(\int_{\bbr^d}\|\mathbf{D}^{\alpha-1}(P^{\nu_\alpha}_t(f))(x)\|^p \mu_\alpha(dx)\right)^{{1/p}} \leq \frac{C_\alpha}{t^{1-\frac{1}{\alpha}}} \gamma_\alpha(q) \|f\|_{L^p(\mu_\alpha)}, 
\end{align*}
which, by density, extends to all $L^p(\mu_\alpha)$.  
Thus, for all $t>0$ and all $f \in L^p(\mu_\alpha)$, $p \in (\alpha/(\alpha-1),+\infty)$,
\begin{align*}
\| \|\mathbf{D}^{\alpha-1}(\mathcal{P}_t(f))\| \|_{L^p(\mu_\alpha)} \leq \frac{\alpha^{1-\frac{1}{\alpha}}C_\alpha}{t^{1-\frac{1}{\alpha}}} \gamma_\alpha(q) \|f\|_{L^p(\mu_\alpha)}.
\end{align*}
Finally, the associated operator of order $r$ is given, for all $f \in \mathcal{S}(\bbr^d)$, by
\begin{align*}
\mathbf{D}^{\alpha-1} \circ \left(E-\mathcal{L}\right)^{-\frac{r}{2}} (f) = \int_{0}^{+\infty} \dfrac{e^{-t}}{\Gamma(\frac{r}{2})t^{1-\frac{r}{2}}} \mathbf{D}^{\alpha-1} (\mathcal{P}_t(f)) dt,
\end{align*}
where $\mathcal{L}$ is the generator of $(\mathcal{P}_t)_{t\geq 0}$. Then, for all $f \in \mathcal{S}(\bbr^d)$,
\begin{align*}
\|\|\mathbf{D}^{\alpha-1} \circ \left(E-\mathcal{L}\right)^{-\frac{r}{2}} (f)\|\|_{L^p(\mu_\alpha)} & \leq  \int_{0}^{+\infty} \dfrac{e^{-t}}{\Gamma(\frac{r}{2})t^{1-\frac{r}{2}}} \|\mathbf{D}^{\alpha-1} (\mathcal{P}_t(f))\|_{L^p(\mu_\alpha)} dt , \\
& \leq \alpha^{1-\frac{1}{\alpha}}C_\alpha \gamma_\alpha(q) \left( \int_{0}^{+\infty} \frac{e^{-t}}{\Gamma(\frac{r}{2}) t^{2-\frac{r}{2}-\frac{1}{\alpha}}} dt \right) \|f\|_{L^p(\mu_\alpha)} , \\
& \leq \alpha^{1-\frac{1}{\alpha}}C_\alpha \gamma_\alpha(q) \dfrac{\Gamma\left(\frac{r}{2}+\frac{1}{\alpha}-1\right)}{\Gamma(\frac{r}{2})} \|f\|_{L^p(\mu_\alpha)},
\end{align*}
which concludes the proof of the corollary.  
\end{proof}

\section{Appendix}
\noindent

 \begin{lem}\label{lem:simple}
 Let $\alpha \in (1,2)$, let $r>0$ and let 
 \begin{align*}
 \tilde{\psi}_{\alpha,r}(s) := \frac{s^{\frac{\alpha-1}{2}}}{\left(2+ s^{\frac{\alpha}{2}}\right)^{\frac{r}{2}}},  
 \end{align*}
$s \in (0,+\infty)$.  Then, for all $k \geq 1$ and all $s \in (0,+\infty)$,
 \begin{align*}
  \tilde{\psi}^{(k)}_{\alpha,r}(s) = \dfrac{s^{\frac{\alpha-1}{2}-k}\tilde{P}_{\alpha,k,r}(s^{\frac{\alpha}{2}})}{(2+s^{\frac{\alpha}{2}})^{\frac{r}{2}+k}},
 \end{align*}
 where $\tilde{P}_{\alpha,k,r}$ is a polynomial of degree $k$ which coefficients depend on $k$, on $\alpha$ and on $r$.  
 \end{lem}
 
 \begin{proof}
 The proof follows from a standard recursive argument. 
 \end{proof}
 
\begin{lem}\label{lem:simple2}
Let $\psi$ be a real-valued function infinitely differentiable on $(0,+\infty)$.~Then, for all $k \geq 1$, all $j \in \{1, \dots, d\}$ and all $\xi \in \bbr^d$ such that $\xi \ne 0$,
\begin{align}\label{eq:simple2}
\partial_{\xi_j}^k \left(\psi\left(\|\xi\|^2\right)\right) = \sum_{p=0}^{\lfloor \frac{k}{2} \rfloor} C_p(k) \left(\xi_j\right)^{k-2p} \psi^{(k-p)}(\|\xi\|^2),
\end{align}
for some $C_p(k)>0$ only depending on $p$ and on $k$. 
\end{lem}
 
\begin{proof}
The proof relies on a classical recursive argument. First, observe that, for all $j \in \{1, \dots, d\}$ and all $\xi \in \bbr^d$ such that $\xi \ne 0$,
\begin{align*}
\partial_{\xi_j} \left(\psi\left(\|\xi\|^2\right)\right) = 2 \xi_j \psi^{(1)}\left(\|\xi\|^2\right). 
\end{align*}
Now, assume that the equality \eqref{eq:simple2} is true for some fixed integer $k \geq 1$. 
Then, for all $j \in \{1, \dots, d \}$ and all $\xi \in \bbr^d$ with $\xi \ne 0$,
\begin{align*}
\partial_{\xi_j}^{k+1} \left(\psi\left(\|\xi\|^2\right)\right)  &=  \partial_{\xi_j} \left( \sum_{p=0}^{\lfloor \frac{k}{2} \rfloor} C_p(k) \left(\xi_j\right)^{k-2p} \psi^{(k-p)}(\|\xi\|^2)\right), \\
&= \sum_{p=0}^{\lfloor \frac{k}{2} \rfloor} C_p(k) \partial_{\xi_j} \left(\left(\xi_j\right)^{k-2p} \psi^{(k-p)}(\|\xi\|^2)\right) . 
\end{align*}
Next, assume that $k$ is even, say, $k =2m$, for some integer $m \geq 1$.  
Then, for all $j \in \{1, \dots, d \}$ and all $\xi \in \bbr^d$, $\xi \ne 0$,
\begin{align*}
\partial_{\xi_j}^{k+1} \left(\psi\left(\|\xi\|^2\right)\right)  &= \sum_{p=0}^{m} C_p(2m) \partial_{\xi_j} \left(\left(\xi_j\right)^{2m-2p} \psi^{(2m-p)}(\|\xi\|^2)\right), \\
& = C_m(2m) \partial_{\xi_j} \left(\psi^{(m)}\left(\|\xi\|^2\right)\right) + \sum_{p=0}^{m-1} C_p(2m)\partial_{\xi_j} \left(\left(\xi_j\right)^{2m-2p} \psi^{(2m-p)}(\|\xi\|^2)\right) , \\
& = C_m(2m) 2 \xi_j \psi^{(m+1)} \left(\|\xi\|^2\right) + \sum_{p=0}^{m-1} C_p(2m) \bigg( (2m-2p) (\xi_j)^{2m-2p-1} \psi^{(2m-p)}(\|\xi\|^2) \\
&\qquad\qquad\qquad +2 (\xi_j)^{2m-2p+1} \psi^{(2m-p+1)}\left(\|\xi\|^2\right) \bigg), \\
& = C_m(2m) 2 \xi_j \psi^{(m+1)} \left(\|\xi\|^2\right) + \sum_{p=0}^{m-1} C_p(2m) \bigg( (2m-2p) (\xi_j)^{2m+1-2(p+1)} \psi^{(2m-p)}(\|\xi\|^2) \\
&\qquad\qquad\qquad+2 (\xi_j)^{2m+1-2p} \psi^{(2m+1-p)}\left(\|\xi\|^2\right) \bigg), \\
& = C_m(2m) 2 \xi_j \psi^{(m+1)} \left(\|\xi\|^2\right) + \sum_{\ell=1}^{m} C_{\ell-1}(2m) (2m-2(\ell-1)) (\xi_j)^{2m+1-2\ell} \psi^{(2m+1-\ell)}(\|\xi\|^2)\\
&\qquad\qquad\qquad + \sum_{p=0}^{m-1} 2 C_p(2m) (\xi_j)^{2m+1-2p} \psi^{(2m+1-p)}\left(\|\xi\|^2\right). 
\end{align*}
Rearranging the sums, for all $j \in \{1, \dots, d \}$ and all $\xi \in \bbr^d$, $\xi \ne 0$,
\begin{align*}
\partial_{\xi_j}^{k+1} \left(\psi\left(\|\xi\|^2\right)\right)  & = \left(2 C_m(2m) + 2C_{m-1}(2m)\right)\xi_j \psi^{(m+1)} \left(\|\xi\|^2\right) \\
&\qquad\qquad +  \sum_{\ell=1}^{m-1} C_{\ell-1}(2m) (2m-2(\ell-1)) (\xi_j)^{2m+1-2\ell} \psi^{(2m+1-\ell)}(\|\xi\|^2) \\
&\qquad\qquad +  \sum_{p=1}^{m-1} 2 C_p(2m) (\xi_j)^{2m+1-2p} \psi^{(2m+1-p)}\left(\|\xi\|^2\right)\\
&\qquad\qquad + 2C_0(2m) (\xi_j)^{2m+1} \psi^{(2m+1)}\left(\|\xi\|^2\right) . 
\end{align*}
Thus, since $k+1=2m+1$ and $\lfloor \frac{k+1}{2} \rfloor = m$, the result follows for $k$ even, 
while a similar reasoning ensures the result for $k$ odd. This concludes the proof of the lemma. 
\end{proof}
\noindent
\begin{lem}\label{lem:double_truncation}
Let $\alpha \in (1,2)$, let $\nu_\alpha$ be a non-degenerate symmetric L\'evy measure on $\bbr^d$ satisfying \eqref{eq:scale} , let $\mu_\alpha$ be the associated $\alpha$-stable probability measure on $\bbr^d$ given by \eqref{def:stable}, let $D^{\alpha-1}$ be defined by \eqref{eq:FracGrad} and let $\mathcal{A}_\alpha$ be defined by \eqref{eq:Stheatgen}.  Then, for all $f \in \mathcal{S}(\bbr^d)$ and all $x \in \bbr^d$, 
\begin{align*}
D^{\alpha-1} \circ \left(- \mathcal{A}_\alpha\right)^{- \frac{\alpha-1}{\alpha}}(f)(x) = \frac{1}{2\Gamma\left(1 - \frac{1}{\alpha}\right)} \int_{\bbr^d} y \tilde{T}_{\alpha,r_{\min}(\alpha)}^y(f)(x) \mu_\alpha(dy),
\end{align*}
with, for all $y \in \bbr^d$,
\begin{align*}
\tilde{T}_{\alpha,r_{\min}(\alpha)}^y(f)(x) = \int_{0}^{+\infty} \left(f(x+t^{\frac{1}{\alpha}}y) - f(x-t^{\frac{1}{\alpha}}y)\right) \frac{dt}{t}, \quad r_{\min}(\alpha) = \frac{2}{\alpha}(\alpha-1).
\end{align*}
\end{lem}

\begin{proof}
Let $\alpha \in (1,2)$, let $\lambda_1$ be a finite non-negative symmetric measure on $\mathbb{S}^{d-1}$ such that
\begin{align*}
c_0(\alpha,d) : = \underset{y \in \mathbb{S}^{d-1}}{\operatorname{inf}} \int_{\mathbb{S}^{d-1}} |\langle y ; u \rangle|^\alpha \lambda_1(du) > 0,
\end{align*}
and let $\nu_\alpha$ be the corresponding $\alpha$-stable L\'evy measure on $\bbr^d$.  The operator $D^{\alpha-1} \circ \left(- \mathcal{A}_\alpha\right)^{-r_{\min}(\alpha)/2}$ is defined, in the Fourier domain,  for all $f \in \mathcal{S}(\bbr^d)$ and all $\xi \in \bbr^d$ with $\xi \ne 0$, by
\begin{align*}
\mathcal{F}\left(D^{\alpha-1} \circ \left(- \mathcal{A}_\alpha\right)^{-\frac{r_{\min}(\alpha)}{2}}(f)\right)(\xi) = \dfrac{\tau_\alpha(\xi)}{\left(\int_{\mathbb{S}^{d-1}} |\langle \xi ; u \rangle|^\alpha \lambda_1(du)\right)^{\frac{\alpha-1}{\alpha}}} \mathcal{F}(f)(\xi),
\end{align*}
with $r_{\min}(\alpha) = 2(\alpha-1)/\alpha$.  Now,  let $R>0$ and let $\rho_\alpha$ and $\rho_{\alpha,R}$ be defined, for all $x \in \bbr$,  by
\begin{align*}
\rho_\alpha(x)= \int_0^{+\infty} \frac{dt}{t} \sin \left(t^{\frac{1}{\alpha}}x\right), \quad   \rho_{\alpha,R}(x)= \int_0^{R} \frac{dt}{t} \sin \left(t^{\frac{1}{\alpha}} x\right).
\end{align*}
By a simple integration and Fubini's theorem, both functions are well-defined and uniformly bounded on $\bbr$.  Moreover,  by definition, for all $x \in \bbr$,
\begin{align*}
\underset{R \rightarrow +\infty}{\lim} \rho_{\alpha,R}(x)=\rho_\alpha(x).
\end{align*}
Thus,  since $\alpha \in (1,2)$, one can define the following symbol on $\bbr^d$: for all $\xi \in \bbr^d$ and all $R>0$, 
\begin{align*}
m_{\alpha,R}(\xi) = \dfrac{i}{\Gamma\left(\frac{r_{\min}(\alpha)}{2}\right)} \int_{\bbr^d} y  \rho_{\alpha,R}(\langle y ; \xi \rangle) \mu_\alpha(dy). 
\end{align*}
Now, let $(\mu_{\alpha,m})_{m \geq 2}$ be the sequence of probability measures on $\bbr^d$ characterized, for all $\xi \in \bbr^d$ and all $m \geq 2$, by
\begin{align*}
\hat{\mu}_{\alpha,m}(\xi) = \exp \left( \int_{\bbr^d} \left(e^{i \langle u ;\xi \rangle}-1- i \langle u ;\xi \rangle\right) \nu_{\alpha,m}(du) \right),
\end{align*}
where $\nu_{\alpha,m}(du) = \bbone_{(0,m)}(\|u\|) \nu_\alpha(du)$ and let $(m_{\alpha,R,m})_{m \geq 2}$ be the corresponding sequence of symbols.~Note that the probability measures $\mu_{\alpha,m}$, $m \geq 2$, are infinitely divisible and have finite moments of any order and that the sequence $(\mu_{\alpha,m})_{m \geq 2}$ converges weakly to $\mu_\alpha$. Next,  by a direct application of Fubini's theorem and standard computations, for all $\xi \in \bbr^d$, all $R>0$ and all $m \geq 2$, 
\begin{align*}
m_{\alpha,R,m}(\xi) & = \dfrac{i}{\Gamma\left(1 - \frac{1}{\alpha}\right)} \int_{\bbr^d} y \rho_{\alpha,R}(\langle y;\xi \rangle) \mu_{\alpha,m}(dy) , \\
& = \dfrac{i}{\Gamma\left(1 - \frac{1}{\alpha}\right)} \int_0^R  \dfrac{1}{t} \left( \int_{\bbr^d} y \sin\left(t^{\frac{1}{\alpha}} \langle y ; \xi \rangle \right)  \mu_{\alpha,m}(dy) \right) dt , \\
& = \dfrac{1}{\Gamma(1- \frac{1}{\alpha})} \int_0^R \hat{\mu}_{\alpha,m} \left(t^{\frac{1}{\alpha}} \xi\right) \left( \int_{\bbr^d} u \left(e^{i t^{\frac{1}{\alpha}} \langle u ; \xi \rangle}-1\right) \nu_{\alpha,m}(du) \right) \frac{dt}{t}.
\end{align*}
Now, for all $t \in (0,R)$ and all $\xi \in \bbr^d$, 
\begin{align*}
\underset{ m \longrightarrow +\infty}{\lim}  \hat{\mu}_{\alpha,m} \left(t^{\frac{1}{\alpha}} \xi\right) \left( \int_{\bbr^d} u \left(e^{i t^{\frac{1}{\alpha}} \langle u ; \xi \rangle}-1\right) \nu_{\alpha,m}(du) \right) = \hat{\mu}_{\alpha} \left(t^{\frac{1}{\alpha}} \xi\right) \left( \int_{\bbr^d} u \left(e^{i t^{\frac{1}{\alpha}} \langle u ; \xi \rangle}-1\right) \nu_{\alpha}(du) \right).
\end{align*}
Moreover,  for all $m \geq 2$, all $t \in (0,R)$ and all $\xi \in \bbr^d$ with $\xi \ne 0$, 
\begin{align*}
\frac{1}{t}\left\| \int_{\mathcal{B}(0,m)} u \left(e^{i t^{\frac{1}{\alpha}} \langle u ; \xi \rangle}-1\right) \nu_{\alpha,m}(du)\right\| \left| \hat{\mu}_{\alpha,m} \left(t^{\frac{1}{\alpha}}\xi\right)\right| \leq \frac{1}{t^{\frac{1}{\alpha}}} \int_{\bbr^d} \|u\| \left| e^{i \langle u ; \xi \rangle} - 1 \right| \nu_\alpha(du).
\end{align*}
Thus, by the Lebesgue dominated convergence theorem, 
\begin{align*}
\underset{m \longrightarrow +\infty}{\lim} m_{\alpha,R,m}(\xi) = \dfrac{1}{\Gamma(1- \frac{1}{\alpha})} \int_0^R \hat{\mu}_{\alpha} \left(t^{\frac{1}{\alpha}} \xi\right) \tau_\alpha(\xi) \frac{dt}{t^{\frac{1}{\alpha}}}.
\end{align*}
Moreover, by standard arguments and since 
\begin{align}\label{eq:uniform_first_moment}
\underset{m \geq 2}{\sup} \int_{\bbr^d} \|y\|^\beta \mu_{\alpha,m}(dy) <+\infty,  \quad \beta \in [1,\alpha),
\end{align}
for all $\xi \in \bbr^d$ and all $R>0$, 
\begin{align*}
\underset{m \longrightarrow +\infty}{\lim} m_{\alpha,R,m}(\xi) = \dfrac{i}{\Gamma\left(1 - \frac{1}{\alpha}\right)} \int_{\bbr^d} y \rho_{\alpha,R}(\langle y;\xi \rangle) \mu_{\alpha}(dy).
\end{align*}
Note that \eqref{eq:uniform_first_moment} follows from the classical decomposition contained, e.g., in \cite[Lemma $1.1$]{MR01}.  Indeed, let us detail the argument. Let $X_{\alpha,m}$ be a random vector of $\bbr^d$ with characteristic function given, for all $\xi \in \bbr^d$ and all $m \geq 2$, by
\begin{align*}
\bbe e^{i \langle \xi; X_{\alpha,m} \rangle} = \exp\left(\int_{\bbr^d}\left(e^{i \langle u ;\xi \rangle} - 1 - i \langle u ; \xi \rangle\right) \nu_{\alpha,m}(du) \right).
 \end{align*}
 Since $m\geq 2$, the following representation in law holds true:
 \begin{align*}
X_{\alpha,m} =_{\mathcal{L}} Y_{\alpha,m} + Z_{\alpha,m},
\end{align*}
with $Y_{\alpha,m}$ and $Z_{\alpha,m}$ independent random vectors of $\bbr^d$ characterized, for all $\xi \in \bbr^d$ and all $m \geq 2$,  by
\begin{align*}
& \bbe e^{i \langle \xi; Y_{\alpha,m} \rangle} = \exp\left(\int_{\mathcal{B}(0,1)}\left(e^{i \langle u ;\xi \rangle} - 1 - i \langle u ; \xi \rangle\right) \nu_{\alpha,m}(du) \right), \\
& \bbe e^{i \langle \xi; Z_{\alpha,m} \rangle} = \exp\left(\int_{\mathcal{B}(0,1)^c}\left(e^{i \langle u ;\xi \rangle} - 1 - i \langle u ; \xi \rangle\right) \nu_{\alpha,m}(du) \right),
\end{align*}
where $\mathcal{B}(0,1)$ is the Euclidean open unit ball of $\bbr^d$. Now, since $m \geq 2$, $Y_{\alpha,m} =_{\mathcal{L}} Y_{\alpha,1}$.  Moreover, let $\ell_m = \nu_{\alpha,m}( u \in \bbr^d:\, \|u\| \geq 1)$ and let $(W^m_i)_{i \geq 1 }$ be a sequence of independent and identically distributed random vectors of $\bbr^d$ such that, for all $A \in \mathcal{B}(\bbr^d)$, 
\begin{align*}
\mathbb{P} \left(W^m_1 \in A\right) = \dfrac{\nu_{\alpha,m}( A \cap \{u \in \bbr^d:\, \|u\| \geq 1\} )}{\ell_m}.
\end{align*}
Finally, let $N_m$ be a Poisson random variable with mean $\ell_m$ independent of $\left(W_i^m\right)_{i \geq 1}$. Then, 
\begin{align*}
Z_{\alpha,m} =_{\mathcal{L}} \sum_{i = 1}^{N_m} W^m_i - \ell_m \bbe W_1^m.
\end{align*}
First, observe that, for all $m \geq 2$,
\begin{align*}
\ell_m \bbe W_1^m = \int_{\|x\| \geq 1} x \nu_{\alpha, m}(dx) = 0 
\end{align*}
since $\lambda_1$ is symmetric.  Then, let $\beta \in [1,\alpha)$. By Minkowski's inequality,
\begin{align*}
\|X_{\alpha,m}\|_{\beta} \leq \|Y_{\alpha,1}\|_{\beta} + \|Z_{\alpha,m}\|_{\beta} \leq \|Y_{\alpha,1}\|_{\beta} + \left\| \sum_{i = 1}^{N_m} W^m_i \right\|_{\beta},
\end{align*}
where $\|X\|_{\beta} = \left(\bbe \|X\|^\beta\right)^{\frac{1}{\beta}}$, for all $X$ random vector of $\bbr^d$.  Finally, 
\begin{align*}
\left\| \sum_{i = 1}^{N_m} W^m_i \right\|^\beta_{\beta} = \bbe \left\|\sum_{i = 1}^{N_m} W^m_i \right\|^\beta = \sum_{k = 0}^{+\infty} \dfrac{\ell_m^k}{k!} e^{-\ell_m} \bbe \left\|\sum_{i = 1}^{k} W^m_i \right\|^\beta.
\end{align*}
Now, by convexity (since $\beta \in [1,\alpha)$),  for all $k \geq 1$, 
\begin{align*} 
\bbe \left\|\sum_{i = 1}^{k} W^m_i \right\|^\beta & \leq k^{\beta-1} \bbe \sum_{i = 1}^k \|W_i^m\|^\beta = k^\beta \bbe \|W_1^m\|^\beta = \dfrac{k^\beta}{\ell_m} \int_{\|u\|\geq 1} \|u\|^\beta \nu_{\alpha,m}(du) , \\
& \leq \dfrac{k^\beta}{\ell_m} \int_{\|u\| \geq 1} \|u\|^\beta \nu_\alpha(du).
\end{align*}
Thus, for all $k \geq 1$, 
\begin{align*}
\dfrac{\ell_m^k}{k!} e^{-\ell_m} \bbe \left\|\sum_{i = 1}^{k} W^m_i \right\|^\beta & \leq e^{- \ell_m}\dfrac{ \ell_m^{k-1} k^\beta}{k!} \int_{\|u\| \geq 1} \|u\|^\beta \nu_\alpha(du) , \\
& \leq \dfrac{ \ell^{k-1} k^\beta}{k!} \int_{\|u\| \geq 1} \|u\|^\beta \nu_\alpha(du),
\end{align*}
since $0 < \ell_m \leq \ell = \nu_{\alpha}( u \in \bbr^d:\, \|u\| \geq 1)$. The right-hand side of the previous inequality is uniform in $m$ and is the general term of a convergent series.  Thus, 
\begin{align*}
\sup_{m \geq 2} \|X_{\alpha,m}\|_{\beta} < +\infty.
\end{align*}
So, for all $\xi \in \bbr^d$ with $\xi \ne 0$ and all $R>0$,
\begin{align}\label{eq:truncation_R}
\dfrac{i}{\Gamma\left(1 - \frac{1}{\alpha}\right)} \int_{\bbr^d} y \rho_{\alpha,R}(\langle y;\xi \rangle) \mu_{\alpha}(dy) = \dfrac{1}{\Gamma(1- \frac{1}{\alpha})} \int_0^R \hat{\mu}_{\alpha} \left(t^{\frac{1}{\alpha}} \xi\right) \tau_\alpha(\xi) \frac{dt}{t^{\frac{1}{\alpha}}}.
\end{align}
In order to conclude the proof, let us pass to the limit ($R \rightarrow +\infty$) in the previous equality.  Thanks to the L\'evy's representation of the characteristic function $\hat{\mu}_{\alpha}$,  for all $\xi \in \bbr^d$ and all $t >0$,
\begin{align*}
\hat{\mu}_{\alpha} \left(t^{\frac{1}{\alpha}} \xi\right) = \exp \left( - t \int_{\mathbb{S}^{d-1}} |\langle \xi ; y \rangle|^\alpha \lambda_1(dy)\right).
\end{align*}
Since $\nu_\alpha$ is non-degenerate,  for all $\xi \in \bbr^d$ and all $t>0$ 
\begin{align*}
\left| \hat{\mu}_{\alpha} \left(t^{\frac{1}{\alpha}} \xi\right)\right| \leq \exp \left(- t c_0(\alpha,d) \|\xi\|^\alpha\right).
\end{align*}
Thus, by the Lebesgue dominated convergence theorem, for all $\xi \in \bbr^d$ with $\xi \ne 0$
\begin{align*}
\underset{R \rightarrow +\infty}{\lim} \int_0^R \hat{\mu}_{\alpha} \left(t^{\frac{1}{\alpha}} \xi\right) \tau_\alpha(\xi) \frac{dt}{t^{\frac{1}{\alpha}}} & = \int_{0}^{+\infty} \exp \left( - t \int_{\mathbb{S}^{d-1}} |\langle \xi;y \rangle|^\alpha \lambda_1(dy)\right) \frac{dt}{t^{\frac{1}{\alpha}}} \tau_\alpha(\xi) \\
& = \Gamma\left(1 - \frac{1}{\alpha}\right)\dfrac{\tau_\alpha(\xi)}{\left(\int_{\mathbb{S}^{d-1}} |\langle \xi ; u \rangle|^\alpha \lambda_1(du)\right)^{\frac{\alpha-1}{\alpha}}}.
\end{align*}
Finally, for the left-hand side of \eqref{eq:truncation_R},  note that 
\begin{align}\label{eq:uniform_R}
\underset{R \rightarrow +\infty}{\lim} \rho_{\alpha,R}(x) = \rho_\alpha(x), \quad x\in \bbr^d,\quad \underset{R \geq 1}{\sup} \, \| \rho_{\alpha,R}\|_{\infty,\bbr} <+\infty.
\end{align}
The uniform bound follows from \cite[Exercice $4.1.1$ page $263$]{G08}. Then, for all $\xi \in \bbr^d$ with $\xi \ne 0$,
\begin{align*}
\dfrac{i}{\Gamma\left(1 - \frac{1}{\alpha}\right)} \int_{\bbr^d} y \rho_{\alpha}(\langle y;\xi \rangle) \mu_{\alpha}(dy) = \dfrac{\tau_\alpha(\xi)}{\left(\int_{\mathbb{S}^{d-1}} |\langle \xi ; u \rangle|^\alpha \lambda_1(du)\right)^{\frac{\alpha-1}{\alpha}}}.
\end{align*}
This concludes the proof of the lemma. 
\end{proof}
\noindent
To end this Appendix, let us provide two results on the ``carr\'e de Mehler". These results are partially based on the setting and the results obtained in \cite{Peszat} to represent the stable OU semigroup as a second quantized operator on a Poisson Fock space and used to study its algebraic and spectral properties. For the sake of completeness, let us recall briefly this setting (see also \cite{LP11}): let $(\Omega, \mathcal{F},\mathbb{P})$ be a probability space and let $\mathbb{Z}_+(E)$ be the space of integer-valued $\sigma$-finite measures on $(E, \mathcal{B})$ endowed with the $\sigma$-field $\mathcal{G}$ making the mappings $\xi \mapsto \xi(B)$ measurable for all $B \in \mathcal{B}$. Then, let $\Pi$ be a Poisson random measure on $E$ with $\sigma$-finite intensity measure $\lambda$. $\Pi$ induces a probability measure $\bbp_\Pi$ on $\left(\mathbb{Z}_+(E),\mathcal{G}\right)$ and a sigma field $\sigma(\Pi)$ on $\Omega$. Next, let $L^2(\bbp_\Pi)$ be the space of equivalence classes of measurable real-valued functionals such that $ \bbe F^2(\Pi)<+\infty$. For $y \in E$, the difference operator $D_y$ is defined, for all measurable $F: \mathbb{Z}_+(E) \rightarrow \bbr$ and all $\xi \in \mathbb{Z}_+(E)$, by $ D_y(F)(\xi):=F(\xi +\delta_y)-F(\xi)$, where $\delta_y$ is the Dirac measure at point $y\in E$. For all $n\geq 2$ and all $(y_1,\dots, y_n) \in E^n$, let us define the iterated difference operator $D^{n}_{y_1,\dots, y_n}$ by the following formula, for all measurable $F: \mathbb{Z}_+(E) \rightarrow \bbr$ and all $\xi \in \mathbb{Z}_+(E)$, $D^n_{y_1,\dots, y_n}(F)(\xi):= \sum_{I\subset \{1,\dots, n\}} (-1)^{n-\# I} F\left(\xi +\sum_{i\in I}\delta_{y_i}\right)$,
where $\# I$ denotes the cardinality of $I$. Clearly, this operator is symmetric in $y_1,\dots, y_n$. Next, for $n \geq 1$, let $L^2(E^n, \lambda^n)$ be the space of (equivalence classes of) $\mathcal{B}^n$-measurable real-valued functions on $(E^n, \lambda^n)$ which are moreover square-integrable with respect to $\lambda^n$ and let $L^2_{(s)}(E^n, \lambda^n)$ be the closed subspace of symmetric functions of $L^2(E^n, \lambda^n)$. 
Then, let us define the following sequence of bounded linear operators from $L^2(\bbp_\Pi)$ to $L^2_{(s)}(E^n, \lambda^n)$, $n \geq 1$, for all measurable $F\in L^2(\bbp_\Pi)$, all $n \geq 1$ and all $(y_1, \dots, y_n)\in E^n$,
$T^0(F)= \bbe F(\Pi)$ and $T^n(F)(y_1,\dots, y_n)= \bbe D^n_{y_1,\dots, y_n}F(\Pi)= \int_{\mathbb{Z}_+(E)} D^n_{y_1,\dots, y_n}F(\xi) \bbp_\Pi(d\xi)$. Moreover, note that, for all $n\geq 1$, 
$$\|T^n\|_{L^2(\bbp_\Pi) \rightarrow L^2_{(s)}(E^n, \lambda^n)}:= \underset{F \in L^2(\bbp_\Pi)\, \bbe F^2(\Pi)\ne 0}{\sup} \dfrac{\left\|T^n(F)\right\|_{L^2(E^n, \lambda^n)}}{(\bbe F^2(\Pi))^{\frac{1}{2}}}\leq \sqrt{n!}.$$ 
Now, let us introduce the chaos decomposition on the Poisson space. For this purpose, let $\mathcal{H}_0=\bbr$ and let $\mathcal{H}_n:=\{I_n(f):\quad f\in L^2_{(s)}(E^n,\lambda^n)\}$, $n\geq 1$, where $I_n$ is the 
multiple Wiener-It\^o integral with respect to the compensated measure $\tilde{\Pi}=\Pi-\lambda$.  
Let $P_n$ be the orthogonal projection of $L^2(\Omega, \sigma(\Pi), \bbp)$ into $\mathcal{H}_n$. Then,
\begin{align*}
L^2(\Omega, \sigma(\Pi), \bbp)= \bigoplus_{n=0}^{+\infty} \mathcal{H}_n,
\end{align*}
with $P_0(F(\Pi))=\bbe F(\Pi)$ and $P_n(F(\Pi)):=\frac{1}{n!} I_n(T^n(F))$, $n\geq 1$ and $F\in L^2(\bbp_\Pi)$. Thanks to \cite[Theorem $1.3$]{LP11}, for all $F \in L^2(\bbp_\Pi)$, $F(\Pi)=\sum_{n=0}^{+\infty}\frac{1}{n!} I_n(T^n (F))$,
where the series converges in $L^2(\Omega, \sigma(\Pi), \bbp)$.  
Moreover, from \cite[Theorem $1.1$]{LP11}, for all $F,G \in L^2(\bbp_\Pi)$, one has
\begin{align*}
\bbe F(\Pi)G(\Pi):=T^0(F)T^0(G)+\sum_{n=1}^{+\infty}\frac{1}{n!} \langle T^n(F), T^n(G) \rangle_{L^2(E^n, \lambda^n)}.
\end{align*}
Finally, for any linear operator $R$ from $E$ to $E$, let us define, for all $n\geq 1$, all $f$ real-valued function on $E^n$ and all $y_1, \dots, y_n \in E^n$, $\rho^n_R (f)(y_1, \dots, y_n)= f(Ry_1, \dots, Ry_n)$. Now, as in \cite[Section $4$]{Peszat}, it is possible to define the second quantization of $R$ as soon as $\rho_R$ is 
a contraction on $L^2(E, \lambda)$. Then, one has, for all $F \in L^2(\bbp_\Pi)$,
\begin{align*}
\Gamma(R)F(\Pi):=\sum_{n=0}^{+\infty}\frac{1}{n!} I_n(\rho_R^n(T^n (F))).
\end{align*}
Note that $\Gamma(R)$ is a contraction acting on $L^2(\bbp_\Pi)$. To finish, let $\Pi$ be a Poisson random measure on $\bbr^d$ with intensity measure $\nu_\alpha$. For any $\xi \in \mathbb{Z}_+(E)$, 
\begin{align*}
\overline{\xi}(dx):=\xi(dx) \bbone_{\|x\|>1}+\left(\xi(dx)-\nu_\alpha(dx)\right)\bbone_{\|x\|\leq 1}.
\end{align*}
Next, set
\begin{align}\label{eq:defX}
X_\alpha=b+ \int_{\bbr^d} x \overline{\Pi}(dx),
\end{align}
Then, $X_\alpha \sim \mu_\alpha$, for some $b\in \bbr^d$ well chosen. Let $j$ be the operator acting on $L^2(\mu_\alpha)$ and defined, for all $f \in L^2(\mu_\alpha)$ and all $\xi \in  \mathbb{Z}_+(E)$, by $j(f)(\xi)=f\left(b+\int_{\bbr^d} x \overline{\xi}(dx)\right)$. Note that $j$ is an isometry from $L^2(\mu_\alpha)$ to $L^2(\bbp_\Pi)$ with $\bbe j(f)(\Pi)^2= \bbe f(X)^2= \int_{\bbr^d} f(x)^2 \mu_\alpha(dx)$. Combined with the previous chaos decomposition, it follows that, for all $f\in L^2(\mu_\alpha)$, 
\begin{align*}
j(f)(\Pi):= \sum_{n=0}^{+\infty} \frac{1}{n!} I_n(T^n(j(f))).
\end{align*}
Observe that the hypothesis $(H.1)$ to $(H.5)$ of \cite[Section $5$]{Peszat} are satisfied, so that, from \cite[Corollary $6.2$]{Peszat}, the following representation holds true on $L^2(\Omega, \sigma(\Pi), \bbp)$
\begin{align}\label{eq:ChaoticExpansion}
P^{\nu_\alpha}_t(f)(X_\alpha):=\sum_{n=0}^{+\infty} \frac{1}{n!} I_n(\rho^n_{e^{-t}}(T^n(j(f)))).
\end{align}
Based on the previous representation, let us give an expression for the $L^2(\mu_\alpha)$ dual semigroup of $(P^{\nu_\alpha}_t)_{t\geq 0}$ (denoted by $((P^{\nu_\alpha}_t)^*)_{t\geq 0}$) and let us prove that both semigroups commute. Then, as a byproduct, one obtains the generator of the product semigroup by a standard result (see, e.g., \cite[Theorem $1$]{Trotter59}).
 
\begin{thm}\label{def:squared_Mehler_L2}
Let $\alpha \in (0,2)$ and let $\nu_\alpha$ be a non-degenerate symmetric L\'evy measure on $\bbr^d$ satisfying \eqref{eq:scale}. Let $\mu_\alpha$ be the non-degenerate $\alpha$-stable probability measure associated with $\nu_\alpha$ and defined by \eqref{def:stable}. Let $(P^{\nu_\alpha}_t)_{t\geq 0}$ be the $L^2(\mu_\alpha)$-extension of the semigroup of operators defined by \eqref{eq:StOUSM} and denote by $((P^{\nu_\alpha}_t)^*)_{t\geq 0}$ its dual semigroup.  
Then, for all $s,t \geq 0$, the operators $P^{\nu_\alpha}_t$ and $(P^{\nu_\alpha}_s)^*$ commute 
and the family of operators $(\mathcal{P}_t)_{t\geq 0}=(P^{\nu_\alpha}_{t/\alpha} \circ (P^{\nu_\alpha}_{t/\alpha})^*)_{t\geq 0}$ is a $C_0$-semigroup of contractions on $L^2(\mu_\alpha)$ 
with generator the closure of the sum of the 
generators of $(P^{\nu_\alpha}_{t/\alpha})_{t\geq 0}$ and of $((P^{\nu_\alpha}_{t/\alpha})^*)_{t\geq 0}$. 
Moreover, for all $t\geq 0$ and all $f\in L^2(\mu_\alpha)$,
\begin{align*}
\mathcal{P}_t(f)(X_\alpha):=\sum_{n=0}^{+\infty} \frac{e^{-n t}}{n!} I_n(T^n(j(f))),
\end{align*}
with $X_\alpha\sim \mu_\alpha$.  
In addition, the domain of $\mathcal{L}$, the $L^2(\mu_\alpha)$-generator 
of $(\mathcal{P}_t)_{t\geq 0}$, is given by
\begin{align*}
D(\mathcal{L})=\{f\in L^2(\mu_\alpha):\quad \sum_{n=1}^{+\infty} \frac{n^2 }{n!}\|T^n(j(f))\|^2_{L^2(\bbr^{n d}, \nu_\alpha^{\otimes n})}<+\infty\},
\end{align*}
and, for all $f\in D(\mathcal{L})$,
\begin{align*}
\mathcal{L}(f)(X_\alpha)=-\sum_{n=1}^{+\infty}\dfrac{n}{n!} I_n(T^n(j(f))),
\end{align*}
where the convergence is in $L^2(\Omega, \sigma(\Pi), \bbp)$. Finally, 
the $L^2(\mu_\alpha)$-generator of the semigroup $(P^{\nu_\alpha}_t)_{t\geq 0}$ is normal, i.e., 
\begin{align*}
\mathcal{L}^{\alpha} (\mathcal{L}^{\alpha})^* = (\mathcal{L}^{\alpha})^* \mathcal{L}^{\alpha}, \quad \mathcal{D}(\mathcal{L}^{\alpha}) = \mathcal{D} ((\mathcal{L}^{\alpha})^*).
\end{align*}
\end{thm}
 
\begin{proof}
Let us start with a computation on the adjoint semigroup of $(P^{\nu_\alpha}_t)_{t\geq 0}$.  
By definition and using the chaotic expansion \eqref{eq:ChaoticExpansion}, 
for all $f,g \in L^2(\mu_\alpha)$ and all $t\geq 0$,
\begin{align*}
\langle (P^{\nu_\alpha}_t)^*(g);f \rangle:=\langle P^{\nu_\alpha}_t(f);g \rangle &
= \bbe P^{\nu_\alpha}_t(f)(X_\alpha)g(X_\alpha)\\
&= \bbe j(P^{\nu_\alpha}_t(f))(\Pi)(j(g))(\Pi)\\
&= \sum_{n=0}^{+\infty} \frac{1}{n!} \langle \rho^n_{e^{-t}} T^n(j(f));T^n(j(g)) 
\rangle_{L^2(\bbr^{nd},\nu_\alpha^{\otimes n} )}.
\end{align*}
But, for all $n\geq 1$ and all $t\geq 0$,
\begin{align*}
\langle \rho^n_{e^{-t}} T^n(j(f));T^n(j(g)) 
\rangle_{L^2(\bbr^{nd},\nu_\alpha^{\otimes n} )}&:= \int_{\bbr^{n d}} T^n(j(f))(e^{-t}y_1,\dots,e^{-t}y_n)\\
&\quad\quad\quad \times T^n(j(g))(y_1,\dots,y_n) \nu_\alpha^{\otimes n}(dy_1,\dots, dy_n), 
\end{align*}
so changing variables in the radial coordinates, $r_i=e^{t}\rho_i$, for all $i \in \{1,\dots, n\}$, 
\begin{align*}
\langle \rho^n_{e^{-t}} T^n(j(f));T^n(j(g)) \rangle_{L^2(\bbr^{nd},\nu_\alpha^{\otimes n} )}&= e^{-n\alpha t} \int_{\bbr^{n d}} T^n(j(f))(y_1,\dots,y_n)\\
&\quad\quad\times T^n(j(g))(e^t y_1,\dots, e^t y_n) \nu_\alpha^{\otimes n}(dy_1,\dots, dy_n)\\
&= e^{-n\alpha t} \langle T^n(j(f)); \rho^n_{e^{t}} T^n(j(g)) \rangle_{L^2(\bbr^{nd},\nu_\alpha^{\otimes n} )}.
\end{align*}
Therefore, 
\begin{align*}
\langle (P^{\nu_\alpha}_t)^*(g);f \rangle= \sum_{n=0}^{+\infty} \frac{1}{n!} e^{-n\alpha t} \langle T^n(j(f)); \rho^n_{e^{t}} T^n(j(g)) \rangle_{L^2(\bbr^{nd},\nu_\alpha^{\otimes n} )}, 
\end{align*}
i.e.,  for all $g \in L^2(\mu_\alpha)$, 
\begin{align*}
(P^{\nu_\alpha}_t)^*(g)(X_\alpha)= \sum_{n=0}^{+\infty} \frac{e^{-n\alpha t}}{n!} I_n(\rho^n_{e^{t}} T^n(j(g)) ),
\end{align*}
so that, for all $g \in L^2(\mu_\alpha)$,
\begin{align*}
(P^{\nu_\alpha}_{t/\alpha})^*(g)(X_\alpha)= \sum_{n=0}^{+\infty} \frac{e^{-n t}}{n!} I_n(\rho^n_{e^{t/\alpha}} T^n(j(g)) ).
\end{align*} 
Let $s,t\geq 0$ and let $f,g \in L^2(\mu_\alpha)$, then 
\begin{align*}
\langle P^{\nu_\alpha}_t \circ (P^{\nu_\alpha}_s)^*(f); g \rangle&= \langle (P^{\nu_\alpha}_s)^*(f); (P^{\nu_\alpha}_t)^*(g) \rangle \\
&= \sum_{n=0}^{+\infty} \frac{e^{-n\alpha (t+s)}}{n!} \langle \rho^n_{e^{s}}T^n(j(f)); \rho^n_{e^{t}}T^n(j(g)) \rangle_{L^2(\bbr^{nd},\nu_\alpha^{\otimes n} )}, 
\end{align*}
and observe that, for all $n\geq 1$ and all $s,t\geq 0$, 
\begin{align*}
\langle \rho^n_{e^{s}}T^n(j(f)); \rho^n_{e^{t}}T^n(j(g)) \rangle_{L^2(\bbr^{nd},\nu_\alpha^{\otimes n} )}&= \int_{\bbr^{nd}} T^n(j(f))(e^s y_1,\dots, e^s y_n)\\
&\quad\quad \times T^n(j(g))(e^t y_1,\dots, e^t y_n)\nu_\alpha^{\otimes n}(dy_1,\dots, dy_n)\\
&=e^{n \alpha s} \int_{\bbr^{nd}} T^n(j(f))( y_1,\dots, y_n)\\
&\quad\quad \times T^n(j(g))(e^{t-s} y_1,\dots, e^{t-s} y_n)\nu_\alpha^{\otimes n}(dy_1,\dots, dy_n),  
\end{align*}
thus, 
\begin{align*}
\langle P^{\nu_\alpha}_t \circ (P^{\nu_\alpha}_s)^*(f); g \rangle &= \sum_{n=0}^{+\infty} \frac{e^{-n \alpha t}}{n!} \langle T^n(j(f)); \rho^n_{e^{t-s}}T^n(j(g)) \rangle_{L^2(\bbr^{nd},\nu_\alpha^{\otimes n} )}\\
&=\sum_{n=0}^{+\infty} \frac{e^{-n \alpha t}}{n!} \langle T^n(j(f)); \rho^n_{e^{t-s}}T^n(j(g)) \rangle_{L^2(\bbr^{nd},\nu_\alpha^{\otimes n} )}.
\end{align*}
Similarly, for all $s,t\geq 0$ and all $f,g \in L^2(\mu_\alpha)$,
\begin{align*}
\langle  (P^{\nu_\alpha}_s)^* \circ P^{\nu_\alpha}_t(f); g \rangle&=\sum_{n=0}^{+\infty} \frac{1}{n!} \langle \rho^n_{e^{-t}}T^n(j(f)); \rho^n_{e^{-s}} T^n(j(g)) \rangle_{L^2(\bbr^{n d}, \nu_\alpha^{\otimes n})}\\
&=\sum_{n=0}^{+\infty} \frac{e^{-n\alpha t}}{n!} \langle T^n(j(f)); \rho^n_{e^{t-s}} T^n(j(g)) \rangle_{L^2(\bbr^{n d}, \nu_\alpha^{\otimes n})}.
\end{align*}
Therefore, for all $s,t \geq 0$ and all $f,g \in L^2(\mu_\alpha)$,
\begin{align*}
\langle  (P^{\nu_\alpha}_s)^* \circ P^{\nu_\alpha}_t(f); g \rangle=\langle P^{\nu_\alpha}_t \circ (P^{\nu_\alpha}_s)^*(f); g \rangle, 
\end{align*}
in other words, $(P^{\nu_\alpha}_s)^* \circ P^{\nu_\alpha}_t= P^{\nu_\alpha}_t \circ (P^{\nu_\alpha}_s)^*$, 
for all $s,t\geq 0$.  By \cite[Theorem $1$]{Trotter59}, it follows that $((P^{\nu_\alpha}_{t/\alpha})^* \circ P^{\nu_\alpha}_{t/\alpha})_{t\geq 0}$ is a $C_0$-semigroup of contractions on $L^2(\mu_\alpha)$ with generator given by the closure of the sum of the generators of the $C_0$-semigroups of contractions $((P^{\nu_\alpha}_{t/\alpha})^*)_{t\geq 0}$ and $(P^{\nu_\alpha}_{t/\alpha})_{t\geq 0}$. Now, for all $f,g\in L^2(\mu_\alpha)$, one has, for all $t\geq 0$,
\begin{align*}
\langle \mathcal{P}_t(f);g\rangle&=\langle P^{\nu_\alpha}_{t/\alpha}(f);P^{\nu_\alpha}_{t/\alpha}(g) \rangle\\
&=\sum_{n=0}^{+\infty} \frac{e^{-n t}}{n!} \langle T^n(j(f)); T^n(j(g)) \rangle_{L^2(\bbr^{n d}, \nu_\alpha^{\otimes n})},
\end{align*}
i.e., the following representation holds true
\begin{align*}
\mathcal{P}_t(f)(X_\alpha)=\sum_{n=0}^{+\infty} \frac{e^{-n t}}{n!} I_n(T^n(j(f))).
\end{align*}
The domain of the $L^2(\mu_\alpha)$-generator is 
\begin{align*}
D(\mathcal{L}):=\left\{f\in L^2(\mu_\alpha):\quad \underset{t\rightarrow 0^+}{\lim} \dfrac{\mathcal{P}_t(f)-f}{t} \operatorname{\quad exists\quad in \quad L^2(\mu_\alpha)} \right\},  
\end{align*}
which is next characterized via the chaos expansion on the Poisson space.  
Let $f \in L^2(\mu_\alpha)$ be such that
\begin{align*}
\sum_{n=1}^{+\infty} \frac{n^2 }{n!} \|T^n(j(f))\|^2_{L^2(\bbr^{nd}, \nu_\alpha^{\otimes n})}<+\infty, 
\end{align*}
and let $L(f)$ be following series (converging in $L^2(\Omega, \sigma(\Pi),\mathbb{P})$), 
\begin{align*}
L(f)=-\sum_{n=1}^{+\infty}\frac{n}{n!} I_n(T^n(j(f))). 
\end{align*}
Then, for all $t>0$, by orthogonality, 
\begin{align}\label{eq:normL2}
\left\|\frac{\mathcal{P}_t(f)(X_\alpha)-f(X_\alpha)}{t}-L(f)\right\|^2_{L^2(\bbp)}=\sum_{n=0}^{+\infty} \left(\dfrac{e^{- n t}-1}{t}+n\right)^2\frac{\|T^n(j(f))\|^2_{L^2(\bbr^{n d}, \nu_\alpha^{\otimes n})}}{n!}, 
\end{align}
and the right-hand side of \eqref{eq:normL2} clearly converges to $0$ as $t\to 0$. 
Now, let $(t_n)_{n\geq 1}$ be a sequence of positive reals converging to $0$ as $n \to +\infty$, then 
$\left((\mathcal{P}_{t_n}(f)-f)/t_n\right)_{n\geq 1}$ is a Cauchy sequence in $L^2(\mu_\alpha)$ which is complete, therefore 
\begin{align*}
\left\{f\in L^2(\mu_\alpha):\quad \sum_{n=1}^{+\infty} \frac{n^2 }{n!}\|T^n(j(f))\|^2_{L^2(\bbr^{n d}, \nu_\alpha^{\otimes n})}<+\infty\right\} \subset D(\mathcal{L}).  
\end{align*}
Next, let $f\in D(\mathcal{L})$. Then, $\mathcal{L}(f)$ belongs to $L^2(\mu_\alpha)$ so that $\mathcal{L}(f)(X_\alpha)$ is a well-defined element in $L^2(\Omega, \sigma(\Pi),\mathbb{P})$ and let us compute its 
chaos expansion.  Since, in $L^2(\Omega, \sigma(\Pi), \bbp)$, 
\begin{align*}
\mathcal{L}(f)(X_\alpha):= \lim_{t\rightarrow 0^+} \frac{\mathcal{P}_t(f)(X_\alpha)-f(X_\alpha)}{t}, 
\end{align*}
by projecting on the $n$-th chaos, one gets, for all $n\geq 1$,
\begin{align*}
\frac{1}{n!} I_n(T^n j(\mathcal{L}(f)))=\frac{-n}{n!} I_n(T^n(j(f))), 
\end{align*}
This concludes the characterization of the domain of $\mathcal{L}$ by reversing our previous inclusion.   
To finish the proof of this theorem note that the normality of $\mathcal{L}^\alpha$ is a direct 
consequence of the commutativity 
of the operators $P^{\nu_\alpha}_t$ and $(P^{\nu_\alpha}_s)^*$, which is valid for all $s,t\geq 0$. 
\end{proof}

For all $f \in L^2(\mu_\alpha)$ and all $g\in \mathcal{D}(\mathcal{L})$, let 
\begin{align*}
\mathcal{E} (f,g)=\langle f; (-\mathcal{L})(g) \rangle_{L^2(\mu_\alpha)}= \sum_{n=1}^{+\infty} \frac{n}{n!} \langle T^n(j(g)); T^n(j(f)) \rangle_{L^2\left(\bbr^{n d}, \nu_\alpha^{\otimes n}\right)},
\end{align*}
and consider the linear subspace of $L^2(\mu_\alpha)$ given by:  
\begin{align*}
\mathcal{D}\left(\mathcal{E}\right):= \left\{f \in L^2(\mu_\alpha):\quad 
\sum_{n=1}^{+\infty} \frac{n }{n!} \|T^n(j(f))\|^2_{L^2(\bbr^{n d}, \nu_\alpha^{\otimes n})}<+\infty \right\}.
\end{align*}
Clearly, $\mathcal{D}(\mathcal{L})\subset\mathcal{D}\left(\mathcal{E}\right)$.  
Then, $\mathcal{D}\left(\mathcal{E}\right)$ is a dense linear subspace of $L^2(\mu_\alpha)$. 
Moreover, for all $f,g \in \mathcal{D}\left(\mathcal{E}\right)$, one has
\begin{align*}
\mathcal{E}(f,g)= \sum_{n=1}^{+\infty} \frac{n}{n!} \langle T^n(j(g)); T^n(j(f)) \rangle_{L^2\left(\bbr^{n d}, \nu_\alpha^{\otimes n}\right)}.
\end{align*}
This symmetric bilinear form $(\mathcal{E}, \mathcal{D}\left(\mathcal{E}\right))$ is clearly closed since $(\mathcal{L},\mathcal{D}(\mathcal{L}))$ is a non-positive self-adjoint operator on $L^2(\mu_\alpha)$.~As proved in Theorem \ref{def:squared_Mehler_L2}, the generator $\mathcal{L}$ is the closure of the sum of generators of the semigroups $(P^{\nu_\alpha}_{t})_{t\geq 0}$ and $((P^{\nu_\alpha}_{t})^*)_{t\geq 0}$ (divided by $\alpha$).Then, for all $(f,g) \in \mathcal{D}(\mathcal{L}^{\alpha}) \subset \mathcal{D}(\mathcal{L})$, 
\begin{align*}
\alpha\, \mathcal{E} (f,g)=\alpha\, \langle f; -\mathcal{L}(g) \rangle_{L^2(\mu_\alpha)}=\langle f; -(\mathcal{L}^{\alpha}+(\mathcal{L}^{\alpha})^*)(g)\rangle_{L^2(\mu_\alpha)},
\end{align*}
so that, for all $f\in \mathcal{D}(\mathcal{L}^{\alpha})$, 
$\mathcal{E} (f,f)=\frac{2}{\alpha} \langle -\mathcal{L}^{\alpha}(f); f \rangle_{L^2(\mu_\alpha)}$.  
Finally, the next proposition ensures that $\mathcal{S}(\bbr^d)\subset \mathcal{D}(\mathcal{L}^{\alpha,p})$, for all $p \in (1,+\infty)$, where $(\mathcal{L}^{\alpha,p},\mathcal{D}(\mathcal{L}^{\alpha,p}))$ is the $L^p(\mu_\alpha)$-generator of the semigroup $(P^{\nu_\alpha}_t)_{t\geq 0}$. 

\begin{prop}\label{prop:first_inclusion_smart_proof}
Let $\alpha\in (0,2)$, let $\nu_\alpha$ be a non-degenerate symmetric L\'evy measure on $\bbr^d$ satisfying \eqref{eq:scale} and let $p \in (1,+\infty)$. Then, $\mathcal{S}(\bbr^d) \subset \mathcal{D}(\mathcal{L}^{\alpha,p})$ and, for all $f \in \mathcal{S}(\bbr^d)$, 
\begin{align*}
\mathcal{L}^{\alpha,p}(f) = \mathcal{L}^\alpha(f), 
\end{align*}
with,
\begin{align*}
\mathcal{L}^\alpha(f)(x) = \left\{
    \begin{array}{ll}
       -\langle x;\nabla(f)(x) \rangle + \int_{\bbr^d} \langle \nabla(f)(x+u) - \nabla(f)(x) ; u \rangle \nu_\alpha(du)  & \mbox{if } \alpha \in (1,2) \\
       - \langle x; \nabla(f)(x) \rangle + \int_{\bbr^d} \left(f(x+u)-f(x) - \langle\nabla(f)(x);u\rangle\bbone_{\|u\|\leq 1}\right) \nu_1(du)  & \mbox{if } \alpha = 1\\
       -\langle x; \nabla(f)(x) \rangle + \alpha \int_{\bbr^d} (f(x+u)-f(x))\nu_\alpha(du)  & \mbox{if } \alpha \in (0,1). 
    \end{array}
\right.
\end{align*} 
\end{prop}

\begin{proof}
\textit{Step 1}: Let us start with the case $\alpha \in (1,2)$. Let $f\in \mathcal{S}(\bbr^d)$. As in \cite[Step $1$ of Proposition $3.5$]{AH19_2}, for all $t>0$ and all $x\in \bbr^d$,
\begin{align}\label{eq:frac_heat_equation_regular}
\dfrac{d}{dt}\left(P^{\nu_\alpha}_t(f)(x)\right) = \mathcal{L}^\alpha\left(P^{\nu_\alpha}_t(f)\right)(x),
\end{align}
Moreover, using to the commutation relation $\nabla (P^{\nu_\alpha}_t(f))= e^{-t} P_t^{\nu_\alpha}(\nabla f)$, for all $t>0$ and all $x\in \bbr^d$, 
\begin{align}\label{eq:TempDeriv}
\dfrac{d}{dt}\left(P^{\nu_\alpha}_t(f)(x)\right) = -e^{-t} \langle x; P_t^{\nu_\alpha}(\nabla f)(x) \rangle +  e^{-t} \int_{\bbr^d} \langle P_t^{\nu_\alpha}(\nabla(f))(x+u)-P_t^{\nu_\alpha}(\nabla(f))(x);u\rangle \nu_\alpha(du).
\end{align}
From \eqref{eq:TempDeriv}, it follows that $t\rightarrow \frac{d}{dt}\left(P^{\nu_\alpha}_t(f)(x)\right)$ is continuous. Thus, for all $t>0$ and all $x\in \bbr^d$,
\begin{align*}
\dfrac{P_t^{\nu_\alpha}(f)(x)-f(x)}{t}&=\frac{1}{t} \int_0^t \dfrac{d}{ds}\left(P_s^{\nu_\alpha}(f)(x)\right)ds\\
& = \frac{1}{t}  \int_0^t  \mathcal{L}^\alpha\left(P^{\nu_\alpha}_s(f)\right)(x) ds.
\end{align*}
At this point, let us prove that, for all $f \in \mathcal{S}(\bbr^d)$, all $x\in \bbr^d$ and all $t\geq 0$,
\begin{align*}
\mathcal{L}^\alpha\left(P^{\nu_\alpha}_t(f)\right)(x) = P^{\nu_\alpha}_t\left(\mathcal{L}^\alpha(f)\right)(x).
\end{align*}
This can be seen as a consequence of the Bismut-type formula proved in Proposition \ref{prop:Bismut}. Indeed, in regard to the drift term in $P^{\nu_\alpha}_t(\mathcal{L}^\alpha(f))(x)$, for all $f \in \mathcal{S}(\bbr^d)$, all $x\in \bbr^d$ and all $t>0$,  
\begin{align*}
(I):=\int_{\bbr^d} \langle xe^{-t} + (1-e^{-\alpha t})^{\frac{1}{\alpha}} y &; \nabla(f) \left(xe^{-t} + (1-e^{-\alpha t})^{\frac{1}{\alpha}} y\right)\rangle \mu_\alpha(dy)  = \langle x; \nabla(P^{\nu_\alpha}_t(f))(x) \rangle \\
& \quad\quad\quad + \left(1-e^{-\alpha t}\right)^{\frac{1}{\alpha}} \int_{\bbr^d} \langle y ; \nabla(f)\left(xe^{-t}+(1-e^{-\alpha t})^{\frac{1}{\alpha}}y\right) \rangle \mu_\alpha(dy) , \\
& = \langle x; \nabla(P^{\nu_\alpha}_t(f))(x) \rangle +e^{(\alpha-1)t} (1-e^{-\alpha t}) \dfrac{e^{-(\alpha-1)t}}{(1-e^{-\alpha t})^{1-\frac{1}{\alpha}}} \\
&\quad\quad\quad \times \int_{\bbr^d} \langle y ; \nabla(f)\left(xe^{-t} + (1-e^{-\alpha t})^{\frac{1}{\alpha}}y\right) \rangle \mu_\alpha(dy) , \\
& = \langle x; \nabla(P^{\nu_\alpha}_t(f))(x) \rangle +e^{(\alpha-1)t} (1-e^{-\alpha t}) \sum_{j=1}^d D_j^{\alpha-1}\left(P_t^{\nu_\alpha}(\partial_j(f))\right)(x) , \\
& =  \langle x; \nabla(P^{\nu_\alpha}_t(f))(x) \rangle + (e^{\alpha t}-1) \sum_{j=1}^d D_j^{\alpha-1}\circ \partial_j\left(P_t^{\nu_\alpha}(f)\right)(x) ,
\end{align*}
As for the non-local term of $P^{\nu_\alpha}_t(\mathcal{L}^\alpha(f))(x)$,
\begin{align*}
(II):= \int_{\bbr^d} \sum_{j=1}^{d} D^{\alpha-1}_j\partial_j(f) (xe^{-t}+(1-e^{-\alpha t})^{\frac{1}{\alpha}}y) \mu_\alpha(dy).
\end{align*}
Using the commutation formulae with respect to the operators $D^{\alpha-1}_j$ and $\partial_j$,
\begin{align*}
(II)= \sum_{j=1}^d e^{(\alpha - 1)t} D_j^{\alpha-1} \left( P^{\nu_\alpha}_t(\partial_j(f))\right)(x) = e^{\alpha t} \sum_{j=1}^{d} (D_j^{\alpha-1} \partial_j P^{\nu_\alpha}_t(f))(x).
\end{align*}
Thus, 
\begin{align*}
P_t^{\nu_\alpha}(\mathcal{L}^\alpha(f))(x) = - \langle x ; \nabla(P^{\nu_\alpha}_t(f))(x) \rangle + \sum_{j=1}^d D_j^{\alpha-1} \partial_j(P^{\nu_\alpha}_t(f))(x) = \mathcal{L}^\alpha(P^{\nu_\alpha}_t(f))(x).
\end{align*}
Then, for all $f \in \mathcal{S}(\bbr^d)$, all $x\in \bbr^d$ and all $t>0$,
\begin{align*}
\dfrac{P_t^{\nu_\alpha}(f)(x)-f(x)}{t} = \frac{1}{t}  \int_0^t  P^{\nu_\alpha}_s(\mathcal{L}^\alpha(f))(x) ds.
\end{align*}
Since $f\in \mathcal{S}(\bbr^d)$, since
\begin{align*}
\|\mathcal{L}^\alpha(f)\|_{\infty} \leq \|\langle x; \nabla(f) \rangle\|_{\infty} + \|\operatorname{Hess}(f)\|_{\infty} \int_{\mathcal{B}(0,1)} \|u\|^2 \nu_\alpha(du) + 2 \|\nabla(f)\|_{\infty} \int_{\mathcal{B}(0,1)^c} \|u\| \nu_\alpha(du), 
\end{align*}
(here $\operatorname{Hess}(f)(x)$ is the Hessian matrix of $f$ at $x\in \bbr^d$ and $\|\operatorname{Hess}(f)\|_{\infty}$ is given by
$$
\|\operatorname{Hess}(f)\|_{\infty} := \underset{x\in \bbr^d}{\sup} \underset{u \in \bbr^d,\, \|u\|=1}{\sup}\|\operatorname{Hess}(f)(x)(u)\|),
$$
and since $P_s^{\nu_\alpha}$ is contractive on $L^\infty(\mu_\alpha)$, the end of the first step follows easily.\\ 
\\
\textit{Step 2}: Next, let us consider the case $\alpha \in (0,1)$. From \cite[Theorem $4.1$]{AH20_3}, for all $t>0$ and all $x\in \bbr^d$
\begin{align*}
\dfrac{P_t^{\nu_\alpha}(f)(x)-f(x)}{t} = \frac{1}{t}  \int_0^t  \mathcal{L}^\alpha\left(P^{\nu_\alpha}_s(f)\right)(x) ds.
\end{align*} 
Once again, let us prove the commutation property: for all $f \in \mathcal{S}(\bbr^d)$, all $t>0$ and all $x\in \bbr^d$, 
\begin{align}\label{eq:commutation_irr_case}
\mathcal{L}^\alpha\left(P^{\nu_\alpha}_t(f)\right)(x) = P^{\nu_\alpha}_t\left(\mathcal{L}^\alpha(f)\right)(x).
\end{align}
Recall that the Fourier symbol of the non-local part of the operator $\mathcal{L}^\alpha$ (up to the prefactor $\alpha$) is given, for all $\xi \in \bbr^d$, by
\begin{align*}
\tau_\alpha(\xi) = \int_{\bbr^d}\left(e^{i\langle u ; \xi \rangle} - 1\right)\nu_\alpha(du). 
\end{align*}
Making us of scale invariance, for all $c>0$ and all $\xi \in \bbr^d$, $\tau_\alpha(c \xi) = c^\alpha\tau_\alpha(\xi)$. Starting with the drift term, for all $f\in \mathcal{S}(\bbr^d)$, all $t\geq 0$ and all $x\in \bbr^d$,
\begin{align*}
\langle x ; \nabla(P^{\nu_\alpha}_t(f))(x) \rangle = e^{-t} \langle x ; P^{\nu_\alpha}_t(\nabla(f))(x) \rangle & = \frac{e^{-t}}{(2\pi)^d} \sum_{j=1}^d \int_{\bbr^d} x_j \mathcal{F}(\partial_j(f))(\xi) e^{i\langle \xi ; xe^{-t}\rangle}\exp\left((1-e^{-\alpha t})\tau_\alpha(\xi)\right) d\xi. 
\end{align*}
Using spherical coordinates and scale invariance again,
\begin{align*}
\langle x ; \nabla(P^{\nu_\alpha}_t(f))(x) \rangle = \frac{i}{(2\pi)^d} \int_{(0,+\infty)\times \mathbb{S}^{d-1}} r e^{-t} \mathcal{F}(f)(ry) \langle x;y \rangle e^{i e^{-t} r \langle x;y \rangle} \exp\left((1-e^{- \alpha t})r^\alpha\tau_\alpha(y)\right) r^{d-1} dr \sigma_L(dy),
\end{align*}
where $\sigma_L$ is the spherical part of the $d$-dimensional Lebesgue measure on the sphere $\mathbb{S}^{d-1}$. Thus, 
\begin{align*}
\langle x ; \nabla(P^{\nu_\alpha}_t(f))(x) \rangle = \frac{1}{(2\pi)^d } \int_{\mathbb{S}^{d-1}} \left( \int_{(0,+\infty)} \mathcal{F}(f)(ry) \dfrac{d}{dr}\left(e^{i e^{-t} r\langle x;y \rangle}\right)\exp\left((1-e^{-\alpha t})r^\alpha \tau_\alpha(y)\right) r^d dr \right) \sigma_L(dy). 
\end{align*}
Integrating by parts in the radial coordinate gives 
\begin{align*}
\langle x ; \nabla(P^{\nu_\alpha}_t(f))(x) \rangle & = \frac{-1}{(2\pi)^d } \int_{\mathbb{S}^{d-1}} \left( \int_{0}^{+\infty} \dfrac{d}{dr}\left(\mathcal{F}(f)(ry)\right)e^{i e^{-t} r\langle x;y \rangle} \exp\left((1-e^{-\alpha t})r^\alpha \tau_\alpha(y)\right) r^d dr \right) \sigma_L(dy)\\
&+\frac{-1}{(2\pi)^d } \int_{\mathbb{S}^{d-1}} \left( \int_{0}^{+\infty} \dfrac{d}{dr}\left(\exp\left((1-e^{-\alpha t})r^\alpha \tau_\alpha(y)\right)\right)\mathcal{F}(f)(ry)e^{i e^{-t} r\langle x;y \rangle} r^d dr \right) \sigma_L(dy) \\
& + \frac{-1}{(2\pi)^d } \int_{\mathbb{S}^{d-1}} \left( \int_{0}^{+\infty}\mathcal{F}(f)(ry)e^{i e^{-t} r\langle x;y \rangle}\exp\left((1-e^{-\alpha t})r^\alpha \tau_\alpha(y)\right) d r^{d-1} dr \right) \sigma_L(dy).
\end{align*}
Differentiating and using classical coordinates,
\begin{align*}
\langle x ; \nabla(P^{\nu_\alpha}_t(f))(x) \rangle & = \frac{i}{(2\pi)^d} \int_{\mathbb{S}^{d-1}}  \left( \int_{0}^{+\infty} \langle y ; \mathcal{F}(xf)(ry) \rangle e^{i e^{-t} r\langle x;y \rangle} \exp\left((1-e^{-\alpha t})r^\alpha \tau_\alpha(y)\right) r^d dr \right) \sigma_L(dy)\\
&+\frac{-\alpha(1-e^{-\alpha t})}{(2\pi)^d } \int_{\mathbb{S}^{d-1}} \bigg( \int_{0}^{+\infty} \tau_\alpha(y) r^{\alpha-1}\exp\left((1-e^{-\alpha t})r^\alpha \tau_\alpha(y)\right)\\
&\times\mathcal{F}(f)(ry)e^{i e^{-t} r\langle x;y \rangle} r^d dr \bigg) \sigma_L(dy) \\
&+ \frac{-d}{(2\pi)^d } \int_{\bbr^d} \mathcal{F}(f)(\xi) e^{i \langle x;e^{-t} \xi \rangle} \dfrac{\hat{\mu}_\alpha(\xi)}{\hat{\mu}_\alpha(e^{-t}\xi)} d\xi, \\
& = \frac{1}{(2\pi)^d } \int_{\bbr^d} \langle i\xi ;\mathcal{F}(xf)(\xi) \rangle e^{i e^{-t} \langle x; \xi \rangle} \dfrac{\hat{\mu}_\alpha(\xi)}{\hat{\mu}_\alpha(e^{-t}\xi)} d\xi - \frac{d}{(2\pi)^d} \int_{\bbr^d} \mathcal{F}(f)(\xi) e^{i \langle x;e^{-t} \xi \rangle} \dfrac{\hat{\mu}_\alpha(\xi)}{\hat{\mu}_\alpha(e^{-t}\xi)} d\xi \\
&- \dfrac{\alpha (1-e^{-\alpha t})}{(2\pi)^d } \int_{\bbr^d} \mathcal{F}(f)(\xi) \tau_\alpha(\xi) e^{i\langle x;\xi e^{-t}\rangle} \dfrac{\hat{\mu}_\alpha(\xi)}{\hat{\mu}_\alpha(e^{-t}\xi)} d\xi , \\
& = \frac{1}{(2\pi)^d} \int_{\bbr^d} \mathcal{F}(\langle x; \nabla(f)\rangle)(\xi) e^{i \langle x;e^{-t}\xi\rangle} \dfrac{\hat{\mu}_\alpha(\xi)}{\hat{\mu}_\alpha(e^{-t}\xi)}d\xi \\
&- \dfrac{\alpha(1-e^{-\alpha t})}{(2\pi)^d } \int_{\bbr^d} \mathcal{F}(f)(\xi) \tau_\alpha(\xi) e^{i\langle x;\xi e^{-t}\rangle} \dfrac{\hat{\mu}_\alpha(\xi)}{\hat{\mu}_\alpha(e^{-t}\xi)} d\xi \\
& = P^{\nu_\alpha}_t\left(\langle x;\nabla(f)\rangle\right)(x) - \alpha (1-e^{-\alpha t}) P_t^{\nu_\alpha}(\mathcal{A}_\alpha(f))(x) , 
\end{align*}
where $\mathcal{A}_\alpha$ is the non-local part of the operator $\mathcal{L}^\alpha$ (up to the prefactor $\alpha$).~Finally, by scale invariance (once again) and Fourier inversion, for all $f\in \mathcal{S}(\bbr^d)$, all $x\in \bbr^d$ and all $t\geq 0$,
\begin{align*}
\mathcal{A}_\alpha \left(P^{\nu_\alpha}_t(f)\right)(x) & = \frac{1}{(2\pi)^d} \int_{\bbr^d} \mathcal{F}(f)(\xi) e^{i \langle x;\xi e^{-t} \rangle} \dfrac{\hat{\mu}_\alpha(\xi)}{\hat{\mu}_\alpha(e^{-t}\xi)} \tau_\alpha(e^{-t} \xi) d\xi, \\
& = \frac{e^{-\alpha t}}{(2\pi)^d} \int_{\bbr^d} \mathcal{F}(f)(\xi) e^{i \langle x;\xi e^{-t} \rangle} \dfrac{\hat{\mu}_\alpha(\xi)}{\hat{\mu}_\alpha(e^{-t}\xi)} \tau_\alpha(\xi) d\xi , \\
& = e^{-\alpha t} P^{\nu_\alpha}_t(\mathcal{A}_\alpha(f))(x)
\end{align*}
The commutation formula \eqref{eq:commutation_irr_case} follows finishing Step $2$.\\
\\
\textit{Step 3}: Finally, let us treat in details the very particular case $\alpha = 1$.~The Fourier symbol of the non-local part of the operator $\mathcal{L}^1$ is given, for all $\xi \in \bbr^d$, by
\begin{align*}
\tau_1(\xi) = \int_{\bbr^d} \left(e^{i\langle u ; \xi \rangle} - 1 - i\langle u ; \xi \rangle \bbone_{\|u\|\leq 1}\right)\nu_1(du). 
\end{align*}
Then, for all $c>0$,
\begin{align*}
\tau_1(c \xi) = c \int_{\mathcal{B}(0,c)} \left(e^{i\langle  u ; \xi \rangle} - 1 - i\langle u ; \xi \rangle\right) \nu_1(du) + c \int_{\mathcal{B}(0,c)^c} \left(e^{i\langle  u ; \xi \rangle} - 1\right) \nu_1(du).
\end{align*}
Now, assuming that $c<1$,
\begin{align*}
\tau_1(c\xi) = c\tau_1(\xi) + c \int_{c \leq \|u\| \leq 1} i \langle u ;\xi \rangle \nu_1(du).
\end{align*}
Similarly, for $c>1$, 
\begin{align*}
\tau_1(c\xi) & = c \tau_1(\xi) - c \int_{1 \leq \|u\| \leq c} i \langle u ;\xi \rangle \nu_1(du).
\end{align*}
Since $\sigma$ is symmetric, it is clear that $\int_{\mathbb{S}^{d-1}} y \sigma(dy) = 0$.~Thus, for all $\xi\in \bbr^d$ and all $c>0$, $\tau_1(c\xi) = c\tau_1(\xi)$. In what follows denote by $\mathcal{A}_1$ the non-local part of the operator $\mathcal{L}^1$ given, for all $f \in \mathcal{S}(\bbr^d)$ and all $x \in \bbr^d$, by
\begin{align*}
\mathcal{A}_1(f)(x) = \int_{\bbr^d} (f(x+u)-f(x) - \langle \nabla(f)(x) ; u \rangle \bbone_{\|u\|\leq 1} ) \nu_1(du) = \frac{1}{(2\pi)^d} \int_{\bbr^d} \mathcal{F}(f)(\xi)e^{i\langle x ; \xi \rangle} \tau_1(\xi) d\xi.
\end{align*}
Recall also that the action of the operator $P^{\nu_1}_t$ is given, for all $x\in \bbr^d$ and all $t\geq 0$, by
\begin{align*}
P^{\nu_1}_t(f)(x) = \frac{1}{(2\pi)^d} \int_{\bbr^d} \mathcal{F}(f)(\xi) e^{i \langle x; \xi e^{-t} \rangle} \dfrac{\hat{\mu}_1(\xi)}{\hat{\mu}_1(e^{-t}\xi)} d\xi = \frac{1}{(2\pi)^d} \int_{\bbr^d} \mathcal{F}(f)(\xi) e^{i\langle x ; \xi e^{-t}\rangle} \exp\left((1-e^{-t})\tau_1(\xi)\right) d\xi. 
\end{align*}
Then, 
\begin{align*}
\mathcal{A}_1\left(P^{\nu_1}_t(f)\right)(x) & = \int_{\bbr^d} \mathcal{F}(f)(\xi) e^{i\langle x; \xi e^{-t}\rangle} \dfrac{\hat{\mu}_1(\xi)}{\hat{\mu}_1(e^{-t}\xi)} \tau_1(e^{-t}\xi) \frac{d\xi}{(2\pi)^d} = e^{-t} P_t^{\nu_1}\left(\mathcal{A}_1(f)\right)(x). 
\end{align*}
A reasoning completely analogous to the one in Step $2$ of the proof ensures that, for all $f\in \mathcal{S}(\bbr^d)$, all $x \in \bbr^d$ and all $t\geq 0$,
\begin{align*}
\mathcal{L}^1\left(P^{\nu_1}_t(f)\right)(x) = P^{\nu_1}_t(\mathcal{L}^1(f))(x). 
\end{align*} 
The end of the proof follows similarly. 
\end{proof}
\noindent
At last, one can see that, 
via \eqref{def:stable}, 
\begin{align*}
\langle -\mathcal{L}^{\alpha}(f); f \rangle_{L^2(\mu_\alpha)} = \frac{\alpha}{2}\int_{\bbr^d} \int_{\bbr^d}(f(x+u)-f(x))^2 \nu_\alpha(du)\mu_\alpha(dx), 
\end{align*}
for all $f \in \mathcal{S}(\bbr^d)$.  
Thus, observe that the semigroup $(\mathcal{P}_t)_{t\geq 0}$ is Markovian since the normal contractions 
operate on the closed form associated with it.  
Next, the measure $\mu_\alpha$ is invariant for $(\mathcal{P}_t)_{t\geq 0}$ in that, 
for all $f\in L^2(\mu_\alpha)$,
\begin{align*}
\int_{\bbr^d} \mathcal{P}_t(f)(x) \mu_\alpha(dx)= \int_{\bbr^d} f(x) \mu_\alpha(dx), 
\end{align*}
and, by duality arguments, $\mathcal{P}_t(1)=1$, for all $t\geq 0$.  
These facts allow to extend $\mathcal{P}_t$, $t\geq 0$, to a contraction on 
every $L^p(\mu_\alpha)$, for all $p\in [1,+\infty]$ by appealing to \cite[Theorem $X.55$]{RS2}. 
These extensions form $C_0$-semigroups on $L^p(\mu_\alpha)$, for all $p\in [1,+\infty)$.  
Moreover, and finally, the following representations hold true on every 
$L^p(\mu_\alpha)$, $p\in (1,+\infty)$:

\begin{prop}\label{prop:squared_Mehler_Lp}
Let $\alpha \in (0,2)$ and let $\nu_\alpha$ be a non-degenerate symmetric L\'evy measure on $\bbr^d$ satisfying \eqref{eq:scale}.~Let $\mu_\alpha$ be the non-degenerate $\alpha$-stable probability measure associated with $\nu_\alpha$ and defined by \eqref{def:stable}.  
Let $p \in (1,+\infty)$, let $(P^{\nu_\alpha}_t)_{t\geq 0}$ be the $L^p(\mu_\alpha)$-extension of the 
semigroup of operators defined by \eqref{eq:StOUSM} and let $((P^{\nu_\alpha}_t)^*)_{t\geq 0}$ be the dual semigroup of $(P^{\nu_\alpha}_t)_{t\geq 0}$ from $L^{q}(\mu_\alpha)$ to $L^{q}(\mu_\alpha)$ 
with $q=p/(p-1)$.  
Then, for all $s,t\geq 0$, $P^{\nu_\alpha}_t$ and $(P^{\nu_\alpha}_s)^*$ commute and 
the family of operators $(\mathcal{P}_t)_{t\geq 0}$ defined, for all $t\geq 0$, by
\begin{align}\label{eq:SqaredMehlerSG}
\mathcal{P}_t=P^{\nu_\alpha}_{t/\alpha} \circ (P^{\nu_\alpha}_{t/\alpha})^*,
\end{align}
is a $C_0$-semigroup of contractions on $L^p(\mu_\alpha)$.  
Moreover, if $(\mathcal{L}^\alpha, \mathcal{D}(\mathcal{L}^\alpha))$ and $((\mathcal{L}^\alpha)^*, \mathcal{D}((\mathcal{L}^\alpha)^*))$ are the generators of the $C_0$-semigroups of 
contractions $(P^{\nu_\alpha}_t)_{t\geq 0}$ and $((P^{\nu_\alpha}_t)^*)_{t\geq 0}$ respectively, then $\mathcal{D}(\mathcal{L}^\alpha) \cap \mathcal{D}((\mathcal{L}^\alpha)^*)$ is invariant with respect to $\mathcal{P}_t$, for $t\geq 0$, is dense in $L^p(\mu_\alpha)$ and, therefore, is a core for $\mathcal{L}$, 
the closure of the sum $(\mathcal{L}^\alpha + (\mathcal{L}^\alpha)^*)/\alpha$. Finally, for all $f \in \mathcal{D}(\mathcal{L}^\alpha)\cap \mathcal{D}((\mathcal{L}^\alpha)^*)$
\begin{align*}
\mathcal{L}(f)=\frac{1}{\alpha}\left(\mathcal{L}^\alpha(f)+(\mathcal{L}^\alpha)^*(f)\right).
\end{align*}
\end{prop}
 
\begin{proof}
Let $p \in (1,+\infty)$ and let $s,t \geq 0$. Since $\mathcal{S}(\bbr^d)$ is dense in $L^p(\mu_\alpha)$ and since $P^{\nu_{\alpha}}_t$ and $(P^{\nu_\alpha}_s)^*$ are bounded operators on $L^p(\mu_\alpha)$ (with norms less or equal to $1$), it is sufficient to prove, for all $f\in \mathcal{S}(\bbr^d)$, that
\begin{align*}
P^{\nu_{\alpha}}_t \circ (P^{\nu_\alpha}_s)^* (f) = (P^{\nu_\alpha}_s)^*  \circ P^{\nu_{\alpha}}_t (f).
\end{align*}
Let $g \in \mathcal{S}(\bbr^d)$.  Since the different extensions are compatible and by 
Theorem~\ref{def:squared_Mehler_L2}, 
\begin{align*}
\langle P^{\nu_{\alpha}}_t \circ (P^{\nu_\alpha}_s)^* (f) ; g\rangle_{L^p(\mu_\alpha),L^{q}(\mu_\alpha)}&= \langle (P^{\nu_\alpha}_s)^* (f) ; (P^{\nu_{\alpha}}_t)^*(g)\rangle_{L^p(\mu_\alpha),L^{q}(\mu_\alpha)},\\
& = \langle (P^{\nu_\alpha}_s)^*(f) ; (P^{\nu_{\alpha}}_t)^*(g)\rangle_{L^2(\mu_\alpha),L^{2}(\mu_\alpha)},\\
& = \langle P^{\nu_{\alpha}}_t(f) ; P^{\nu_\alpha}_s(g)\rangle_{L^2(\mu_\alpha),L^{2}(\mu_\alpha)},\\
& = \langle P^{\nu_{\alpha}}_t(f) ; P^{\nu_\alpha}_s(g)\rangle_{L^p(\mu_\alpha),L^{q}(\mu_\alpha)},\\
& =  \langle (P^{\nu_\alpha}_s)^*\circ P^{\nu_{\alpha}}_t(f) ; g\rangle_{L^p(\mu_\alpha),L^{q}(\mu_\alpha)},
\end{align*}
which implies the commutativity of $ P^{\nu_{\alpha}}_t $ and $(P^{\nu_\alpha}_s)^*$, for all $s,t\geq 0$. 
It then follows that $(\mathcal{P}_t)_{t\geq 0}$ is a $C_0$-semigroup of contractions on $L^p(\mu_\alpha)$ which generator, denoted by $\mathcal{L}$, is the closure of 
the sum $(\mathcal{L}^{\alpha} + (\mathcal{L}^{\alpha})^*)/\alpha$. Finally, \cite[$2.7$ Product Semigroups]{EN00} finishes the proof of the proposition. 
\end{proof}

\end{document}